\documentclass[11pt,a4paper]{amsart}
\usepackage[foot]{amsaddr}
\usepackage{a4wide}
\usepackage[english]{babel}
\usepackage{cite}
\usepackage{amsmath}
\usepackage{amssymb}
\usepackage{graphicx}
\usepackage{amsthm}
\usepackage{kvoptions}
\usepackage{thmtools}
\usepackage[hidelinks]{hyperref}
\usepackage{enumerate}
\usepackage{enumitem}
\usepackage{bbm}
\usepackage{bm}
\usepackage{mathrsfs}

\usepackage{comment}

\numberwithin{equation}{section}

\newcommand\rate[1]{\limsup_{N\to\infty} \frac{1}{N} \log \E \left[\, #1 \, \right]}
\newcommand\rated[1]{\liminf_{N\to\infty} \frac{1}{N} \log \E \left[\, #1 \, \right]}
\newcommand\ratep[1]{\limsup_{N\to\infty} \frac{1}{N} \log P^N \big(\, #1 \, \big)}

\newcommand{\law}{\mathrm{Law}}

\newtheorem{thm}{Theorem}[section]
\newtheorem{prop}[thm]{Proposition}
\newtheorem{lm}[thm]{Lemma}
\newtheorem{cory}[thm]{Corollary}

\theoremstyle{definition}
\newtheorem{defi}[thm]{Definition}

\newtheorem{remark}[thm]{Remark}

\newcommand{\N}{\mathbb{N}}
\newcommand{\R}{\mathbb{R}}
\newcommand{\REx}{\bar{\R}}
\newcommand{\E}{\mathbb{E}}
\newcommand{\Lip}{\text{Lip}}

\newcommand{\FLip}{\mathcal{F}\mathrm{Lip}}

\newcommand{\cX}{\mathcal{X}}
\newcommand{\cA}{\mathcal{A}}
\newcommand{\cB}{\mathcal{B}}
\newcommand{\cE}{\mathcal{E}}
\newcommand{\cF}{\mathcal{F}}
\newcommand{\cG}{\mathcal{G}}
\newcommand{\cH}{\mathcal{H}}
\newcommand{\cP}{\mathcal{P}}

\newcommand{\x}{{\boldsymbol{x}}}
\renewcommand{\o}{{\boldsymbol{\omega}}}

\usepackage{xcolor}

\title{Large deviations for singularly interacting diffusions}
\author{Jasper Hoeksema\textsuperscript{*}}
\address{\textsuperscript{*} Department of Mathematics and Computer Science, Eindhoven University of Technology, 5600 MB Eindhoven,The Netherlands; email addresses: j.hoeksema@tue.nl, o.t.c.tse@tue.nl}
\author{Thomas Holding}
\author{Mario Maurelli\textsuperscript{\textdagger}}
\address{\textsuperscript{\textdagger} Dipartimento di Matematica `Federigo Enriques', Universit\`a degli Studi di Milano, via Saldini 50, 20133 Milano, Italy; email address: mario.maurelli@unimi.it}
\author{Oliver Tse\textsuperscript{*}}

\date{\today}

\begin{document}

\begin{abstract}
    In this paper we prove a large deviation principle (LDP) for the empirical measure of a general system of mean-field interacting diffusions with singular drift (as the number of particles tends to infinity) and show convergence to the associated McKean--Vlasov equation. Along the way, we prove an extended version of the Varadhan Integral Lemma for a discontinuous change of measure and subsequently an LDP for Gibbs and Gibbs-like measures with singular potentials.
\end{abstract}

\maketitle

\vspace{-2em}

\tableofcontents

\newpage

\section{Introduction}

In this work we study the limiting behaviour of weakly interacting, or mean field, diffusions, where the interaction depends only on the empirical measures of the particles. For every $N\in\N$ the particle system is defined by the coupled stochastic differential equations (SDEs)
\begin{align}\label{eq:intro_SDE}
\left\{\quad\begin{aligned}
&d X^{N,i}_t = b_t\left(X_t^{N,i},\frac{1}{N}\sum_{i=1}^N \delta_{X_t^{N,i}}\right) d t + d W^i_t,\\
&X^{N,i}_0  \text{ i.i.d~with law } \rho_0.
\end{aligned}\right.
\end{align}
Here $W^1,\ldots,W^N$ are independent $d$-dimensional Brownian motions, $\rho_0$ is a given initial distribution, and $b$ is a measure-dependent \emph{drift} vector. A common example for $b$ is of the form
\begin{equation}\label{eq:intro_drift}
b_t(x,\mu) = \int_{\R^d} \varphi(t, x-y)\,d\mu(y),
\end{equation}
for some interaction kernel $\varphi$.  Such type of drifts commonly appear in models of classical physical systems, biological systems such as the collective motion of micro-organisms (bacteria, cells, etc.), and flocking and swarming behavior of animals, granular media, as well as models in opinion formation.

\medskip

When $b$ is sufficiently regular, the limiting behaviour of the particle system for a large particle number is well understood. For example, when $b$ is Lipschitz and bounded, the empirical measure
\[
z^N_{\boldsymbol{X}} := \frac{1}{N}\sum_{i=1}^N \delta_{X^{N,i}},
\]
converges to the law of the McKean-Vlasov equation \cite{sznitman1991,Tanaka1984}
\begin{equation}\label{eq:intro_mckean}
\left\{\quad\begin{aligned}
&d X_t = b_t(X_t,\law(X_t))\, d t+d W_t,\\
&X_0  \text{ with law } \rho_0.
\end{aligned}\right.
\end{equation}
Moreover a large deviation principle (LDP) holds for the empirical measure $z^N_{\boldsymbol{X}}$, as $N\rightarrow \infty$, see e.g. \cite{Dawsont1987,budhiraja2012,CDFM2018}.

The case of a \textit{singular} interaction, that is irregular $b$, has been widely studied too. The convergence of the system \eqref{eq:intro_SDE} to the corresponding McKean-Vlasov equation has been shown for various examples of singular drifts, most of them of the form \eqref{eq:intro_drift} with singular interaction kernel $\varphi$, e.g. \cite{Fournier2014,Godinho2015,Berman2016,Jabin2018,Bresch2019,Jab2019} (see Section~\ref{sec:uniqueness} for detailed explanations).

\medskip

However, establishing LDPs for these singular drifts has remained unsolved. Apart from the work by Fontbona in \cite{Fontbona2004}, where an LDP for the time-marginals of $(z^N_{\boldsymbol{X}})$ was shown for a repulsive kernel $\varphi(x)=1/x$, little is known to our knowledge. We aim to fill this gap, by providing LDP results and new tools for a large class of singular measure-dependent drifts.

As the main example, we consider the following drift
\begin{equation}\label{eq:intro_drift2}
b_t(x,\mu) := \psi\left(x,\mu,\int_{\R^d} \varphi(t, x-y)\,d \mu(y)\right),
\end{equation}
where $\psi:\R^d \times \cP(\R^d)\times  \R^d\to \R^d$, with $\cP(\R^d)$ the space of probability measures equipped with the bounded Lipschitz metric. We will show an LDP when $\psi$ is Lipschitz and the interaction kernel is in an appropriate $L^p$ space. More precisely, combining several key statements from Section \ref{s_process} (see Proposition \ref{thm_pmark1}, Remark \ref{rem:intro_sde} and Proposition \ref{cor_poc_ex}), we have the following theorem:

\begin{thm}\label{thm:intro_sde}
Suppose that 
\begin{enumerate}[label=(\roman*)]
    \item $\psi:\R^d \times \cP(\R^d)\times  \R^d\to \R^d$ is jointly globally Lipschitz
    \vspace{1pt}
    \item there exists a constant $L$ such that 
    	\begin{equation*}
    \psi \left(x,\mu, z \right) \leq L(1+|z|), \qquad \forall x,z\in \R^d, \mu\in \cP(\R^d).
    \end{equation*}
\item for all $\beta>0$
\[	    \int_{\R^d} e^{\beta |x|}\, d\rho_0 (x)<\infty,	\]
\item  for $p,q\in[2,\infty]$ with $d/p + 2/q <1$,
\[\varphi\in L^q\big((0,T),L^p(\R^d)\big) +L^\infty((0,T)\times\R^d).\]
\end{enumerate}
\vspace{2pt}
Then the family $\{Q^N\}$ of laws of empirical measures $z^N_{\boldsymbol{X}}$ for $\boldsymbol{X}=(X^{N,1},\ldots,X^{N,N})$ satisfying \eqref{eq:intro_SDE}, with drift $b$ as in \eqref{eq:intro_drift2},  has an LDP with rate function
	\[
	\cF(\mu)=
	\begin{cases}
	R(\mu\|\mathbb{W}^{\mu}) & \mbox{if $R(\mu\|\mathbb{W})<\infty$}, \\
	+\infty & \mbox{otherwise},
	\end{cases}
	\]
	where $\mathbb{W}^{\mu}$ is the law of a process $X^{\mu}_t$ satisfying the SDE
	\begin{equation*}
	d X^{\mu}_t = b(X_t^\mu,\mu_t)\, d t+d W_t,
	\end{equation*}
	and $\mathbb{W}=\mathrm{Law}(W)$, where $W$ is a Brownian motion with initial law $\rho_0$.
	
	Furthermore, $z^N_{\boldsymbol{X}}$ converges almost surely to the unique minimizer of $\cF(\mu)$, which is the unique law of the solution to the McKean-Vlasov SDE \eqref{eq:intro_mckean}. 
\end{thm}

An application of Theorem \ref{thm:intro_sde} is the case of a drift $b$ of the form \eqref{eq:intro_drift}, with
\begin{align*}
&\varphi(t,z)=  |z|^{\alpha} g\left(\frac{z}{|z|}\right)1_{|z|\le R} +h(z)1_{|z|>R},
\end{align*}
with $g:\mathbb{S}^{d-1}\rightarrow \mathbb{S}^{d-1}$ and $h:\R^d\rightarrow\R^d$ both Borel bounded, $R>0$, and exponent $\alpha$ satisfying
\begin{align*}
    \alpha>-1 \;\text{ for }\; d\ge 2,\quad \text{ and }\quad
\alpha>-1/2 \;\text{ for }\; d=1.
\end{align*}

In fact, we prove LDP and convergence to the McKean-Vlasov equation for systems with singular drifts under more general assumptions, see  Theorem \ref{thm_pmainlip} and the examples in Sections \ref{s_examples}, \ref{sec:uniqueness}, where we include many-particle interaction (that is, dependence of $b$ on $\mu^{\otimes k}$) and interaction kernels $\varphi$ merely satisfying
\begin{equation*}
    \E \left[ e^{\beta \int_0^T |\varphi|^2\left(t,W^1_t,W^2_t\right)\, dt} \right] < \infty, \qquad \forall \beta\in \R,
\end{equation*}
where $W^1,W^2$ are independent Brownian motions with common initial law $\rho_0$.

Note that even the convergence result to the McKean-Vlasov SDE in Theorem \ref{thm:intro_sde} is new: while some works do show convergence for the class of drifts \eqref{eq:intro_drift} with even more singular $\varphi$ \cite{Fournier2014,Jabin2018,Bresch2019}, we are not aware of a result that covers drifts of the form \eqref{eq:intro_drift2} under our assumptions.
\medskip

Our proof of the LDP relies on using a singular change of measure via Girsanov's theorem and an approximation by regular drifts. To deal with this, we extend a classical tool in large deviation theory, Varadhan's Integral Lemma. We believe that this extension and other underlying results are relevant on their own and they might also prove helpful in future works on singular drifts. Let us briefly our strategy.

In a nutshell, Varadhan's Integral Lemma (or Varadhan's Lemma for short) allows one to transfer the LDP through a continuous change of measure and is in fact a natural extension of Laplace's method to infinite dimensional spaces. To be precise, let $\{z^N\}$ be a family of random variables on a probability space $(\Omega,\cA,\mathbb{P})$, taking values in a Polish space $\cX$ endowed with its Borel $\sigma$-algebra $\cB(\cX)$. The classical Varadhan's Lemma (cf.\ \cite{Varadhan1984,Moral2003,Dembo10}) reads as follows. 
\begin{prop}[Varadhan's Integral Lemma]\label{thm_vara}
	Suppose $P^N$ satisfies an LDP with good rate function $I:\cX\to[0,\infty]$, and let $\cE:\cX\to\R$ be any continuous and bounded function. Then the family of measures $\{Q^N\}$ defined by
	\begin{align}\label{eq:Q_induced-1}
	\frac{d Q^N}{dP^N}(\mu):=\frac{1}{Z_N}e^{-N\cE(\mu)}\qquad \text{for\, $P^N$-almost every $\mu\in\cX$},
	\end{align}
	with normalizing constants $Z_N$, satisfies an LDP with good rate function 
	\begin{equation}\label{eq:Q_induced-1b}
	F(\mu):=(I+\cE)(\mu)-\inf_{\nu \in \cX} (I+\cE)(\nu).
	\end{equation}
\end{prop}
Subsequently, various extensions have been developed in the past decades to relax the assumption of continuity; for example to deal with singular functionals or contractions for Gibbs measures on $\R^d$ \cite{Bodineau1999,Chafai2014,Dupuis2020,Hardin2016}, or on abstract spaces \cite{Leonard1987,Eichelsbacher1998,Eichelsbacher2002,Moral2003,Berman2016a,Zelada2017,Liu2018}.
We refer to the background paragraphs in Sections \ref{LDPs} and \ref{sgibbs} for a more detailed discussion, and only highlight a few points here.

\medskip

A common thread in some of the extensions above are various approximation arguments. For example, as outlined in \cite{Dembo10} on \emph{exponential approximations}, the family $\{Q^N\}$ satisfies an LDP if there exists another family $\{Q^N_\lambda\}$ which satisfies an LDP for each $\lambda>0$ and approximate $Q^N$ in some exponentially good way (as $\lambda\rightarrow0$). Liu and Wu \cite{Liu2018} make use of techniques involving exponential approximations and prove LDPs for Gibbs measures with singular potential.
However, in our setting, we cannot rely on their result since, for general drift as in \eqref{eq:intro_drift2}, the associated $\cE$ is not actually in the form of a Gibbs energy.
Hence, we have developed the following extension (see Theorem \ref{thm_direct} for the precise statement): 

\begin{thm}[Extended Varadhan Integral Lemma]\label{thm_direct2}
	Let $P^N=\law(z^N)$ be a family satisfying an LDP with rate function $I$,  and $\cE,\cE^N:\cX\to[-\infty,\infty]$ measurable functions. Moreover, define the family of measures $\{Q^N\}$ by \eqref{eq:Q_induced-1} and $\cF:\cX\to[0,\infty]$ by \eqref{eq:Q_induced-1b}.
	
	Now suppose that for all $\lambda>0$ there exists $\cE_{\lambda}$ and a family $\{\cE_{\lambda}^N\}_N$ such that the family $\{Q_{\lambda}^N\}_N$ satisfies an LDP with a rate function $\cF_{\lambda}$ (with $Q_{\lambda}^N, \cF_{\lambda}$ defined similarly as in \eqref{eq:Q_induced-1},\eqref{eq:Q_induced-1b}).  
	Moreover, suppose that for some $\gamma>1$, 
	\begin{equation*}
		\begin{aligned}
		\rate{e^{-\gamma N\cE_{\lambda}^N(z^N)}} <+\infty,\\
		\inf_{\mu\in\cX} (I+\gamma\cE_\lambda)(\mu) >-\infty,
		\end{aligned}
	\end{equation*}
	for every $\lambda>0$, and such that for every $\beta \in \R$,
	\begin{equation*}
		\begin{aligned}
		\limsup_{\lambda\to 0}  \rate{e^{\beta N(\cE^N-\cE_{\lambda}^N)(z^N)}} &=0,\\
		\limsup_{\lambda\to 0}  \sup_{\mu \in \cX} \, \bigl(\beta (\cE_{\lambda}-\cE) - I\bigr)(\mu) &=0.
		\end{aligned}
	\end{equation*}
	Then the family $\{Q^N\}$ satisfies an LDP with rate function $\cF$. 
\end{thm}

We provide a self-contained proof which only relies on basic large deviation theory and elementary convexity estimates. The two main points of this extension are that we do not require $\mathcal{E}$ to be continuous and that we also do not require the approximating energies $\cE_{\lambda}^N$ to be continuous, but merely to induce an LDP in the sense described above. The latter point is an essential tool in proving Theorem \ref{thm:intro_sde}. Namely, even for regular drifts, the change of measure provided via Girsanov's Theorem is not necessarily continuous. 

Other extensions and applications of Varadhan's Lemma have also been developed to deal with this, for example to prove LDPs for weakly interacting diffusions with regular drifts [PdH96, DZ03, DFMS17]. Moreover, in [DFMS17], one of the authors of this manuscript developed an enhanced version of Sanov’s theorem in the rough path setting, which allows for Varadhan’s Integral Lemma to be applied. However, to our knowledge, none of them have been used to prove LDPs for weakly interacting diffusions with \emph{singular} interaction.

\medskip

It should be noted that one cannot expect to establish an LDP via a change of measure for every singular drift. This is indeed the case when the law of the interacting particle system is not absolutely continuous with respect to the law of the non-interacting system. A particular example in which this occurs is the Keller-Segel model with $\varphi(t,z)=-\nabla \log{z}$, for which a notion of propagation of chaos was established in \cite{Bresch2019} but an LDP result remains open.

However, even in the large class of drifts for which the system can be described via a change of measure there is still a gap between those for which there is a known LDP and propagation of chaos, and those for which others have merely shown propagation of chaos. We believe that not only are our results such as Theorem \ref{thm:intro_sde} a sizable step in closing this gap, but that the general tools we provide will help close it even further.

\medskip

\paragraph{\bf\em Organization and highlights of the paper.}  In Section \ref{LDPs} we recall basic definitions and results in large deviation theory, and provide an extension of Varadhan's Lemma in Theorem~\ref{thm_direct}. Next, in Section~\ref{sgibbs}, we use this to prove LDPs for empirical measures of mean-field Gibbs systems, where the log-densities $\cE_V^N:\cP(S)\to [-\infty,\infty]$ are parameterized by a family of Borel functions $V^N:S^k\to[-\infty,\infty]$, $k\in\N$ (see \eqref{eqg_en} for a precise definition). In Theorem~\ref{thm_gibbs1}, we provide sufficient conditions, in terms of suitable approximations of $\{V^N,V\}$, under which $\{\cE_V^N,\cE_V\}$ induces an LDP---this result is in the same spirit as \cite{Liu2018}. We further prove a result (Theorem \ref{thm_gibbs2}) that generalizes Theorem~\ref{thm_gibbs1} to include `Gibbs-like' measures---which are not of Gibbs form but such that the error $\cE^N_{V}-\cE^N_{V_{\lambda}}$ can be essentially bounded by Gibbs measures---that simplifies our task in proving LDPs for weakly interacting diffusions in Section~\ref{s_process}.

In the latter, with the results of Sections~\ref{LDPs} and \ref{sgibbs} at hand, establishing an LDP for a system of weakly interacting diffusions amounts to (1) proving a (change-of-measure) representation formula (Girsanov's formula) for the laws $\{Q^N\}$ of the empirical measures $z^N_{\boldsymbol{X}}$ associated to the solution $\boldsymbol{X}=(X^{N,1},\ldots,X^{N,N})$ of \eqref{eq:intro_SDE}; and (2) proving the existence of a family $\{b_\lambda^N\}$ of ``exponentially good'' approximations for $b$ (cf.\ Theorem~\ref{thm_pmainlip} and the concrete examples in Section~\ref{s_examples}), which implies Theorem \eqref{eq:intro_SDE} considered above with a drift $b$ specified by \eqref{eq:intro_drift}. 
For drifts $b$ satisfying the assumption of Theorem~\ref{thm:intro_sde}, and for all the examples in Subsections \ref{s_examples}, we further show that the rate function $\cF$ associated to $\{Q^N\}$ attains a unique minimizer (see Subsection~\ref{sec:uniqueness} for the general case), which implies almost sure convergence. 

Finally, for the sake of completeness and consistency, we have included a relatively large appendix containing technical results and proofs utilized throughout the article, which we believe to be either new, or helpful to the reader.

\medskip

\paragraph{\bf\em Acknowledgements.} 
The authors thank M.A.~Peletier for proposing the problem and for fruitful discussions at the start of the work. J.H.\ and O.T.\ acknowledges support from NWO Vidi grant 016.Vidi.189.102, ``Dynamical-Variational Transport Costs and Application to Variational Evolutions". Parts of this work were undertaken when M.M.\ was at Weierstrass Institute for Applied Analysis and Stochastics, Germany, at Technische Universit\"at Berlin, Germany, and at University of York, UK. M.M.\ acknowledges support from the Royal Society via the Newton International Fellowship NF170448 ``Stochastic Euler equations and the Kraichnan model'', from project PRIN 2015233N54\textunderscore002 ``Deterministic and stochastic evolution equations'' from the Italian Ministry of Education, University and Research, and from the Hausdorff Research Institute for Mathematics in Bonn under the Junior Trimester Program ``Randomness, PDEs and Nonlinear Fluctuations''.

\section{An Extension of Varadhan Integral Lemma}\label{LDPs}

In this chapter we extend the classical Varadhan's Integral Lemma (cf.~\cite{Dembo10,Varadhan1984}), to allow for establishing LDPs via a change of measure with possibly discontinuous density.

\subsection{Notations and preliminary results}

We introduce some notation that we will use throughout the manuscript. The space $\cX$ is a Polish space endowed with its Borel $\sigma$-algebra $\mathcal{B}(\cX)$. The symbol $\cP(\cX)$ denotes the set of all probability measures on $\cX$ and we use the letters $P$, $Q$, and $P^N$, $Q^N$, ... for probability measures on $\cX$. With a little abuse of notation, we use $\mu$ both for a generic element of $\cX$ and for the canonical random variable on $\cX$ ($\mu(x)=x$ for all $x$ in $\cX$); this notation is unusual, but it will be convenient in the next sections, where $\cX$ will be itself a space of probability measures. We often consider (without loss of generality) $P^N$ as the law of an $\cX$-valued random variable $z^N$, defined on a probability space $(\Omega,\mathcal{A},\mathbb{P})$ (independent of $N$); $\E$ denotes the expectation with respect to $\mathbb{P}$.\\

We recall a definition of a large deviation principle (LDP).
\begin{defi}\label{def_ldp}
A family of measures $\{Q^N\}\subset\cP(\cX)$ satisfies an LDP with {\em rate function} $\cF:\cX\to [0,\infty]$ if (1) $\cF$ is lower semi-continuous, if (2) for every Borel set $A$, 
\[
-\inf_{\mu \in A^o} \, \cF(\mu) \leq \liminf_{N\to \infty} \frac{1}{N} \log Q^N(A) \le \limsup_{N\to \infty} \frac{1}{N} \log Q^N(A)  \leq   -\inf_{\mu \in \bar{A}} \, \cF(\mu),
\]
and if (3) the family $\{Q^N\}$ is \emph{exponentially tight}, i.e.\ there is a sequence of compact sets $K_M\subset\cX$ such that
\[
\limsup_{M\to \infty}\limsup_{N\to \infty} \frac{1}{N}\log Q^N(\cX\setminus K_M) = -\infty.
\]
\end{defi}
We denote the domain of $\cF$ by $D(\cF):=\{\mu\in\cX \,|\,\cF(\mu)<\infty\}$.
\begin{remark}
As shown in \cite[p.~8,~120]{Dembo10}, Definition \ref{def_ldp} in Polish spaces is equivalent to stating that (2) holds with a \emph{good rate function} $\cF$, i.e.\ $\cF$ having compact sub-level sets. 
\end{remark}
\begin{remark}
    Let $\{z^N\}$ be a family of $\cX$-valued random variables such that $Q^N=\law(z^N)\in\cP(\cX)$ satisfies an LDP with rate function $\cF$. If the minimizer $\mu^*\in\cX$ of $\cF$ is unique, then the LDP implies the convergence $Q^N\to \delta_{\mu^*}$ weakly. In fact, by a standard argument we obtain a stronger result: almost sure convergence of the random variables $z^N$ to $\mu^*$, as stated below. For a proof, see for example \cite[Theorem A.2]{Schlottke2019}.
    \begin{lm}\label{lm:almost-sure}
Suppose $P^N$ satisfies an LDP with rate function $\cF$, and that $\cF$ has a unique minimizer $\mu^*$. Then $z^N$ converges $\mathbb{P}$-almost surely to $\mu^*$.
\end{lm}
\end{remark}

Now, let $P^N=\law(z^N)$ be probability measures on $\cX$ satisfying a large deviation principle with rate function $I:\cX\to [0,\infty]$. We consider pairs $(\{\cE^N\},\cE)$ (denoted $(\cE^N,\cE)$ for short) of a sequence of Borel functions $\cE^N:\cX\to\REx$ and a Borel function $\cE:\cX\to\REx$, and study whether an LDP may be established for the induced measures $Q^N$,
\begin{equation}\label{eq:Q_induced-2}
\frac{d Q^N}{d P^N}(\mu):=\frac{1}{{Z_N}}e^{- N \cE^N(\mu)}\qquad \text{for\, $P^N$-almost every $\mu\in\cX$},
\end{equation}
where the normalization constants $Z_N$ are assumed to be finite for all $N\in\N$. 

Precisely, we define $J$ and $\cF$ as follows: 
\[
J(\mu):= \left\{\begin{aligned}
&I(\mu) + \cE(\mu) \qquad &\mu\in D(I),\\
&+\infty \qquad &\mu \not \in D(I),
\end{aligned}\right.
\]
and, if $\inf_{\mu\in\cX} J(\mu)$ is finite, we define
\begin{equation*}
\cF(\mu):=J(\mu)-\inf_{\mu\in\cX} J(\mu).
\end{equation*}
Finally, note that by construction for any Borel set $A$
\[\inf_{\mu \in A\cap D(I)} (\cE+I)(\mu) = \inf_{\mu\in A} J(\mu).\]
Then the property we investigate is given in the following definition.
\begin{defi}\label{defi:ldp_pair}
We say that $(\cE^N,\cE)$ induces an LDP if $\inf_{\mu\in\cX} J(\mu)$ is finite, $\{Q^N\}$ (defined as in \eqref{eq:Q_induced-2}) satisfies an LDP with rate function $\cF$ and satisfies the so-called \emph{Laplace principle},
\begin{equation}\label{eq_ldplaplace}
\lim_{N\to \infty} \frac{1}{N}\log \E \left[ e^{- N \cE^N(z^N)}\right]=-\inf_{\mu\in\cX} J(\mu).
\end{equation}
\end{defi}
The following lemma provides a characterization in terms of an \textit{unnormalized} LDP.
\begin{lm}\label{lm_ldpchar}
The following statements are equivalent:
\begin{enumerate}[label=(\roman*)]
    \item The pair $(\cE^N,\cE)$ induces an LDP (according to Definition~\ref{defi:ldp_pair});
    \item $\inf_{\mu\in \cX} J(\mu)\in\R$, $J$ is lower semi-continuous, the family $\{Q^N\}$ is exponentially tight, and for every Borel set $A$,
    \begin{equation}\label{eq_ldpchar2}
\begin{aligned}
-\inf_{\mu \in A^o} \, J(\mu) &\leq \rated{e^{- N \cE^N(z^N)}\,1_A} \\
&\leq \rate{e^{- N \cE^N(z^N)}\,1_A} &\leq   -\inf_{\mu \in \bar{A}} \, J(\mu).
\end{aligned}
\end{equation}
\end{enumerate}
\end{lm}

\begin{proof}
Suppose $(\cE^N,\cE)$ induces an LDP. Exponential tightness of $Q^N$ follows from the definition of an LDP, and since $\cF$ is lower semi-continuous $J$ is as well. By the Laplace principle \eqref{eq_ldplaplace},
\[
\lim_{N\to \infty} \frac{1}{N}\log \E \left[ e^{- N \cE^N(z^N)}\right]=-\inf_{\mu\in\cX} J(\mu),
\]
where the right-hand side is assumed to be finite. For for any $A\in\cB(\cX)$, we have that
\begin{align*}
    \frac{1}{N} \log \E \left[\, e^{- N \cE^N(z^N)}\,1_A \, \right] &= \frac{1}{N} \log \E \left[\, \frac{1}{Z^N}e^{- N \cE^N(z^N)}\,1_A \, \right] + \frac{1}{N}\log Z^N \\
    &= \frac{1}{N} \log Q^N(A) + \frac{1}{N}\log Z^N.
\end{align*}
Therefore, by \eqref{eq_ldplaplace} and the LDP of $Q^N$, we then obtain
\begin{align*}
\rate{e^{- N \cE^N(z^N)}\,1_A} 
&=\limsup_{N\to \infty} \frac{1}{N} \log Q^N(A)-\inf_{\mu\in\cX} J(\mu)\\
&\leq - \inf_{\mu \in \bar{A}} \cF(\mu)-\inf_{\mu\in\cX} J(\mu)=-\inf_{\mu\in\bar{A}} J(\mu).
\end{align*}
The lower bound follows similarly, and hence \eqref{eq_ldpchar2} is satisfied. 

Conversely, assume that $J$ is lower semi-continuous, $\inf_{\mu\in\cX} J(\mu)$ is finite and that \eqref{eq_ldpchar2} holds. Then $\cF$ is lower semi-continuous as well, and by \eqref{eq_ldpchar2} applied to $A=\cX$,
\[
\lim_{N\to \infty} \frac{1}{N}\log \E \left[ e^{- N \cE^N(z^N)}\right]=-\inf_{\mu \in \cX} J(\mu),
\]
where the right-hand side is assumed to be finite. Now normalizing by $Z_N$ and proceeding as above it follows that for any $A\in\cB(\cX)$,
\[
-\inf_{\mu \in A^o} \, \cF(\mu) \leq\liminf_{N\to \infty} \frac{1}{N} \log Q^N(A) \leq \limsup_{N\to \infty} \frac{1}{N} \log Q^N(A)  \leq -\inf_{\mu \in \bar{A}} \, \cF(\mu).
\]
Along with the exponential tightness of $Q^N$, this implies that $(\cE^N,\cE)$ induces an LDP.
\end{proof}

\begin{remark}
	Notice that the classical Varadhan's Integral Lemma is recovered when $\cE$ is continuous and bounded, and $\cE^N=\cE$ for all $N\in\N$.
\end{remark}

\subsection{An extended Varadhan Integral Lemma}
Now we present the main result of this section.

\begin{thm}[Extended Varadhan Integral Lemma]\label{thm_direct}
Let $P^N=\law(z^N)$ be a family satisfying an LDP with rate function $I$. Let $(\cE_{\lambda}^N,\cE_{\lambda})$ be pairs inducing an LDP for all $\lambda>0$. Moreover, suppose that the pair $(\cE^N,\cE)$ is such that for some $\gamma>1$, 
\begin{subequations}\label{eq_unifexpint}
\begin{align}
\rate{e^{-\gamma N\cE_{\lambda}^N(z^N)}} <+\infty,\label{eq_unifexpinta}\\
\inf_{\mu\in D(I)} (I+\gamma\cE_\lambda)(\mu) >-\infty,\label{eq_unifexpintb}
\end{align}
\end{subequations}
for every $\lambda>0$, and that a constant $K\in\R$ exists, such that for every $\beta \in \R$,
	\begin{subequations}\label{eq_rconv5}
	\begin{align}
	\limsup_{\lambda\to 0}  \rate{e^{\beta N(\cE^N-\cE_{\lambda}^N)(z^N)}} &\le K \label{rconv5a},\\
	\limsup_{\lambda\to 0}  \sup_{\mu \in D(I)} \, \bigl(\beta (\cE_{\lambda}-\cE) - I\bigr)(\mu) &\le K \label{rconv5b}.
	\end{align}
	\end{subequations}
Then the family $\{Q^N\}$ defined by \eqref{eq:Q_induced-2} satisfies an LDP with rate function $\cF$. In particular, the pair $(\cE^N,\cE)$ also induces an LDP.
\end{thm}

\begin{remark}\label{rem_LDPcomm}
We give some comments on the assumptions:
\begin{enumerate}[label=(\roman*)]
\item It was shown in \cite{Hoek} that under \eqref{eq_unifexpintb}, condition \eqref{rconv5b} is equivalent to the uniform convergence of $\cE_\lambda$ to $\cE$ on the sub-level sets $\{ \mu \, |\, I(\mu)\leq M\}$ of $I$ for any $M\in\R$, and that \eqref{rconv5a} implies:
\begin{equation}\label{eq_ldpgu}
\text{For any $\delta>0$}:\quad \lim_{\lambda \to 0} \ratep{|\cE-\cE_{\lambda}|(z^N)>\delta}=-\infty. 
\end{equation}
\item Condition \eqref{rconv5b} could appear redundant for those who are familiar with LDPs: if $(\cE^N-\cE_{\lambda}^N,\cE-\cE_{\lambda})$ is \emph{a priori} known to induce an LDP, \eqref{rconv5b} and \eqref{rconv5a} are equivalent. But in the proof we need to approximate $\cE^N$ and $\cE$ separately, which requires us to have both conditions. Nevertheless, we do expect from this reasoning that bounds for \eqref{rconv5a} are also bounds for \eqref{rconv5b}. We will see that this indeed the case for the interacting particle systems in Section \ref{sgibbs}.
\item Note that condition \eqref{rconv5a} implies (cf.\ Lemma \ref{lm_convex2})
	\[
	\lim_{\lambda \to 0} \limsup_{N\to \infty}  \frac{1}{N} \log \E \left[ e^{-N \beta\cE^N_{\lambda}(z^N)}\,1_A \right] = \limsup_{N\to \infty}  \frac{1}{N} \log \E \left[ e^{-N \beta\cE^N(z^N)}\,1_A \right],
	\]
	and 
    \[
	\lim_{\lambda \to 0} \liminf_{N\to \infty}  \frac{1}{N} \log \E \left[ e^{-N \beta\cE^N_{\lambda}(z^N)}\,1_A \right] = \liminf_{N\to \infty}  \frac{1}{N} \log \E \left[ e^{-N \beta\cE^N(z^N)}\,1_A \right].
	\]
for any Borel set $A\in\cB(\cX)$ and $\beta\in\R$, which is considerably stronger than simply inequalities for open and closed sets. An open question is whether this is stronger than, equivalent to or weaker than (or none of the former) the statement that the induced measures $Q_{\lambda}^N$ \emph{exponentially approximates} $Q^N$ as $\lambda \to 0$ (cf.\ \cite[p.\ 130]{Dembo10}). 

\item An alternative approach to prove the type of results in Theorem~\ref{thm_direct} could be to get an LDP for $z^N$ in a larger space with a stronger topology, where $\mathcal{E}^N$ is a continuous function, and then apply the classical Varadhan Lemma; the LDP in the stronger topology could be obtained by exponential approximation, as in \cite{Dembo10}. This strategy is used, for example in \cite{Eichelsbacher2002,Liu2018}, in the context of certain singular Gibbs measures.
\end{enumerate}
\end{remark}

We briefly explain the strategy to prove Theorem \ref{thm_direct}. For a Borel set $A\in\cB(\cX)$, we define the following functionals on the space of Borel functions $\cE$ on $\cX$,
\begin{align}\label{phi_def}
\left\{\quad
\begin{aligned}
\phi_A(\cE)&:= -\inf_{\mu \in A\cap D(I)} (I+\cE)(\mu),\\
\phi^N_A(\cE)&:=\frac{1}{N}\log \E\left[e^{-N \cE(z^N)} 1_A \right],\qquad n\in\N.
\end{aligned}\right.
\end{align}
We will show in Lemma \ref{lm_ldpconv} that $\phi^N_A$ and $\phi_A$ are convex and from above (in $A$) by $\phi^N_{\cX}$ and $\phi_{\cX}$ respectively. Moreover, the fact that $(\cE^N,\cE)$ induces an LDP can be read as a set of variational inequalities for $\phi^N_A$ and $\phi_A$ for each $A\in\cB(\cX)$, i.e.
\[
\phi_{A^o}( \cE) \leq \liminf_{N\to \infty} \phi^N_A\big( \cE^N\big)
\leq \limsup_{N\to \infty} \phi^N_A\big( \cE^N\big)\leq\phi_{\bar{A}}( \cE).
\]
Finally, the convergence of \eqref{eq_rconv5} over $\cX$ will be seen to imply corresponding statements over every set $A\in\cB(\cX)$, which implies bounds on $\phi_A^N(\cE^N-\cE_{\lambda}^N)$ and $\phi_A(\cE-\cE_{\lambda})$. 
Hence the extended Varadhan Integral Lemma is morally equivalent to a type of stability of variational inequalities for convex functionals, for which we can use Theorem \ref{thm_conv1} in Appendix \ref{s_convex}.

\medskip

Here are the convexity properties and bounds for $\phi_A$ and $\phi_A^N$:

\begin{lm}\label{lm_ldpconv}
For any Borel set $A\in\cB(\cX)$ and any $N\in\N$, the functionals $\phi_A$ and $\phi^N_A$ defined in \eqref{phi_def} are convex and bounded from above by $\phi_\cX$ and $\phi^N_\cX$ respectively (that is, $\phi_A\le \phi_\cX$ and $\phi_A^N \le \phi_\cX^N$ for every $A\in\cB(\cX)$ and any $N\in\N$).
\end{lm}
\begin{proof}
For any $N\in\N$ and any $\alpha \in (0,1)$, for any Borel functions $\cE_1$, $\cE_2$, it holds by H\"older's inequality (with exponents $1/\alpha$ and $1/(1-\alpha)$)
\begin{align}\label{eq_set0}
\log \E \left[ e^{-N (\alpha \cE_1+(1-\alpha)\cE_2)(z^N)} \, 1_A \right] &= \log \E \left[ e^{- \alpha N  \cE_1 (z^N)} \, 1_A \, e^{-(1-\alpha)N \cE_{2}(z^N)} \, 1_A \right] \nonumber \\
&\leq \alpha \log \E  \left[ e^{-N \cE_1(z^N)} \, 1_A \right] + (1-\alpha) \log \E  \left[ e^{-N \cE_2(z^N)} \, 1_A \right].
\end{align}
Convexity of $\phi^N_A$ follows by dividing \eqref{eq_set0} by $N$. The bound $\phi^N_A\le \phi^N_\cX$ follows from the positivity of the exponential and the monotonicity of the logarithm. Finally, for any $\alpha \in (0,1)$,
	\begin{align*}
	\inf_{\mu \in A\cap D(I)} \, I+(\alpha \cE_1+(1-\alpha) \cE_2)&= \inf_{\mu \in A\cap D(I)} \, [\alpha (I+\cE_1)+(1-\alpha)(I+\cE_2)]\\
	&\geq \alpha \inf_{\mu \in A\cap D(I)} \, (I+\cE_1)+(1-\alpha) \inf_{\mu \in A\cap D(I)} \, (I+\cE_2),
	\end{align*}
which gives convexity of $\phi_A$. The bound $\phi_A\le \phi_\cX$ is easily verified.
\end{proof}
We can now prove Theorem \ref{thm_direct}.
\begin{proof}[Proof of Theorem \ref{thm_direct}]
Recall, for any Borel set $A\in\cB(\cX)$ and Borel function $\cG$, we have by definition
\[
\phi_{A^o}(\cG) = -\inf_{\mu \in A^o\cap D(I)} (I+\cG)(\mu),\qquad
\phi_{\bar A}(\cG) = -\inf_{\mu \in \bar{A}\cap D(I)} (I+\cG)(\mu),
\]
and
\[
\phi_A^N(\cG)=\frac{1}{N}\log \E\left[e^{-N \cG(z^N)} 1_A \right].
\]
Lemma \ref{lm_ldpconv} gives that $\phi_{A^o}$, $\phi_{\bar A}$ and $\phi_A^N$ are convex for every $N\in\N$. Moreover, by the bounds $\phi^N_A\le \phi^N_\cX$ and $\phi_{\bar{A}}\le \phi_\cX$, assumptions \eqref{eq_unifexpint} imply, for some $\gamma>1$ (independent of $\lambda$),
\begin{align*}
\left.
\begin{aligned}
    \limsup_{N\to\infty}\phi_A^N(\gamma\cE^N_\lambda) <+\infty\\
\phi_{\bar A}(\gamma\cE_\lambda) <+\infty
\end{aligned}\quad \right\}\quad \text{for every $\lambda>0$,}
\end{align*}
while assumptions \eqref{eq_rconv5} imply
\begin{align*}
\left.
\begin{aligned}
    \limsup_{\lambda\to0}\limsup_{N\to\infty}\phi_A^N(\beta(\cE^N-\cE^N_\lambda)) \le K\\
\limsup_{\lambda\to0}\phi_{\bar A}(\beta(\cE-\cE_\lambda)) \le K
\end{aligned}\quad\right\}\quad \text{for every $\beta\in\R$.}
\end{align*}
By Lemma \ref{lm_ldpchar}, the assumption that $(\cE_{\lambda}^N,\cE_{\lambda})$ induces an LDP is characterized by
	\begin{equation}\label{eq_ldpcrp2}
	\phi_{A^o}( \cE_{\lambda}) \leq \liminf_{N\to \infty} \phi_A^N\big( \cE_{\lambda}^N\big) \leq \limsup_{N\to \infty} \phi_A^N\big( \cE_{\lambda}^N\big)\leq\phi_{\bar A}( \cE_{\lambda}).
	\end{equation}
We are now in a position to apply Theorem \ref{thm_conv1}, which implies that \eqref{eq_ldpcrp2} also holds for $(\cE^N,\cE)$ (cf.\ \eqref{eq_ctx1}), i.e.\
\[
 \phi_{A^o}( \cE) \leq \liminf_{N\to \infty} \phi_A^N\big( \cE^N\big) \leq \limsup_{N\to \infty} \phi_A^N\big( \cE^N\big)\leq\phi_{\bar A}( \cE).
\]
Moreover, for every $\gamma'\in(0,\gamma)$, both $\limsup_{N\to\infty} \phi_A^N(\gamma'\cE^N)<+\infty$ and $\phi_{\bar A}(\gamma'\cE) <+\infty$ (cf.\ \eqref{eq_ctx1a} and \eqref{eq_ctx1b}). In particular, for $A=\cX$, we have that $-\phi_\cX(\cE)$ is finite, and  
\begin{equation}\label{eq_ldpthb1}
\rate{e^{- \gamma N \cE^N(z^N)}}<\infty.
\end{equation}
By Lemma \ref{lm_ldpchar}, we can conclude that $(\cE^N,\cE)$ induces an LDP provided we show lower semi-continuity of $J$ and exponential tightness of $Q^N$. 

First, from \eqref{rconv5b} we can conclude that for every $K'>K$ and $\beta\geq 0$ there exists a  large enough $\lambda^*(K',\beta)$ such that
\begin{equation*}
    |\cE_{\lambda}-\cE|(\mu)\leq \frac{K'+I(\mu)}{\beta}, \quad \forall \mu\in \cX, \forall \lambda\geq \lambda^*(K',\beta). 
\end{equation*}
In particular we derive that $\cE_{\lambda}$ converges pointwise to $\cE$ on $D(I)$ (in fact, the convergence is uniform on sub-level sets of $I$). Next, note that a Borel function $J$ is lower semi-continuous if and only if for every $\mu \in \cX$
\begin{equation*}
    \liminf_{\epsilon\to 0}\, \inf_{\nu\in B_{\epsilon}(\mu)}J(\nu)\geq J(\mu),
\end{equation*}
which in the case of
\[
J_{\cG}(\mu):= \left\{\begin{aligned}
&I(\mu) + \cG(\mu) \qquad &\mu\in D(I),\\
&+\infty \qquad &\mu \not \in D(I),
\end{aligned}\right.
\]
for a Borel function $\cG$ can be rewritten as 
\begin{equation}\label{eq:ldp_lsc1}
\limsup_{\epsilon\to 0} \phi_{B_{\epsilon}(\mu)}(\cG)\leq -J_{\cG}(\mu).
\end{equation}
To show this for $\cG=\cE$, fix $\mu$, and note that by convexity for any $\alpha\in [0,1)$, $\lambda$, $\epsilon$,
\begin{equation*}
\begin{aligned}
    \phi_{B_{\epsilon}(\mu)}(\alpha \cE)&\leq \alpha \phi_{B_{\epsilon}(\mu)}(\cE_{\lambda})+(1-\alpha)\phi_{B_{\epsilon}(\mu)}\left(\alpha(1-\alpha)^{-1}(\cE-\cE_{\lambda})\right)\\
    &\leq \alpha \phi_{B_{\epsilon}(\mu)}(\cE_{\lambda})+(1-\alpha)\phi_{\cX}\left(\alpha(1-\alpha)^{-1}(\cE-\cE_{\lambda})\right).\\
\end{aligned}
\end{equation*}
Since the left-hand side is independent of $\lambda$, taking subsequently limits in $\epsilon$ and $\lambda$ and using the lower semi-continuity of $I+\cE_{\lambda}$ we derive
\begin{equation*}
\begin{aligned}
   \limsup_{\epsilon\to 0} \phi_{B_{\epsilon}(\mu)}(\alpha \cE)&\leq \alpha \liminf_{\lambda\to 0} \limsup_{\epsilon\to 0} \phi_{B_{\epsilon}(\mu)}(\cE_{\lambda})+(1-\alpha)\limsup_{\lambda\to 0}\phi_{\cX}\left(\alpha(1-\alpha)^{-1}(\cE-\cE_{\lambda})\right)\\
    &\leq -\alpha \limsup_{\lambda\to 0} J_{\cE_{\lambda}}(\mu)+(1-\alpha)K.\\
\end{aligned}
\end{equation*}
Since $g(\alpha):=\limsup_{\epsilon\to 0} \phi_{B_{\epsilon}(\mu)}(\alpha \cE)$ is convex and bounded from above around $\alpha=1$ we conclude by Lemma \ref{lm_convex1} after letting $\alpha\to 1$,
\[ \limsup_{\epsilon\to 0} \phi_{B_{\epsilon}(\mu)}(\cE)\leq -\limsup_{\lambda\to 0}J_{\cE_{\lambda}}(\mu). \]
Now, to establish \eqref{eq:ldp_lsc1} for $\cG=\cE$, note that either $I(\mu)=+\infty$ in which case the inequality trivially holds, or we have $\mu \in D(I)$ and thus we employ the pointwise convergence of $\cE_{\lambda}$ to $\cE$.

Finally, to prove exponential tightness, fix an arbitrary $M\ge 1$ and let $K_M$ be a compact set in $\cX$ such that
\[
\limsup_{N\to\infty} \frac{1}{N}\log P^N(\cX\setminus K_M) < -M.
\]
By H\"older inequality we derive, for every $\alpha$ in $(0,1)$,
\begin{align*}
\frac{1}{N}\log \E\left[e^{-N \cE^N(z^N)} 1_{\cX\setminus K_M} \right]&\leq \frac{\alpha}{N}\log \E\left[e^{\alpha^{-1}N \cE^N(z^N)}  \right]+\frac{1-\alpha}{N}\log P^N\big(\cX\setminus K_M\big).
\end{align*}
After taking the limit supremum in $N$, the first term on the right-hand side is independent of $M$ and finite by \eqref{eq_ldpthb1}, provided $\alpha^{-1}<\gamma$. Therefore, 
\begin{align*}
\limsup_{M\to \infty}\limsup_{N\to\infty}\frac{1}{N}\log \E\left[e^{-N \cE^N(z^N)} 1_{\cX\setminus K_M} \right] =-\infty,
\end{align*}
which proves the exponential tightness of $Q^N$ and concludes the proof.
\end{proof}

\paragraph{\bf\em Background} We are only aware of one paper extending Varadhan lemma for discontinuous log-densities in a general framework, namely \cite{Moral2003}, which however uses different assumptions. Instead, extensions of Varadhan's lemma for particular contexts have been proven, often involving conditions like \eqref{eq_ldpgu}.
For example, \cite{Eichelsbacher1998} proves an extended contraction principle which is closely related to exponential approximations (cf.\ \cite{Dembo10}); a localized version of \eqref{eq_ldpgu} is used in \cite{Bodineau1999} for their concept of quasi-continuity to prove LDPs for vortex systems; the papers \cite{Eichelsbacher2002,Liu2018} use an alternative strategy to get an extension of the classical Varadhan Lemma in the context of singular Gibbs measures, see Remark \ref{rem_LDPcomm}(iv).

\section{Gibbs measures}\label{sgibbs}

\subsection{Notations and preliminary results}

In this section we consider the setting for Gibbs measures over weakly interacting particle systems.

Let $S$ be a Polish space, endowed with its Borel $\sigma$-algebra $\mathcal{B}(S)$. Let $\mu_0\in\cP(S)$ be a given reference measure. For simplicity, we will suppose in the following that $\mu_0$ is \textit{non-atomic}, i.e. it has no atoms (cf. Remark \ref{rmk_nonatomic}). We are given i.i.d.~random variables $\omega_i$, $i\in\mathbb{N}$, defined on some probability space $(\Omega,\mathcal{A},\mathbb{P})$, with values in $S$ and common law $\mu_0$; we can think of $\omega_i$ as non-interacting particles. We denote by $\mathbb{E}$ the expectation with respect to $\mathbb{P}$.

To describe our interacting particle system, we fix $k$ in $\mathbb{N}$, $k\ge 2$, and take a $k$-particle interaction potential $V^N$ on $S$, i.e.\ a Borel function $V^N:S^k\to \REx$. We define the energy $E_V^N:S^N \to \REx$ of an $N$-particle configuration ($N\gg k$) by
\[
E_V^N(x_1,\dots,x_N) := \frac{1}{N^k}\, \sum_{i_1,\ldots i_k \text{ distinct}} V^N(x_{i_1},\dots,x_{i_k}),
\]
(when the $N$-dependence is made explicit in the superscript, with a little abuse of notation we use $E^N_{V}$ instead of $E^N_{V^N}$). Then the interacting particle system is described by the probability measure $\mathbb{Q}^N_V$ on $(\Omega,\mathcal{A})$ defined by
\[
    \mathbb{Q}^N_V=
    \frac{1}{Z_V^N}e^{-NE_V^N(\omega_1,\ldots \omega_N)}\mathbb{P},
\]
where $Z_V^N$ is the normalizing constant, assumed to be finite. Under $\mathbb{Q}^N_V$, the particles $\omega_i$ are subject to interaction via the potential $V^N$. Note that the energy $E^N_V$ is invariant under permutation, or, in other words, depends only on the positions of the particle $\omega_i$ and not on their label $i$, via a fixed interaction potential $V^N$---this is a {\em mean-field} interaction. Note also that particle configurations $(x_1,\ldots,x_N)$ are more likely, according to $\mathbb{Q}^N_V$, if $V^N$ assumes lower values in these configurations.

Our main interest is in large deviations for the empirical measures associated with $\omega_i$ under the probability measure $\mathbb{Q}^N_V$. For this reason, we consider the state space $\cX=\cP(S)$, equipped with the weak topology (w.r.t.\ continuous and bounded functions on $S$), which turns $\cP(S)$ into a Polish space \cite[Theorem D.8]{Dembo10}.
For each $N\in\N$, we denote by $z^N_{\bullet}:S^N\to \cP(S)$ the continuous map 
\[
 S^N\ni(x_1,\ldots,x_N)=:\x\mapsto z^N_\x:= \frac{1}{N}\sum_{i=1}^N \delta_{x_i}\in \cP(S).
\]
We denote by $P^N \in \cP(\cP(S))$, resp. $Q^N_V \in \cP(\cP(S))$ the law of the empirical measure $z^N_{\o}$ for $\o=(\omega_1,\ldots,\omega_N)$ under $\mathbb{P}$ (where $\omega_i$ $i=1,\ldots N$ are i.i.d. with common law $\mu_0$), resp. under $\mathbb{Q}^N_V$. We aim at giving an LDP for $Q^N_V$.

We further define the function $\cE_V^N: \cP(S)\rightarrow \REx$ by
\begin{equation}\label{eqg_en}
\cE_V^N(\mu):= \begin{cases}
\int_{(S^k)'} V^N\,d \mu^{\otimes k} & \text{if $V^N\in L^1(\mu^{\otimes k})$},\\
+\infty & \text{otherwise},
\end{cases}
\end{equation}
where $(S^k)'$ is  $S^k$ but with the diagonals removed, i.e.\ 
\[
(S^k)':=\Bigl\{ (x_1,\dots,x_k) \in S^k \,\Big|\,  x_i\neq x_j, \, \forall i,j\in \{ 1,\dots,k\}\text{ with }i\neq j\Bigr\}.
\]
Notice that when $\omega_i$, $i=1,\ldots,N$ are i.i.d.\ random variables with common law $\mu_0$, then
\[
    E_V^N(\omega_1,\ldots, \omega_N) = \int_{(S^k)'} V^N d(z^N_{\o})^{\otimes k} = \cE_V^N(z^N_{\o})\qquad \text{$\mathbb{P}^{\otimes N}$-a.e.}
\]
due to the non-atomic property of the reference measure $\mu_0$. Hence, by construction and the mean-field interaction property, the interacting particle system may then be recast as a change of measure in $\cP(\cP(S))$, namely
\begin{equation}\label{eq:Q_induced-3}
 Q_V^N := \frac{1}{Z_V^N}e^{-N \cE^N_V(\mu)}P^N.
\end{equation}

Given a Borel function $V:S^k\rightarrow \REx$, we define similarly the function $\cE_V: \cP(S)\rightarrow \REx$ by
\begin{equation}\label{eqg_en1b}
\cE_V(\mu):= \begin{cases}
\int_{S^k} V\, d \mu^{\otimes k} & \text{if $V\in L^1(\mu^{\otimes k})$},\\
+\infty & \text{otherwise},
\end{cases}
\end{equation}
The functions $\cE_V^N$ and $\cE_V$ are Borel maps on $\cP(S)$ (cf.\ Appendix~\ref{app_meas}, in particular Lemma \ref{lm_intercon} and Corollary \ref{thm_minteru}) and are defined identically except for the $N$ dependence in $V$ and the domain of integration, i.e., $(S^k)'$ instead of $S^k$.

\begin{remark}\label{rmk_nonatomic}A few comments on the assumption that $\mu_0$ is non-atomic:
\begin{enumerate}[label=(\roman*)]
\item The energy $E_V^N$ as defined above does not rule out self-interaction, i.e.,\ two particles occupying the same position ($x_i=x_j$ for $i\ne j$). While $E_V^N$ is meaningful for bounded potentials $V^N$, this may cause an issue when $V^N$ is singular on the diagonal. The non-atomicity of $\mu_0$ resolves this issue: indeed, if $\omega_i$ are i.i.d.\ random variables on $(\Omega,\cA,\mathbb{P})$ with common law $\mu_0$, then $\mathbb{P}(\omega_i=\omega_j,\; i\ne j) = 0$, which then allows $E_V^N(\omega_1,\ldots,\omega_N)$ to be defined $\mathbb{P}$-a.s..

\item    However, if the reference measure $\mu_0$ is atomic but the energy for the $N$-particle system is still given by $E^N_V$ there is also an alternative method. Namely, fix $N$ and note that $E_V^N(\x)=E_V^N(\mathbf{y})$ for any $\mathbf{x},\mathbf{y}\in S^N$ with $z^N_{\x}=z^N_{\mathbf{y}}$. Hence the function $\cE^N_V: \cP(S)\rightarrow \REx$ 
\begin{equation*}
\cE^N_V(\mu):=\left\{\begin{aligned}
&E^N((z^N_{\bullet})^{-1}(\mu)) \qquad &&\mu\in z^N_{\bullet}(S^N), \\\
&+\infty \qquad &&\mbox{otherwise,}
\end{aligned}\right.
\end{equation*}
is well defined. Moreover, it is easy to verify that $z^N_{\bullet}(S^N)$ is closed and for $V^N\in C_b(S^N)$ the map $E^N((z^N_{\bullet})^{-1}(\mu))$ is continuous on $z^N_{\bullet}(S^N)$, hence $\cE^N_V(\mu)$ is Borel. By a monotone class argument similar as in Appendix~\ref{app_meas} one can extend this to all Borel $V^N$, and we can then proceed as in the rest of this section. However, note that $\cE_V^N$ might no longer be of integral form as in \eqref{eqg_en}.
\end{enumerate}
\end{remark}
Now we give the classical results concerning LDP. We start with the non-interacting case, namely Sanov theorem:

\begin{thm}[Sanov's theorem]\label{thm_sanov}
The family $\{P^N\}$ of laws of the non-interacting particle system satisfies an LDP with rate function $I:\cX\to \REx$, where
\[
I(\mu):=R(\mu\|\mu_0)
\]
is the relative entropy of $\mu \in \cX$ with respect to $\mu_0$.
\end{thm}
We recall that the relative entropy is defined as
\begin{equation*}
 R(\nu\|\mu):= \begin{cases}
    \int_S \log \frac{d \nu}{d \mu}\,  d \nu &\text{if $\nu\ll \mu$} \\
    +\infty &\text{otherwise}.
 \end{cases}
\end{equation*}
For the LDP for the interacting particle systems, we introduce the following notation:
\[
J_V(\mu) := 
\begin{cases}
\cE_V(\mu)+I(\mu) & \text{for $\mu \in D(I)$},\\
+ \infty &\mbox{otherwise},
\end{cases}
\]
with $D(I):=\{ \mu \, | \, R(\mu\|\mu_0) < \infty \}$, and, if $\inf_{\cX} J_V>-\infty$,
\begin{equation}\label{gibbs_ratefunction}
    \cF_V(\mu):=J_V(\mu)-\inf_{\nu \in \cX} J_V(\nu).
\end{equation}

We now give an LDP in the case when $V^N=V$ is in $C_b(S^k)$. In this case, $\cE_V$ is also continuous and bounded, and so the LDP for $Q^N_V$ is essentially a consequence of the classical Varadhan Lemma. The only (and technical) difference with the classical Varadhan Lemma comes from the missing diagonal in $(S^k)'$, which in general causes $\cE_V^N$ to not be continuous.

\begin{lm}\label{lmg_uapprox}
	Suppose $V:S^k\to \R$ is continuous and bounded. Then $(\cE^N_V,\cE_V)$ induces an LDP (in the sense of Definition \ref{defi:ldp_pair}).
	In particular, the family $\{Q^N_{V}\}$ given by \eqref{eq:Q_induced-3} satisfies an LDP with rate function $\cF_{V}$.
\end{lm}

\begin{proof}
By applying Lemma \ref{lm_intercon} $k$-times, we get that, for any continuous and bounded $V$, the function $\cE_V$ is continuous and bounded on $\cP(S)$. Hence, by the classical Varadhan Lemma (cf.\ Proposition \ref{thm_vara}), the couple $(\cE_V,\cE_V)$ induces an LDP in the sense of Definition \ref{def_ldp}.

Then, approximating $(\cE^N_V,\cE_V)$ with $(\cE_V,\cE_V)$, we have that $(\cE^N_V,\cE_V)$ induces an LDP by Theorem \ref{thm_direct}, provided we show that, for some $\gamma>1$ and $K\in\R$,
\begin{equation}\label{eq_grconvc}\
\begin{aligned}
\rate{e^{-\gamma N\cE_V^N(z^N_{\o})}} &< + \infty,\\
\rate{e^{\beta N|\cE_V^N-\cE_V|(z^N_{\o})}} &\le K\qquad\text{for every $\beta \geq 0$}.
\end{aligned}
\end{equation}
The first limit follows easily from the boundedness of $V$. For the second limit, we can bound away all the self-interactions to obtain
	\begin{align}\label{eq_gecomp}
	\begin{aligned}
	\big|\cE^N_V(z^N_{\o})-\cE_V(z^N_{\o})\big| &= \frac{1}{N^k}\left| \, \sum_{i_1,\dots, i_k\text{ all distinct}} V(\omega_{i_1},\dots,\omega_{i_k}) - \sum_{i_1, \dots, i_k} V(\omega_{i_1},\dots,\omega_{i_k}) \, \right|\\
	&\leq \frac{1}{N^k} \frac{k(k-1)}{2} N^{k-1} \|V\|_{\infty} = \frac{k(k-1)}{2 N}\|V\|_{\infty}.
	\end{aligned}
	\end{align}
In the inequality above, we used that the number of $k$-tuples $(i_1,\ldots i_k)$ with at least two equal indices is bounded by $N^{k-1} k(k-1)/2$. The second limit in \eqref{eq_grconvc} follows easily.
\end{proof}

Thus, we are in the same setting as in the previous section, i.e., we have created a large class of family of functions $(V^N,V)$ such that their induced interacting particle systems satisfy a certain LDP. Hence, next we will show how to extend this class by approximation. 

\subsection{Main results}

We give our main result for this section, which serves as a tool for LDPs for Gibbs measures with a possibly discontinuous interaction potential. The result brings the general LDP result of Theorem~\ref{thm_direct} into the Gibbs measure context. 
In the following, $(f)^-$ denotes the negative part of the function $f$.

\begin{thm}\label{thm_gibbs1}
Let $(V_\lambda^N,V_\lambda)$ be a family of Borel functions on $S^k$ such that $(\cE^N_{V_\lambda},\cE_{V_\lambda})$ induces an LDP. Let $(V^N,V)$ be a family of Borel functions on $S^k$ and assume that, for some $\gamma>1$,
\begin{equation}\label{eq_unif_int1}
\left.\begin{aligned}
\limsup_{N\to\infty}\log \int_{S^k} e^{\gamma k |(V^N_{\lambda})^-|}\, d\mu_0^{\otimes k} <+\infty\\
\log \int_{S^k} e^{\gamma k |(V_{\lambda})^-|}\, d\mu_0^{\otimes k} <+\infty
\end{aligned}\quad\right\}\quad \text{for every $\lambda>0$},
\end{equation}
and that, for some $K\in\R$,
\begin{equation}\label{eq_conv1}
\left.\begin{aligned}
\limsup_{\lambda\to 0} \limsup_{N\to\infty} \log \int_{S^k} e^{\beta |V^N-V^N_{\lambda}|} \, d\mu_0^{\otimes k}\le K,\\
\limsup_{\lambda\to 0} \log \int_{S^k} e^{\beta |V-V_{\lambda}|} \, d\mu_0^{\otimes k}\le K.
\end{aligned}\quad\right\}\quad\text{for every $\beta \ge 0$}.
\end{equation}
Then $(\cE^N_V,\cE_V)$ induces an LDP. In particular, the family of induced interacting particle systems $Q_{V}^N$ satisfies an LDP with normalized rate function $\cF_V$ given in \eqref{gibbs_ratefunction}.
\end{thm}
Similarly to Theorem \ref{thm_direct} for general LDPs, informally this results states that $Q^N_\lambda$ satisfy an LDP if there exists interacting potentials $V^N_\lambda$ and $V_\lambda$ which approximate $V^N$ and $V$ in an exponentially good way and whose corresponding interacting systems $Q^N_{V_\lambda}$ satisfy an LDP, for each $\lambda$. Again, this allows the following generalization:
\begin{enumerate}
\item It allows $V$ to be discontinuous.
\item The only requirement on LDPs for approximants is that $(\cE_{V_{\lambda}}^N,\cE_{V_{\lambda}})$ induces an LDP, not that $V_\lambda$ are continuous.
\item We can allow for the potential $V^N$ in the interacting particle system to depend on $N\in\N$ (cf.\ Remark \ref{rmk_sum_interactions} below).
\end{enumerate}
In the case where the sequence of functions $V^N$ is constant and equal to $V$, an LDP follows whenever $V$ satisfies the appropriate exponential moment condition. 
\begin{cory}\label{cor_gibbs2}
	Suppose that $V^N=V$ for all $N\in\N$, and such that
	\begin{equation}\label{eq_gexpm1}
\int_{S^k} e^{\beta |V|} \, d\mu_0^{\otimes k} <\infty\qquad \text{for all $\beta \ge 0$}.
	\end{equation}
Then $(\cE^N_V,\cE_V)$ induces an LDP.
\end{cory}

\begin{proof}
	The result follows from Lemma \ref{lmg_uapprox} and Theorem \ref{thm_qap1}. Indeed, by Theorem~\ref{thm_qap1}, with the choice $\mu=\mu_0^{\otimes k}$, $X=S^k$, there exists a sequence $(V_{\lambda})\subset C_b(S^k)$ such that
	\[
	\lim_{\lambda\to 0} \log \int_{S^k} e^{\beta |V-V_{\lambda}|} \, d\mu_0^{\otimes k} =0\qquad \text{for any $\beta\geq 0$}.
	\]
	Since the family $V_{\lambda}$ induces an LDP by Lemma \ref{lmg_uapprox}, one can see that \eqref{eq_unif_int1} and \eqref{eq_conv1} are satisfied. 
\end{proof}

\begin{remark}\label{rem:selfinteraction}
Recall that in the definition of the energy $E_V^N$ for the $N$-particle configuration we have excluded self-interaction, due to possible singularities in $V$. However, if $V$ is bounded (but not necessarily continuous), Corollary~\ref{cor_gibbs2} remains valid when the energy  does include self-interactions, i.e.,
\[
    E_V^N(x_1,\ldots,x_N) := \frac{1}{N^k} \sum_{i_1,\ldots,i_k} V(x_{i_1},\ldots,x_{i_k}).
\]
In this case, $(\cE_V^N,\cE_V)$ induces an LDP with
\[
    \cE_V^N(\mu) := \cE_V(\mu) =  \int_{S^k} V\,d \mu^{\otimes k}\qquad\text{for all $\mu\in\cP(S^k)$}.
\]
Indeed, this holds simply due to the estimate \eqref{eq_gecomp}.
\end{remark}

\begin{remark}\label{rmk_sum_interactions}
The setting of Theorem \ref{thm_gibbs1} includes also the case of interactions among different number of particles. For example, let ($N$-independent) interaction potentials $U_k:S^\ell\rightarrow \R$, $k=1,2,3$, be given and assume that the energy function $E^N_U$ is the sum of these three interactions, i.e.
\[
E^N_U (x_1,\ldots x_N) = \frac{1}{N^3}\sum_{i_1,i_2,i_3\text{ distinct}} \!\!U_3(x_{i_1},x_{i_2},x_{i_3}) + \frac{1}{N^2}\sum_{i_1\neq i_2}U_2(x_{i_1},x_{i_2}) +\frac{1}{N}\sum_{i_1}U_1(x_{i_1}).
\]
Therefore, by taking
\[
 V^N(x_1,x_2,x_3) = U_3(x_1,x_2,x_3) +\frac{N}{N-2}U_2(x_1,x_2) +\frac{N^2}{(N-1)(N-2)}U_1(x_1),
\]
we see that $E^N_U = E^N_V$ (for $x_1,\ldots x_N$ all distinct) and we are in the previous setting.
\end{remark}

The proof of Theorem \ref{thm_gibbs1} relies heavily on Theorem \ref{thm_direct} and the following bounds. These bounds convert the approximation properties for $V$ into those for $\cE_V$, needed to apply Theorem \ref{thm_direct}.

\begin{lm}\label{lm_gcomp}
	For any $k,N\in\N$ with $N>1$, and nonnegative Borel functions $V,\,V^N:S^k\to\REx$, the following inequalities hold true:
	\begin{subequations}\label{eq_gcomp1}
	\begin{align}
	\frac{1}{N} \log \E \left[ e^{N \cE_V^N(z^N_{\o})} \right] &\leq \frac{1}{k} \log \int_{S^k} e^{k\frac{N}{N-1}\, V^N} d\mu_0^{\otimes k},\label{eq_gcomp1a}\\
	\sup_{\mu\in D} \bigl( \cE_V(\mu)-R(\mu\|\mu_0)\bigr) &\leq \frac{1}{k} \log \int_{S^k} e^{k V}  d\mu_0^{\otimes k},\label{eq_gcomp1b}
	\end{align}
	\end{subequations}
	where $D=\{\nu\in\cP(S)\;|\; R(\nu\|\mu_0)<+\infty\}$.
\end{lm}

\begin{proof}
For the proof of \eqref{eq_gcomp1a}, we use the Hoeffding decomposition \cite{Hoe1963} for $\cE^N_V$, which reads
\[
\cE^N_V(z^N_{\o}) = \frac{N!}{N^k(N-k)!} \frac{1}{N!}\sum_{\sigma\in \mathcal{S}_N} \frac{1}{[N/k]}\sum_{j=1}^{[N/k]} V^N(\omega_{\sigma(jk-k+1)},\ldots \omega_{\sigma(jk)}),
\]
where $\mathcal{S}_N$ is the group of permutations of $\{1,\ldots N\}$ and $[N/k]$ denotes the integer part of $N/k$. The point of this decomposition is to group the elements $V^N(\omega_{i_1},\ldots \omega_{i_k})$ into an average (over possible permutations $\sigma$) of averages (over $j$) of $O(N)$ independent elements (that is, for fixed $\sigma$, the elements within the nested average are independent). 

By Jensen's inequality, applied to the exponential function and the average over $\sigma$,
\begin{align*}
\E \left[ e^{N \cE_V^N(z^N_{\o})} \right] &= \E\left[ \exp \left[ \frac{1}{N!}\sum_{\sigma\in \mathcal{S}_N} \frac{N!}{N^k(N-k)!} \frac{N}{[N/k]}\sum_{j=1}^{[N/k]} V^N(\omega_{(\sigma(jk-k+1))},\ldots \omega_{\sigma(jk)}) \right] \right]\\
&\le  \frac{1}{N!}\sum_{\sigma\in \mathcal{S}_N} \E\left[ \exp \left[ \frac{N!}{N^k(N-k)!} \frac{N}{[N/k]}\sum_{j=1}^{[N/k]} V^N(\omega_{(\sigma(jk-k+1))},\ldots \omega_{\sigma(jk)}) \right] \right]\\
&= \E\left[ \exp \left[ \frac{N!}{N^k(N-k)!} \frac{N}{[N/k]}\sum_{j=1}^{[N/k]} V^N(\omega_{jk-k+1},\ldots \omega_{jk}) \right] \right],
\end{align*}
where we used the fact that that $\omega_i$ are exchangeable (as i.i.d.) in the last line. Since the random variables $(\omega_{jk-k+1},\ldots \omega_{jk})$ are independent in $j$, we have therefore
\begin{align*}
\E \left[ e^{N \cE_V^N(z^N_{\o})} \right] &\le \E\left[ \prod_{j=1}^{[N/k]} \exp \left[ \frac{N!}{N^k(N-k)!} \frac{N}{[N/k]} V^N(\omega_{jk-k+1},\ldots \omega_{jk}) \right] \right]\\
& = \prod_{j=1}^{[N/k]} \E\left[ \exp \left[ \frac{N!}{N^k(N-k)!} \frac{N}{[N/k]} V^N(\omega_{jk-k+1},\ldots \omega_{jk}) \right] \right]\\
& = \E\left[ \exp \left[ \frac{N!}{N^k(N-k)!} \frac{N}{[N/k]} V^N(\omega_{1},\ldots \omega_{k}) \right] \right]^{[N/k]}, 
\end{align*}
where we used again that $\omega_i$ are exchangeable in the second equality. Since $N!\le N^k(N-k)!$ and $N/[N/k]\le N k/(N-1)$ for every $k,N\in\N$, $N>1$, we then obtain
\[
 \E \left[ e^{N \cE_V^N(z^N_{\o})} \right] \le \E\left[ \exp \left[ \frac{N}{N-1}k\,V^N(\omega_{1},\ldots \omega_{k}) \right] \right]^{[N/k]}.
\]
Taking the logarithm and noting that $[N/k]\le N/k$ for every $k,N\in\N$ yields
\[
  \frac{1}{N}\log\E \left[ e^{N \cE_V^N(z^N_{\o})} \right] \le \frac{1}{k}\log\E\left[ \exp \left[ \frac{N}{N-1}k\,V^N(\omega_{1},\ldots \omega_{k}) \right] \right],
\]
which concludes the proof of \eqref{eq_gcomp1a}.

To prove \eqref{eq_gcomp1b}, we first recall the additivity property of the relative entropy, i.e.
\[
R(\mu^{\otimes k}\|\mu_0^{\otimes k}) = k R(\mu\|\mu_0).
\]
Hence, for any $\mu \in D$, we have that
	\begin{align*}
	k \Bigl( \cE_V(\mu) - R(\mu\|\mu_0)\Bigr) 	= \int_{S^k} k V\, d\mu^{\otimes k} - R(\mu^{\otimes k}\|\mu_0^{\otimes k})
	\le \sup_{\nu \in D^k} \biggl\{\int_{S^k} kV\, d\nu - R(\nu\|\mu_0^{\otimes k})\biggr\},
	\end{align*}
	where $D^k=\{ \nu\in\cP(S^k)\;|\; R(\nu\|\mu_0^{\otimes k})<+\infty\}$.
	By Lemma \ref{lm_entvar} we have that
	\[
	\sup_{\nu \in D^k} \biggl\{\int_{S^k} k V d\nu - R(\nu\|\mu_0^{\otimes k}) \biggr\} = \log \int_{S^k} e^{k V} \,d\mu_0^{\otimes k}.
	\]
	Therefore, if the right-hand side is finite, the desired estimate \eqref{eq_gcomp1b} follows.
\end{proof}

\begin{proof}[Proof of Theorem \ref{thm_gibbs1}]
In order to apply Theorem \ref{thm_direct}, we have to show that there exists some $\gamma>1$, such that for every $\lambda>0$,
\begin{subequations}
    \begin{align}
	\rate{e^{-\gamma N\cE_{V_\lambda}^N(z^N_{\o})}} &<+\infty,\label{eq_gunifexpinta}\\
	\inf_{\mu\in D} \bigl(R(\mu\|\mu_0)+\gamma\cE_{V_\lambda}(\mu)\bigr) &>-\infty.\label{eq_gunifexpintb}
    \end{align}
\end{subequations}
and that, for some $K\in\R$, the following holds true for all $\beta\geq 0$:
\begin{subequations}
	\begin{align}
	\limsup_{\lambda\to 0}  \rate{e^{\beta N|\cE_V^N-\cE_{V_\lambda}^N|(z^N_{\o})}} &\le K \label{grconv5a},\\
	\limsup_{\lambda\to 0}  \sup_{\mu \in D} \, \bigl(\beta |\cE_V-\cE_{V_\lambda}|(\mu)-R(\mu\|\mu_0)\bigr)&\le K \label{grconv5b}.
	\end{align}
\end{subequations}

Due to linearity and \eqref{eq_gcomp1a} of Lemma~\ref{lm_gcomp}, we have for some $\gamma'\in(1,\gamma)$:
\[
 \frac{1}{N} \log \E \left[ e^{-\gamma' N \cE_{V_\lambda}^N(z^N_{\o})} \right] \le \frac{1}{N} \log \E \left[ e^{\gamma' N \cE_{\gamma'|(V_\lambda)^-|}^N(z^N_{\o})} \right] \le \frac{1}{k} \log \int_{S^k} e^{k\frac{N}{N-1}\,\gamma' |(V_\lambda^N)^-|}\, d\mu_0^{\otimes k}.
\]
Since $\gamma'N/(N-1) \le \gamma$ for all $N\ge N_\gamma:=\gamma/(\gamma-\gamma')$, we obtain from assumption \eqref{eq_unif_int1} the finiteness of the right-hand side uniformly in $N$ for $N\ge N_\gamma$. Hence, taking the $\limsup$ yields \eqref{eq_gunifexpinta}.

As for \eqref{eq_gunifexpintb}, we apply \eqref{eq_gcomp1b} of Lemma~\ref{lm_gcomp} to any $\mu\in \cP(S)$ with $R(\mu\|\mu_0)<+\infty$, to obtain
\[
 R(\mu\|\mu_0) + \gamma'\cE_{V_\lambda}(\mu) \ge R(\mu\|\mu_0) - \cE_{\gamma'|(V_\lambda)^-|}(\mu) \ge -\frac{1}{k} \log \int_{S^k} e^{\gamma' k |(V_\lambda)^-|}\,d\mu_0^{\otimes k} >-\infty.
\]

Similarly, we apply Lemma~\ref{lm_gcomp} to obtain
\begin{align*}
\frac{1}{N} \log \E \left[ e^{\beta N |\cE_V^N-\cE_{V_{\lambda}}^N|(z^N_{\o})} \right] 
&\le \frac{1}{k} \log \int_{S^k} e^{k\frac{N}{N-1}\,\beta |V^N-V^N_{\lambda}|} \, d\mu_0^{\otimes k},\\
 \sup_{\mu \in D} \, \big(\beta |\cE_V-\cE_{V_\lambda}|-I\big)(\mu) 
	&\leq \frac{1}{k} \log \int_{S^k} e^{k\beta|V-V_\lambda|}  \, d\mu_0^{\otimes k},
\end{align*}
which by assumption \eqref{eq_conv1} yields \eqref{grconv5a} and \eqref{grconv5b}. The conclusion follows applying Theorem \ref{thm_direct}.
\end{proof}

\paragraph{\bf\em Background}
As mentioned in Section \ref{LDPs}, various extension principles to singular functionals or contractions have arisen. See for example \cite{Eichelsbacher1998} on various generalizations of Sanov's theorem, in which a stronger topology is used---defined by the property that all functions $\cE_V$ for which $k=1$ and $V$ satisfies \eqref{eq_gexpm1} are continuous with respect to this topology. This was subsequently generalized in \cite{Eichelsbacher2002} to the case $k\ge 2$.
Note that an application of the classical Varadhan Lemma to the result in \cite{Eichelsbacher2002} should yield a very similar result to our Corollary \ref{cor_gibbs2}; this type of argument has been used in \cite[Lemma 2.4]{EicZaj2003} in the proof of a moderate deviation principle for a bounded interaction kernel. More recently, in \cite{Liu2018}, a similar setting as this section was studied, i.e.,\ large deviations for mean field Gibbs measures on Polish spaces, involving singular potentials. In particular, their results also include the case of Corollary \ref{cor_gibbs2}. 

Other LDP results for Gibbs measures with singular potentials have been proven, see e.g.\ \cite{Berman2016a,Chafai2014,Dupuis2020,Rey2018}, with locally compact base space $S$, \cite{Zelada2017} with potentials satisfying a bound from below and a lower semi-continuity assumption. In particular in \cite{Berman2016a} the case where \eqref{eq_gcomp1} only holds for some $\beta$ instead of all---which, in $\R^d$, allows for potentials $V$ with a logarithmic singularity---is considered, on compact Polish spaces and assuming $V$ lower semi-continuous.

It should be noted that inequalities related to \eqref{eq_gcomp1a}, which in our case is derived from the Hoeffding decomposition has also arisen in different context and different names. For example, in \cite{Eichelsbacher2002} a similar inequality stems from the existence of regular partitions of complete hypergraphs due to Baranyai, and in \cite{Liu2018} modifications of decoupling inequalities of de la Pe\~na were used. Hoeffding decomposition has been used directly in some generalizations to \cite{Eichelsbacher2002}, for example \cite{Eic2004}.

\subsection{Extension to Gibbs-like potentials}\label{s_gibbs_like}

The previous results use the Gibbs structure of the potential $\cE_V$ to reduce the LDP problem to the context of Section \ref{LDPs}. However, for this reduction to hold, only some Gibbs-like bounds are needed. This allows to prove LDPs for empirical measures not coming from Gibbs laws, as soon as Gibbs-type bounds are possible. As we will see in the example of Subsection \ref{s_ex3}, this is the case of interacting diffusions where the drift depends non-linearly on the empirical measure and/or on its $k$-times tensor product.

As before, let $S$ be a Polish space with its Borel $\sigma$-algebra $\cB(S)$, $\mu_0\in\cP(S)$ be a given reference measure. The state space for the LDP is $\cX=\cP(S)$, equipped with the weak topology. As before, $P^N$ denotes the law of the empirical measure $z^N_{\o}$, where $\omega_i$, $i=1,\ldots N$ are i.i.d.~random variables defined on some probability space $(\Omega,\mathcal{A},\mathbb{P})$, with values in $S$ and common law $\mu_0$; $\mathbb{E}$ denotes the expectation with respect to $\mathbb{P}$. Now let $\mathcal{E}^N:\cP(S)\rightarrow \REx$, $N\in \mathbb{N}$ and $\mathcal{E}:\cP(S)\rightarrow \REx$ be Borel functions, let $Q^N\ll P^N$ be the probability measure on $\cP(S)$ given by
\begin{align*}
    Q^N = \frac{1}{Z^N}e^{-N\mathcal{E}^N(\mu)}P^N,
\end{align*}
where $Z^N$ is the renormalization constant, assumed to be finite. We further recall the notations $\mathcal{E}^N_V(\mu)$ and $\mathcal{E}^N_V(\mu)$ given in \eqref{eqg_en} and \eqref{eqg_en1b} for any Borel function $V:S^k\rightarrow \REx$. As before, we denote $D=\{\nu\in\cP(S)\;|\; R(\nu\|\mu_0)<\infty\}$.

\medskip

Our main result is the following:

\begin{thm}\label{thm_gibbs2}
Let $(\mathcal{E}_\lambda^N,\mathcal{E}_\lambda)$ be a family of LDP inducing pairs for all $\lambda>0$ such that \eqref{eq_unifexpint} holds for some $\gamma>1$ (independent of $\lambda$). Assume that, for every $\lambda>0$ and every $\beta\in \R$, for every $\mu$ in $D$, there holds
\begin{subequations}
\begin{gather}
    \limsup_{N\to\infty} \frac{1}{N}\log \E \left[ e^{\beta N (\mathcal{E}^N -\mathcal{E}_\lambda^N)(z^N_{\o})} \right]
    \leq C + C\,\limsup_{N\to\infty}\frac{1}{N}\log \E \left[ e^{c_\beta N \cE^N_{G_\lambda^N}(z^N_{\o})} \right],\label{eq_Gibbs_ineq_1}\\
    |\beta(\cE-\cE_\lambda)(\mu)| -R(\mu\|\mu_0)
    \leq C + C \log \int_{S} \exp\left(c_\beta \int_{S^{k-1}} G_\lambda(x,y) \,d\mu^{\otimes(k-1)}(y)\right) d\mu_0(x),\label{eq_Gibbs_ineq_2}
\end{gather}
\end{subequations}
for some constant $C>0$ independent of $\beta$, $\lambda$ and $\mu$, some $c_\beta\ge0$ independent of $\lambda$ and $\mu$, and some nonnegative Borel functions $G^N_\lambda,G_\lambda:S^k \rightarrow \REx$ (independent of $\beta$ and $\mu$). Assume also that $G^N_\lambda$ and $G_\lambda$ satisfy, for some $K\in\R$ independent of $\beta$, for every $\beta \ge0$,
\begin{subequations}
 \begin{align}
    \limsup_{\lambda\to 0}\limsup_{N\rightarrow \infty} \int_{S^k} e^{\beta G^N_\lambda} d\mu_0^{\otimes k} &\le K,\label{eq_exp_int_1}\\
    \limsup_{\lambda\to 0} \int_{S^k} e^{\beta G_\lambda} d\mu_0^{\otimes k} &\le K.\label{eq_exp_int_2}
\end{align}
\end{subequations}
Then the pair $(\mathcal{E}^N, \mathcal{E})$ induces an LDP with the normalized rate function $\cF_V$ in \eqref{gibbs_ratefunction}.
\end{thm}

\begin{remark}
    A simple condition for the inequalities \eqref{eq_Gibbs_ineq_1} and \eqref{eq_Gibbs_ineq_2} to hold is if
    \[
        \left|\cE^N-\cE^N_{\lambda}\right|(\mu)\leq \cE^N_{G^N_{\lambda}}(\mu),\qquad
    \left|\cE-\cE_{\lambda}\right|(\mu)\leq \cE_{G_{\lambda}}(\mu)\qquad\text{for all $\mu\in \cP(\cX)$}.
    \]
\end{remark}

\begin{proof}
The proof is similar to that of Theorem \ref{thm_gibbs1}. The result follows from Theorem \ref{thm_direct} provided we verify condition \eqref{eq_rconv5}. For condition \eqref{rconv5a}, by assumption \eqref{eq_Gibbs_ineq_1} and Lemma \ref{lm_gcomp}, we have
\begin{align*}
    \limsup_{\lambda\to 0} \limsup_{N\to \infty}\frac{1}{N}\log \E \left[ e^{\beta N (\cE^N -\cE_\lambda)(z^N_{\o})} \right]
    &\leq C + C\limsup_{\lambda\to 0} \limsup_{N\to \infty}\frac{1}{N}\log \E\left[ e^{c_\beta \cE^N_{G_\lambda}(z^N_{\o})} \right]\\
    &\leq C + C\limsup_{\lambda\to 0} \limsup_{N\to \infty}\frac{1}{k}\log \int_{S^k}e^{k\frac{N}{N-1}\,c_\beta G^N_\lambda}d\mu_0^{\otimes k},
\end{align*}
for all $\beta\in\R$. Hence, by assumption \eqref{eq_exp_int_1}, we get
\begin{align*}
    \limsup_{\lambda\to 0} \limsup_{N\to \infty} \frac{1}{N}\log \E \left[ e^{\beta N (\cE^N -\cE^N_\lambda)(z^N_{\o})} \right] \leq C +C\frac{\log K}{k}.
\end{align*}
For condition \eqref{rconv5b}, by assumption \eqref{eq_Gibbs_ineq_2} and Lemma \ref{lm_estrel}, we have, for every $\mu$ in $D$, for every $\lambda>0$ and every $\beta\in\R$,
\begin{align*}
    &|\beta(\cE-\cE_\lambda)(\mu)| -(1+C(k-1)) R(\mu\|\mu_0)\\
    &\leq -C(k-1)R(\mu\|\mu_0) +C +C \log \int_{S} \exp\left(c_\beta \int_{S^{k-1}} G_\lambda(x,y) \,d\mu^{\otimes(k-1)}(y)\right) d\mu_0(x)\\
    &\leq -C(k-1)R(\mu\|\mu_0) +C +C(k-1)R(\mu\|\mu_0) +C \log \int_{S^k} e^{c_\beta G_\lambda} d\mu_0^{\otimes k}
\end{align*}
Hence, taking the $\sup$ over $\mu$ in $D$ and then the $\limsup$ over $\lambda$, by assumption \eqref{eq_exp_int_2} we get
\begin{align*}
    \limsup_{\lambda\to 0} \sup_{\mu \in D} |\beta(\cE-\cE_\lambda)(\mu)| -(1+C(k-1)) R(\mu\|\mu_0) \le C +C \log K.
\end{align*}
Condition \eqref{rconv5b} follows by simply dividing by $1+C(k-1)$. The proof is complete.
\end{proof}

\begin{remark}\label{rmk_bd_Gibbs_like}
As the above proof shows (and as consequence of Lemmas \ref{lm_gcomp} and \ref{lm_estrel}), the assumptions \eqref{eq_exp_int_1} and \eqref{eq_exp_int_2} imply, respectively, the following two bounds:
\begin{align*}
&\limsup_{\lambda\to 0} \limsup_{N\to \infty}\frac{1}{N}\log \E\left[ e^{c_\beta \cE^N_{G_\lambda}(z^N_{\o})} \right] \le \frac{\log K}{k} <\infty,\\
&\limsup_{\lambda\to 0} \log \int_{S} \exp\left(c_\beta \int_{S^{k-1}} G_\lambda(x,y) \,d\mu^{\otimes(k-1)}(y)\right) d\mu_0(x) \le (k-1)R(\mu\|\mu_0) +\log K <\infty.
\end{align*}
\end{remark}

\section{Singularly interacting diffusions}\label{s_process}

\subsection{Notations and preliminary results}\label{s_proc_expl}

In this section we study LDPs for a system of mean field interacting diffusions, defined by a singular drift. We will put this problem in the framework of Gibbs-like structure of the previous section (cf.\ Section~\ref{s_gibbs_like}).

For $N\in \N$, we consider the following system of interacting SDEs
\begin{align}
\left\{\quad\begin{aligned}\label{eq_SDEm}
 &d X^{N,i}_t = b_t^N\left(X_t^{N,i},\frac{1}{N}\sum_{j=1}^N\delta_{X^{N,j}_t}\right) d t + d W^i_t,\quad i=1,\ldots N,\\
&X^{N,i}_0  \text{ i.i.d~with law } \rho_0.
\end{aligned}\right.
\end{align}
Here the {\em drift} $b^N:[0,T]\times \R^d\times \cP(\R^d)\rightarrow \R^d$ is a Borel map, where $\cP(\R^d)$ is endowed with the (metrizable, complete and separable) topology of weak convergence. The processes $W^i$, $i\in \N$, are independent $d$-dimensional Brownian motions on a standard filtered probability space $(\Omega,\cA,(\cF_t)_t,\mathbb{P})$, $\rho_0\in\cP(\R^d)$ is the law of the i.i.d.~initial data $X^{N,i}_0$.

We now set $S:=C([0,T];\R^d)$, the space of continuous paths in $\R^d$, and note that $X^{N,i} \in S$ for each $i=1,\ldots,N$. Moreover, we set $\tilde Q_{b^N}^N \in \cP(S^N)$ to be the law of $\boldsymbol{X}^N:=(X^{N,1},\ldots,X^{N,N})$ defined by the system \eqref{eq_SDEm}. We set $\mathbb{W} \in\cP(S)$ to be the Wiener measure with $\rho_0$ as marginal at time $0$ and $\tilde P^N\in \cP(S^N)$ the law of $N\in\N$ independent Brownian motions, i.e.\ $\tilde P^N:=\mathbb{W}^{\otimes N}$. With a little abuse of notation, we will consider $W^i$ to be $d$-dimensional independent Brownian motions with initial law $\rho_0$, unless differently specified; similarly, we will use $W$ for a $d$-dimensional Brownian motion with initial law $\rho_0$, unless differently specified.

Consider the \emph{empirical process} $z^N_{\boldsymbol{X}}\in \cP(S)$,  
\begin{equation*}
z^N_{\boldsymbol{X}}:=\frac{1}{N}\sum_{i=1}^N \delta_{X^{N,i}},
\end{equation*}
and let $Q_{b^N}^N \in \cP(\cP(S))$ be the laws of the random variable $z^N_{\boldsymbol{X}}$ induced by $\tilde Q_{b^N}^N$. Similarly, let $P^N$ be the law for the empirical process $z^N_{\boldsymbol{W}}$ of the non-interacting system induced by $\tilde P^N$. We will use $z^N_\x$, for $\x \in S^N$, to denote the empirical measure associated with $\x$, and $z^N$ for the canonical process (that is, the identity process) on $\cP(S)$. We will often use the notation $\langle f,\mu\rangle$ to denote $\int f \,d\mu$.

Recall that, as in Section \ref{sgibbs}, the sequence $P^N$ satisfies an LDP with rate function
\[
I:\cP(S)\to\REx;\qquad I(\mu):=R(\mu\|\mathbb{W}),
\]
where $R(\mu\|\mathbb{W})$ is the relative entropy of $\mu$ with respect to the Wiener measure $\mathbb{W}$.

For the particle system, a Gibbs-like representation holds. Indeed, for a Borel map $b:[0,T]\times\R^d\times \cP(\R^d)\rightarrow\R^d$, we introduce the potentials
\begin{align}
V_b^N(x,\mu) := -V_b^{N,2}(x,\mu) + \frac{1}{2}V_b^1(x,\mu)\quad\text{with}\quad
\left\{\begin{aligned}
    V_b^1(x,\mu) &:= \int_0^T |b_t(x_t,\mu_t)|^2 dt, \\
    V_b^{N,2}(x,\mu) &:= \int_0^T b_t(x_t,\mu_t)\cdot dx_t,
\end{aligned}\right.\label{eq_V1_V2}
\end{align}
where $V_b^{N,2}:S\times\cP(S)\to\R$ is defined as stochastic integral under the law of $(W^i,z^N_{\boldsymbol{W}})$ on $(x,\mu)$. We define the corresponding log-density as
\begin{align*}
    \cE_b^N(\mu) := \begin{cases}
        \int_S V_b^N(x,\mu)\,d\mu(x) & \text{if } V^1_b(\cdot,\mu),V^{N,2}_b(\cdot,\mu) \in L^1(\mu),\\
        0 & \text{otherwise},
    \end{cases}\qquad \mu\in\cP(S).
\end{align*}
See Section \ref{s_def_log_densities} for the precise definition and measurability properties of $\cE_b^N$. Similarly we define $V_b$ and $\cE_b$ by replacing $V^{N,2}_b$ with
\begin{align*}
    V_b^2(x,\mu) := \int_0^T b_t(x_t,\mu_t)\cdot dx_t,
\end{align*}
now as stochastic integral at a deterministic $\mu$. See again Section \ref{s_def_log_densities} for the precise definition of $\cE_b$ and its measurability properties. The Gibbs-like representation of the particle system is given by the following lemma, which is essentially a consequence of Girsanov theorem and the mean field form of the interaction ($b^N$ depending on $z^N_{\boldsymbol{X}}$), the details of the proof are postponed to Subsection \ref{s_proof_SDE_Gibbslike}.

\begin{lm}\label{lm_prep1}
Fix $N\in\N$. Assume that
\begin{align*}
\E\left[\exp\left(\frac{N}{2} \int_S \int_0^T |b^N_t(x_t,z^N_{\boldsymbol{W},t})|^2 dt\, dz^N_{\boldsymbol{W},t}(x)\right)\right]<\infty.
\end{align*}
Then there exists a weak solution to the system \eqref{eq_SDEm}, which is unique under the constraint 
\begin{align}
    \int_0^T |b^N_t(x_t,z^N_t)|^2 dt \in L^1(S,z^N) \qquad \text{for }Q^N_{b^N}\text{-a.e. }z^N. \label{eq_particle_uniq}
\end{align}
For this law, we have the following representations:
\begin{align}
\frac{d \tilde{Q}_{b^N}^N}{d \tilde{P}^N}(x^1,\ldots x^N) =e^{-N \cE^N_{b^N}(z^N_\x)},\qquad\qquad
\frac{d Q_{b^N}^N}{d P^N}(\mu) =e^{-N \cE^N_{b^N}(\mu)}.\label{eq_Girs_density_particle}
\end{align}
\end{lm}

Now we provide a class of drifts $(b^N,b)$ which induce an LDP for the law $Q^N_{b^N}$. Morally, this is the class of Lipschitz drifts (with respect to the $1$-Wasserstein distance). The class of $\FLip$-inducing pairs in the definition below allows for some margin to include the case of drifts without self-interactions, similarly to Lemma \ref{lmg_uapprox}. In the following, $\cP_1(\R^d)$ denotes the subset of $\cP(\R^d)$ of all probability measures on $\R^d$ with finite first moment; $W_1$ denotes the $1$-Wasserstein distance on $\cP_1(\R^d)$.

\begin{defi}
The class $\FLip$ consists bounded Borel functions $b:[0,T]\times\R^d\times\cP(\R^d)\to\R^d$ such that the map $(x,\mu)\mapsto b(t,x,\mu)$ is globally Lipschitz continuous on $\cP_1(\R^d)$ with respect to the $1$-Wasserstein distance, uniformly in $t\in [0,T]$, i.e., for some $M_b\ge 0$,
\begin{equation*}
|b(t,x,\mu)-b(t,y,\nu)|\le M_b \bigl( |x-y| + W_1(\mu,\nu)\bigr),
\end{equation*}
for all $t\in[0,T]$, $(x,y)\in\R^d$ and $\mu,\nu\in\cP_1(\R^d)$.

Moreover, the pair $(\{b^N\},b)$ (in short $(b^N,b)$) is called $\FLip$-inducing (subjected to $\{P^N\}$) if it satisfies the following conditions:
 \begin{enumerate}
 \item $b\in \FLip$;
 \item The sequence $b^N$ is uniformly bounded, i.e. $\sup_{N\in\N} \|b^N\|_{\infty}<\infty$;
 \item There exists a sequence $c_N$ with $c_N\to 0$ as $N\to\infty$ such that 
     \begin{equation}\label{eq_FLipver2}
         \int_0^T \bigl\langle |b_t-b_{t}^N|^2(\cdot,z^N_t),z^N_t \bigr\rangle \, d t \leq c_N,\qquad\text{for $P^N$-almost every $z^N$}.
     \end{equation}
 \end{enumerate}
\smallskip
\end{defi}

\begin{lm}\label{lem_LDP_b_Lip}
	Assume that the initial law $\rho_0$ satisfies
	\begin{align}
	    \int_{\R^d} e^{\beta |x|}\, d\rho_0 (x)<\infty,\qquad \forall \beta>0.\label{eq_IC_exp_int}
	\end{align}
	Assume that $(b^N,b)$ is $\FLip$-inducing. Then $(\cE^N_{b^N},\cE_b)$ induces an LDP. 
\end{lm}

\begin{remark}
    The exponential condition \eqref{eq_IC_exp_int} is required only to apply \cite[Theorem 34]{CDFM2018}: that result needs \eqref{eq_IC_exp_int} because it works with the 1-Wasserstein topology instead of the weak topology (the LDP in weak topology follows then by contraction principle). For this reason we suspect that, in the space $\cP(S)$ with the weak topology, the condition \eqref{eq_IC_exp_int} is not necessary.
\end{remark}

\begin{proof}[Proof of Lemma~\ref{lem_LDP_b_Lip}]
The LDP induced by $(\cE^N_b,\cE_b)$ follows from \cite[Theorem 34]{CDFM2018} (see also \cite{Fischer2014}). There an LDP is proved for $Q^N_b$ with rate function $R(\mu\| \mathbb{W}^\mu)$, where $\mathbb{W}^\mu$ is defined as in Theorem \ref{thm_pmainlip} (as the law of the solution to \eqref{eq_SDEc2b}). By Lemma \ref{lm_prep2}, we have
	\begin{align*}
	    R(\mu \| \mathbb{W}^\mu) = R(\mu \| \mathbb{W}) -\mathcal{E}_b(\mu).
	\end{align*}
	In particular, the infimum of the $R(\cdot\| \mathbb{W})-\cE_b$ is the infimum of the left-hand side above, that is $0$. The Laplace principle is then trivially satisfied as $e^{-N\cE_b}(\mu)$ is the density of the Girsanov transform.\\

Hence, note that by Theorem \ref{thm_direct} it is enough to show \eqref{eq_unifexpint} and \eqref{rconv5a} for $\cE_{\lambda}^N:=\cE^N_{b^N}$ and $\cE_{\lambda}=\cE^N=\cE:=\cE_{b}$. To derive \eqref{eq_unifexpinta}, we have for any $\gamma\in \R$, 
\begin{equation}\label{eq_flipsb1}
\begin{aligned}
    \E\left[ e^{\gamma N \langle V^{N,2}_{b^N}(\cdot, z^N_{\boldsymbol{W}}),z^N_{\boldsymbol{W}}\rangle } \right] &= \E\left[ e^{\gamma \sum_{i=1}^N \int_0^T b^N_{t}(W_t^i, z^N_{\boldsymbol{W},t})\cdot dW_t^i } \right] \\
    &\le \E\left[ e^{2\gamma^2 \sum_{i=1}^N \int_0^T |b^N_{t}(W_t^i, z^N_{\boldsymbol{W},t})|^2 d t } \right]^{\frac{1}{2}} \le e^{\gamma^2 N T\|b^N\|_{\infty}^2}.
\end{aligned}
\end{equation}
where Lemma~\ref{lm_estnk} was applied in the first inequality. Since $V_{b^N}^1$ is trivially bounded by $T\|b^N\|_\infty^2$, we then obtain
\begin{align*}
    \E\left[e^{-\gamma N\mathcal{E}^N(z^N_{\boldsymbol{W}})}\right] \le e^{(\gamma^2+\gamma/2) N T\|b^N\|_{\infty}^2}.
\end{align*}
Since $\sup_{N\in\N}\|b^N\|_{\infty}:=M <\infty$, \eqref{eq_unifexpinta} is satisfied.

Moreover, using Lemma~\ref{lm_entvar} we have for any $\gamma>1$ and any $\mu\in\cP(S)$ with $R(\mu\|\mathbb{W})<\infty$, that
\begin{align*}
    \gamma \cE_{b}(\mu) \le R(\mu\|\mathbb{W}) + \log \E \left[ e^{\gamma |V_{b}|(W,\mu)} \right].
\end{align*}
In a similar fashion as before, we can estimate the second term on the right-hand side uniformly in $\mu$ to obtain
\begin{align*}
 R(\mu\|\mathbb{W}) + \gamma\cE_{b}(\mu) \ge -a,
\end{align*}
for some constant $a_0>0$ independent of $\mu\in\cP(S)$, which gives \eqref{eq_unifexpintb}. \\

For \eqref{rconv5a}, we will consider the part of $\cE^N_{b^N}-\cE^N_{b}$ determined by $V^1$ and $V^{N,2}$ separately. First, note that $V^{N,2}_{b^N}-V^{N,2}_{b}=V^{N,2}_{b^N-b}$ and similar to the argument of \eqref{eq_flipsb1} we have for all $\beta\in \R$
\[ 
 \E\left[ e^{\beta N \langle V^{N,2}_{b^N-b}(\cdot, z^N_{\boldsymbol{W}}),z^N_{\boldsymbol{W}}\rangle } \right]   \leq  \E\left[ e^{2\beta^2 \sum_{i=1}^N \int_0^T |b^N_{t}-b_t|^2(W_t^i, z^N_{\boldsymbol{W},t}) d t } \right]^{1/2}  \leq e^{c_N N \beta^2 }. \]
Secondly, since 
\[ \beta \left\|b^N|^2-|b|^2\right|\leq \beta (|b^N|+|b|)|b^N-b| \leq \tfrac{1}{2}(M+\|b\|_{\infty})^2+\tfrac{\beta^2}{2}|b^N-b|^2,   \]
we derive
\begin{equation*}
 \begin{aligned}
 \E\left[ e^{\beta N \left(\langle V^{1}_{b^N}(\cdot, z^N_{\boldsymbol{W}}),z^N_{\boldsymbol{W}}\rangle-\langle V^{1}_{b}(\cdot, z^N_{\boldsymbol{W}}),z^N_{\boldsymbol{W}}\rangle\right) } \right] &= \E\left[ e^{\beta \sum_{i=1}^N \int_0^T \left(|b^N_t|^2-|b_t|^2\right)(W_t^i, z^N_{\boldsymbol{W},t}) d t  } \right] \\
 &\leq \E\left[ e^{\tfrac{TN}{2}(M+\|b\|_{\infty})+\tfrac{\beta^2}{2}\sum_{i=1}^N \int_0^T \left(|b^N_t-b_t|^2\right)(W_t^i, z^N_{\boldsymbol{W},t}) d t  } \right] \\
 &\leq e^{\tfrac{TN}{2}(M+\|b\|_{\infty})+\tfrac{c_N N \beta^2}{2}  }.
 \end{aligned}
 \end{equation*}
Finally, via Cauchy-Schwarz, we conclude that for all $\beta\in \R$
\begin{equation*}
    \limsup_{N\to \infty} \frac{1}{N} \log \E \left[ e^{ \beta\left(\cE^N_{b^N}(z^N_{\boldsymbol{W}})-\cE^N_{b}(z^N_{\boldsymbol{W}})\right) } \right] \leq \tfrac{T}{2} (M+\|b\|_{\infty}). 
\end{equation*}
\end{proof}

\subsection{The main result}

Now we give the main result of this section, which states the LDP for the system \eqref{eq_SDEm}. The assumptions may seem involved at first glance, but their meaning is not difficult: we have an LDP for the system \eqref{eq_SDEm} as soon as we can approximate in a suitable way the drift $b$ by regular drifts $b_\lambda$, along the Brownian empirical measure $z^N_{\boldsymbol{W}}$ and the couple of Brownian path $W$ and measure $\mu$ with finite relative entropy.

\begin{thm}\label{thm_pmainlip}
	Assume the condition \eqref{eq_IC_exp_int} on the initial law $\rho_0$. Suppose there exists a sequence of $\FLip$-inducing drifts $(b^N_\lambda,b_{\lambda})_{\lambda>0}$  and a sequence $(g_\lambda)_{\lambda>0}$ of Borel functions $g_\lambda:[0,T]\times (\R^d)^k\to [0,+\infty)$, $k\in\N$, such that the following conditions hold:
	\begin{enumerate}
	    \item[(i)] for every $\lambda>0$, for $P^N$-almost every $z^N\in\cP(S)$,
	    \begin{equation}\label{eq_b_g_n}
        \int_0^T \bigl\langle |b_t^N-b_{\lambda,t}^N|^2(\cdot,z^N_t),z^N_t \bigr\rangle\,dt \le \int_0^T \int_{((\R^d)^{k})'} g_\lambda (t,x_1,\ldots,x_k)\, d(z^N_t)^{\otimes k}\,dt;
     \end{equation}
     \item[(ii)] for every $\lambda>0$ and every $\mu\in\cP(S)$ with $R(\mu\| \mathbb{W})<\infty$ and $\mathbb{W}$-almost every $W$,
     \begin{equation}\label{eq_b_g}
        \int_0^T  |b_t-b_{\lambda,t}|^2(W_t,\mu_t)\,dt \le \int_0^T \int_{(\R^d)^{k}} g_\lambda (t,W_t,y)\, d\mu_t^{\otimes k-1}(y)\,dt.
     \end{equation}
	\end{enumerate}
     Suppose also that $(g_\lambda)_{\lambda>0}$ satisfies for some $K>0$, 
     \begin{equation}\label{eq_pmain_alt2}
         \limsup_{\lambda\to 0} \E \left[ e^{\beta \int_0^T  g_\lambda(t,W_t^{1},\ldots,W_t^{k})\,dt} \right] \leq K, \qquad \forall \beta\in \R,
     \end{equation}
    where $W^1,\dots,W^k$ are independent Brownian motions with common initial law $\rho_0$.

     Then the family $\{Q_{b^N}^N\}$ of laws of $z^N_{\boldsymbol{X}}$ (for $\boldsymbol{X}=(X^{N,1},\ldots,X^{N,N})$ satisfying \eqref{eq_SDEm}) has an LDP with rate function
	\begin{equation}\label{eq:ratefunction-diffusions}
	\cF(\mu)=
	\begin{cases}
	 R(\mu\|\mathbb{W}^{\mu}) & \mbox{if $R(\mu\|\mathbb{W})<\infty$} \\
	 +\infty & \mbox{otherwise},
	\end{cases}
	\end{equation}
	where $\mathbb{W}^{\mu}$ is the law of the process $X^{\mu}_t$ satisfying the SDE
	\begin{equation}\label{eq_SDEc2b}
	d X^{\mu}_t = b_t(X_t^\mu,\mu)  \, d t+d W_t.
	\end{equation}
\end{thm}

\begin{remark}
Note that the zeros of the rate function $\cF$ are exactly the solution to the McKean--Vlasov SDE
\begin{align*}
\left\{
\begin{aligned}
&dX_t = b_t(X_t,\law(X_t))\,dt +dW_t,\\
&X_0 \text{ with law }\rho_0,
\end{aligned}
\right.
\end{align*}
with $R(\law(X_t)\|\mathbb{W})<\infty$. In particular, since $\cF$ has at least one zero, there exists at least one solution to the McKean--Vlasov SDE with finite relative entropy (with respect to $\mathbb{W}$).
\end{remark}
We will see in the proof of Theorem~\ref{thm_pmainlip} that, under the assumptions of the above theorem, $Q^N_{b^N}$ is well-defined by Lemma \ref{lm_prep1}.

\medskip

Moreover, since \eqref{eq_pmain_alt2} is quite general, an application of Khasminskii's lemma provides us with the following sufficient condition: 
\begin{lm}\label{lm_khashapp}
Let $(g_\lambda)_{\lambda>0}$ be a sequence of Borel functions $g_\lambda:[0,T]\times (\R^d)^k\to[0,+\infty)$, $k\in\N$, that satisfies
\begin{equation}\label{eq_pmainlip1}
         \limsup_{\lambda\to 0} \sup_{x_1,\ldots,x_k\in \R^d} \E^{x_1,\ldots,x_k}\left[ \int_0^T  g_\lambda(t,W_t^{1,x_1},\ldots,W_t^{k,x_k})\,dt \right] =0,
     \end{equation}
     where the expectation is over $k$ independent Brownian motions $W^{1,x_1},\ldots,W^{k,x_k}$ starting at points $x_1,\ldots,x_k\in\R^d$ respectively. Then \eqref{eq_pmain_alt2} is satisfied. 
\end{lm}

\begin{proof}
By Khasminskii's lemma (cf.\ Lemma~\ref{lm_estkha}), \eqref{eq_pmainlip1} implies that, for every $\beta\ge0$ and every $0<\alpha<1$,
\begin{align*}
    \sup_{(x_1,\ldots x_k)\in \R^{kd}}\E\left[e^{\beta \int^T_0 g_\lambda (t,W^{1,x_1}_t,\ldots W^{k,x_k}_t)\,dt}\right]\le \frac{1}{1-\alpha},\qquad\text{for all $\lambda>0$ sufficiently small},
\end{align*}
and so, averaging $(x_1,\ldots x_k)$ over $\rho_0^{\otimes k}$ and using Jensen inequality yields
\begin{align*}
    \E\left[e^{\beta \int^T_0 g_\lambda (t,W^1_t,\ldots W^k_t)\,dt}\right] \le \frac{1}{1-\alpha}\qquad\text{for all $\lambda>0$ sufficiently small}.
\end{align*}
Therefore we have, for every $\beta\ge0$,
\begin{equation*}
      \limsup_{\lambda\to 0} \E \left[ e^{\beta \int_0^T  g_\lambda(t,W_t^{1},\ldots,W_t^{k})\,dt} \right] = 1,
\end{equation*}
which gives \eqref{eq_pmain_alt2} with $K=1$.
\end{proof}

\begin{proof}[Proof of Theorem~\ref{thm_pmainlip}]
We would like to apply Theorem \ref{thm_gibbs2} to $\cE^N_{b^N}$, $\cE_{b}$, $\cE^N_{b^N_\lambda}$, $\cE_{b_\lambda}$, with
\[
G^N_\lambda(x^1,\ldots, x^k) = G_\lambda(x^1,\ldots, x^k) = \int_0^T g_\lambda(t,x^1_t,\ldots x^k_t) \,dt,\qquad x^i\in S,\, i=1,\ldots,N.
\]
We claim that the conditions \eqref{eq_unifexpint}, \eqref{eq_Gibbs_ineq_1}, \eqref{eq_Gibbs_ineq_2} and \eqref{eq_exp_int_1}, \eqref{eq_exp_int_2} hold. Then Theorem \ref{thm_gibbs2} and Lemma \ref{lem_LDP_b_Lip} give an LDP for $Q^N_{b^N}$ with rate function $R(\mu\|\mathbb{W})+\cE_b(\mu)$.

We now prove the claims on the above conditions and the form \eqref{eq:ratefunction-diffusions} for the rate function. In particular, we prove that \eqref{eq_b_g_n} and \eqref{eq_b_g} imply \eqref{eq_Gibbs_ineq_1} and \eqref{eq_Gibbs_ineq_2} respectively. Finally, note that \eqref{eq_pmain_alt2} directly implies \eqref{eq_exp_int_1} and \eqref{eq_exp_int_2}, and \eqref{eq_unifexpint} follows as in the proof of Lemma \ref{lem_LDP_b_Lip}.

\medskip
\paragraph{\indent\em Preliminary uniform bounds:}
We call, for $\mu\in\cP(S)$,
\begin{align*}
    \overline{K}_{\lambda,\beta} &:= \limsup_{N\to\infty} \frac{1}{N}\log \E \left[e^{\beta N \langle \int_0^T |b^N_t-b^N_{\lambda,t}(\cdot,z^N_{\boldsymbol{W},t})|^2dt, z^N \rangle}\right],\\
    \overline{H}_{\lambda,\beta}(\mu) &:= \log \E\left[e^{\beta \int_0^T|b-b_{\lambda}(W,\mu)|^2 dt}\right].
\end{align*}
We will show that there exists $\lambda_0>0$ and $\beta_0\gg 1$ arbitrarily large (more precisely, for every $\beta_0>0$ large, there exists $\lambda_0>0$), such that for all $\beta\le \beta_0$:
\begin{align}\label{eq:finite_lambda}
    \begin{aligned}
    \overline{K}_{\lambda_0,\beta} &\le \overline{K}_{\lambda_0,\beta_0} <\infty,\\
    \sup_{\mu,R(\mu\|\mathbb{W})<\infty} \overline{H}_{\lambda_0,\beta}(\mu) -(k-1)R(\mu\|\mathbb{W}) &\le \sup_{\mu,R(\mu\|\mathbb{W})<\infty}\overline{H}_{\lambda_0,\beta_0}(\mu) -(k-1)R(\mu\|\mathbb{W}) < \infty.
    \end{aligned}
\end{align}
We start with the proof for $\overline{K}_{\lambda_0,\beta_0}$. Applying assumption \eqref{eq_b_g_n}, we get
\begin{align*}
\overline{K}_{\lambda_0,\beta_0} &:= \limsup_{N\to\infty} \frac{1}{N}\log \E \left[e^{\beta_0 N \langle \int_0^T |b^N_t-b^N_{\lambda_0,t}(\cdot,z^N_{\boldsymbol{W},t})|^2dt, z^N_{\boldsymbol{W}} \rangle}\right]\\
&\le \limsup_{N\to\infty} \frac{1}{N}\log \E \left[e^{\beta_0 N \int_{(S^k)^\prime} G_{\lambda_0}(x^1,\ldots, x^k)\, d(z^N_{\boldsymbol{W}})^{\otimes k}(x^1,\ldots x^k) }\right].
\end{align*}
Since \eqref{eq_exp_int_1} holds, we can apply Remark \ref{rmk_bd_Gibbs_like}: for any $\beta_0$, the above right-hand side is finite for $\lambda_0>0$ sufficiently small. For $\overline{H}_{\lambda_0,\beta_0}(\mu)$, we have 
\begin{align*}
\overline{H}_{\lambda_0,\beta_0}(\mu) -(k-1)R(\mu\|\mathbb{W})&:= \log \E\left[e^{\beta_0 \int_0^T|b-b_{\lambda_0}(W,\mu)|^2 dt}\right] -(k-1)R(\mu\|\mathbb{W})\\
&\le \log \E\left[ e^{\beta_0 \int_{S^{k-1}} G_{\lambda_0}(W,y) d\mu^{\otimes (k-1)}(y) }\right] -(k-1)R(\mu\|\mathbb{W}).
\end{align*}
Since \eqref{eq_exp_int_2} holds, we can apply Remark \ref{rmk_bd_Gibbs_like}: for any $\beta_0$, the above right-hand side is finite and bounded uniformly over $\mu$ for $\lambda_0>0$ sufficiently small.

As a consequence of \eqref{eq:finite_lambda}, we have (for $N$ large at least)
\begin{subequations}
\begin{align}
&\E\left[e^{\frac12 \int_S V^1_{b^N}(x,z^N_{\boldsymbol{W}})\, dz^N_{\boldsymbol{W}}(x)}\right] <\infty,\\
&\E\left[e^{\frac12 \int_S V^1_b(W,\mu)}\right] <\infty\qquad \forall \mu\text{ with }R(\mu\|\mathbb{W})<\infty.\label{eq_finite_exp}
\end{align}
\end{subequations}
In particular, $V^1_{b^N}(\cdot,z^N)$ and $V^{2,N}_{b^N}(\cdot,z^N)$ are in $L^1(z^N)$ for $\mathbb{P^N}$-a.e.\ $z^N$ and also, for every $\mu$ with $R(\mu\|\mathbb{W})<\infty$, $V^1_b(\cdot,\mu)$ and $V^2_b(\cdot,\mu)$ are in $L^1(\mu)$, see Subsection \ref{s_def_log_densities} for details. Moreover, by Lemma \ref{lm_prep1}, at least for $N$ large, the system \ref{eq_SDEm} admits a weak solution, with unique (under the additional constraint) law has density given by \eqref{eq_Girs_density_particle}. Finally, the inequality \eqref{eq_finite_exp} holds replacing $1/2$ with any $\beta>0$ in the exponential, in particular $\cE_b$ is Borel by Lemma \ref{lm_Borel_cE}.

\medskip
\paragraph{\indent\em Verification of \eqref{eq_Gibbs_ineq_1} and \eqref{eq_Gibbs_ineq_2}:} We fix $\lambda_0$ and $\beta_0$ such that \eqref{eq:finite_lambda}. In view of \eqref{eq_Gibbs_ineq_1}, we show some easy uniform bounds. Using the inequality $|b^N|^2\le 2|b^N-b^N_{\lambda_0}|^2+2|b^N_{\lambda_0}|^2$ and applying H\"older inequality, we get, for every $\ell\ge0$,
\begin{align}
\begin{aligned}\label{eq_exp_bound_bN}
    \limsup_{N\to \infty} \frac{1}{N}\log \E \left[e^{\ell N \langle \int_0^T\langle |b^N_t(\cdot,z^N_{\boldsymbol{W},t})|^2dt, z^N_{\boldsymbol{W}} \rangle}\right]
    &\leq  \limsup_{N\to \infty} \frac{1}{2N}\log \E \left[e^{4\ell N \langle \int_0^T|b^N_t-b^N_{\lambda_0,t}(\cdot,z^N_{\boldsymbol{W},t})|^2dt, z^N_{\boldsymbol{W}} \rangle}\right] \\
    &\hspace{1em}+ \limsup_{N\to \infty} \frac{1}{2N}\log \E \left[e^{4\ell N \langle \int_0^T|b^N_{\lambda_0,t}(\cdot,z^N_{\boldsymbol{W},t})|^2dt, z^N_{\boldsymbol{W}} \rangle}\right]\\ &=:\frac{1}{2}\overline{K}_{\lambda_0,4\ell}+\frac{1}{2}\overline{K}^\prime_{\lambda_0,4\ell} <\infty,
\end{aligned}
\end{align}
where $\overline{K}^\prime_{\lambda_0,4\ell}$ is finite because $\sup_{N\in\N}\|b^N_{\lambda_0}\|_\infty<\infty$. Using now the inequality $|b^N_\lambda|^2\le 2|b^N-b^N_{\lambda}|^2+2|b^N|^2$ and proceeding similarly, we also get, for every $\ell\ge0$,
\begin{align}\label{eq_exp_bound_b}
    \limsup_{N\to \infty} \frac{1}{N}\log \E \left[e^{\ell N \langle \int_0^T|b^N_{\lambda,t}(\cdot,z^N_{\boldsymbol{W},t})|^2dt, z^N_{\boldsymbol{W}} \rangle}\right] \leq \frac{1}{2}\bigl( \overline{K}_{\lambda,4\ell}+\overline{K}_{\lambda_0,4\ell}+\overline{K}^\prime_{\lambda_0,4\ell}\bigr).
\end{align}

To show \eqref{eq_Gibbs_ineq_1}, we write, for $\beta\in \R$,
\begin{align}\label{eq_V_V1_V2}
    \E\left[e^{-\beta N (\cE^N_{b^N}-\cE^N_{b^N_\lambda})(z^N_{\boldsymbol{W}})}\right]
    \leq \E\left[e^{-\beta N \langle (V^1_{b^N}-V^1_{b^N_\lambda})(\cdot,z^N_{\boldsymbol{W}}),z^N_{\boldsymbol{W}} \rangle}\right]^{1/2} \E\left[e^{2\beta N \langle (V^2_{b^N}-V^2_{b^N_\lambda})(\cdot,z^N_{\boldsymbol{W}}),z^N_{\boldsymbol{W}} \rangle}\right]^{1/2}
\end{align}
 and we control the differences $V_{b^N}^1-V_{b^N_\lambda}^1$ and $V_{b^N}^2-V_{b^N_\lambda}^2$ separately. Using the inequality
 \begin{align*}
     \beta(|b^N|^2-|b^N_\lambda|^2)\le |b^N|^2+|b^N_\lambda|^2+\frac{\beta^2}{2}|b^N-b^N_\lambda|^2,
 \end{align*}
 and applying H\"older's inequality and the bounds \eqref{eq_exp_bound_bN} and \eqref{eq_exp_bound_b} (with $\ell=4$), we get
 \begin{align}
 \begin{aligned}\label{eq_V1}
     &\limsup_{N\to\infty}\frac{1}{N}\log \E[e^{\beta N \langle (V^1_{b^N}-V^1_{b^N_\lambda})(\cdot,z^N_{\boldsymbol{W}}),z^N_{\boldsymbol{W}} \rangle}]\\
     &\hspace{2em}\leq \frac{1}{4}\overline{K}_{\lambda,16} +\frac{1}{2}\bigl(\overline{K}_{\lambda_0,16}+\overline{K}^\prime_{\lambda_0,16}\bigr) +\limsup_{N\to \infty} \frac{1}{2N}\log \E [e^{\beta^2 N \langle \int_0^T|b^N_t-b^N_{\lambda,t}(\cdot,z^N_{\boldsymbol{W},t})|^2dt, z^N_{\boldsymbol{W}} \rangle}]\\
     &\hspace{2em}= \frac{1}{4}\overline{K}_{\lambda,16} +\frac{1}{2}\bigl(\overline{K}_{\lambda_0,16}+\overline{K}^\prime_{\lambda_0,16}\bigr) + \frac{1}{2}\overline{K}_{\lambda,\beta^2}.
 \end{aligned}
 \end{align}
 For $V_{b^N}^2-V_{b^N_\lambda}^2$, we use Lemma~\ref{lm_estnk} to obtain
\begin{align*}
    \E\left[e^{2\beta N\langle  V_{b^N}^2(\cdot,z^N_{\boldsymbol{W}})-V_{b^N_\lambda}^2(\cdot,z^N_{\boldsymbol{W}}),z^N_{\boldsymbol{W}}\rangle} \right] \le \E\left[ e^{4\beta^2 \sum_{i=1}^N \int_0^T |b^N_t - b^N_{\lambda,t}|^2(W_t^i, z^N_{\boldsymbol{W},t})\, d t } \right]^{\frac{1}{2}},
\end{align*}
and therefore
\begin{align}\label{eq_V2}
    \limsup_{N\to \infty}\frac{1}{N}\log E^N\left[e^{4\beta N\langle  V_{b^N}^2(\cdot,z^N_{\boldsymbol{W}})-V_{b^N_\lambda}^2(\cdot,z^N_{\boldsymbol{W}}),z^N_{\boldsymbol{W}}\rangle} \right]
    \leq \frac{1}{2}\overline{K}_{\lambda,4\beta}.
\end{align}
Putting together the inequalities \eqref{eq_V_V1_V2}, \eqref{eq_V1} and \eqref{eq_V2}, we get, for some constant $c>0$ (independent of $\beta$) and some $c_\beta>0$,
\begin{align}\label{eq_V_particle}
    \limsup_{N\to\infty}\frac{1}{N}\log \E\left[e^{-\beta N (\cE^N_{b^N}-\cE^N_{b^N_\lambda})(z^N_{\boldsymbol{W}})}\right]
    \leq \bar{c} + c\overline{K}_{\lambda,c_\beta},
\end{align}
with $\bar{c}=(\overline{K}_{\lambda_0,16}+\overline{K}^\prime_{\lambda_0,16})/2\ge 0$. The assumption \eqref{eq_b_g_n} gives, for some new $c_\beta>0$,
\begin{align*}
    &\limsup_{N\to\infty}\frac{1}{N}\log \E\left[e^{-\beta N (\cE^N_{b^N}-\cE^N_{b^N_\lambda})(z^N_{\boldsymbol{W}})}\right]\\
    &\hspace{4em}\leq \bar{c} +c\limsup_{N\to\infty}\frac{1}{N}\log \E\left[\exp\left(c_\beta N \int^T_0\int_{((\R^d)^k)'} g_\lambda(t,x_1,\ldots x_k) \,d(z^N_{\boldsymbol{W},t})^{\otimes k} \,dt\right)\right]\\
    &\hspace{4em}\leq \bar{c} +c\limsup_{N\to\infty}\frac{1}{N}\log \E\left[\exp\left(c_\beta N \int_{(S^k)'} \int^T_0 g_\lambda(t,x^1_t,\ldots x^k_t) \,dt \,d(z^N_{\boldsymbol{W}})^{\otimes k}\right)\right] 
\end{align*}
and \eqref{eq_Gibbs_ineq_1} follows.

As for \eqref{eq_Gibbs_ineq_2}, we use Lemma~\ref{lm_entvar} to obtain, for every $\mu$ with $R(\mu\|\mathbb{W})<\infty$,
\begin{align*}
    \beta |\cE_b(\mu) - \cE_{b_\lambda}(\mu)| \le R(\mu\|\mathbb{W}) + \log \E \left[ e^{\beta |V_b(W,\mu)-V_{b_\lambda}(W,\mu)|} \right].
\end{align*}
Using the same arguments as before, we get, for some $c'>0$ (independent of $\beta$) and some $c_\beta>0$,
\begin{align}
    \beta |\cE_b(\mu) - \cE_{b_\lambda}(\mu)|
    \leq \frac{k+1}{2} R(\mu\|\mu_0) +\bar{c} +c' \overline{H}_{\lambda,c_\beta},\label{eq_V_entropy}
\end{align}
where $\bar{c}>0$ is such that $\bar{c}\ge (\overline{H}_{\lambda_0,16}(\mu)-(k-1)R(\mu\|\mu_0)+\overline{H}'_{\lambda_0,16})/2$ for every $\mu$ with $R(\mu\|\mathbb{W})<\infty$, and $\overline{H}'_{\lambda_0,\beta} = \log \E[e^{\beta \int^T_0 |b_{\lambda_0}(W,\mu)|^2 dt}]$. The assumption \eqref{eq_b_g} gives, for some new $c_\beta$,
\begin{align*}
    \beta |\cE_{b^N}(\mu) - \cE_{b^N_\lambda}(\mu)|
    \leq \frac{k+1}{2}R(\mu\|\mu_0) +\bar{c} +c'\log \E\left[e^{c_\beta \int_{S^k} \int^T_0 g_\lambda(t,W_t,x^2_t\ldots x^k_t) \,dt \,d\mu^{\otimes (k-1)}}\right]
\end{align*}
and \eqref{eq_Gibbs_ineq_2} follows.

\medskip
\paragraph{\indent\em Proof of \eqref{eq:ratefunction-diffusions}:}
By \eqref{eq_finite_exp}, $b$ satisfies the assumption of Lemma \ref{lm_prep2}, which then implies the representation formula \eqref{eq:ratefunction-diffusions} for the rate function. The proof is complete.
\end{proof}

The above proof shows that we can relax some of the assumptions, as we show below. 
\begin{prop}\label{cory:exp_moments}
The results of Theorem \ref{thm_pmainlip} (namely the LDP for $Q^N_{b^N}$ with rate function $\cF$) remain valid if any of the following statements hold:
 \begin{enumerate}[label=(\alph*)]
     \item Instead of $\FLip$-inducing drifts $(b^N_\lambda,b_{\lambda})_{\lambda>0}$, 
     we assume that for every $\lambda>0$ the family $Q^N_{b_{\lambda}^N}$ has an LDP with rate function $\cF_{b_\lambda}$ (defined similarly to $\cF$ via \eqref{eq:ratefunction-diffusions}, \eqref{eq_SDEc2b}), and
      for every $\beta \in \R$, 
\begin{equation}\label{eq_pmain_malt}
\begin{aligned}
         \limsup_{N\to\infty} \frac{1}{N}\log \E \left[ \exp\left(N\beta\int_0^T \bigl\langle |b^N_{\lambda,t}|^2(\cdot,z^N_{\boldsymbol{W},t}),z^N_{\boldsymbol{W},t} \bigr\rangle\right)dt \right] <\infty,\\
         \sup_{\mu,R(\mu\|\mathbb{W})<\infty} \log \E\left[ \exp\left(\beta\int_0^T |b_{\lambda,t}|^2(W_t,\mu_t)\,dt\right) \right] -R(\mu\|\mathbb{W}) <\infty. 
\end{aligned}
\end{equation}
\item Instead of the existence of $g_{\lambda}$ and \eqref{eq_b_g_n}, \eqref{eq_b_g}, we assume there exists a constant $K\in \R$ such that for every $\beta \in \R$,
    \begin{equation}\label{eq_pmain_alt1}
     \begin{aligned}
         \limsup_{\lambda\to 0}\limsup_{N\to\infty} \frac{1}{N}\log \E \left[ \exp\left(N\beta\int_0^T \bigl\langle |b^N_t-b^N_{\lambda,t}|^2(\cdot,z^N_{\boldsymbol{W},t}),z^N_{\boldsymbol{W},t} \bigr\rangle\right)dt \right] \le K,\\
         \limsup_{\lambda\to 0} \sup_{\mu,R(\mu\|\mathbb{W})<\infty} \log \E\left[ \exp\left(\beta\int_0^T |b_t-b_{\lambda,t}|^2(W_t,\mu_t)\,dt\right) \right] -R(\mu\|\mathbb{W}) \le K. 
     \end{aligned}
     \end{equation}
 \end{enumerate}   
\end{prop}

\begin{proof}
The inequalities \eqref{eq_pmain_malt} imply via Lemmas~\ref{lm_prep1}, \ref{lm_estnk} and \ref{lm_prep2} both the representations of $Q^N_{b_{\lambda}^N}$ and $\cF_{b_\lambda}$ in terms of $\cE^N_{b^N_{\lambda}},\cE_{b_{\lambda}}$, and the estimates \eqref{eq_unifexpint}. Moreover, \eqref{eq_pmain_alt1} combined with \eqref{eq_V_particle} and \eqref{eq_V_entropy} imply conditions \eqref{rconv5a} and \eqref{rconv5b}. Hence we can use directly Theorem \ref{thm_direct} to deduce the LDP (the representation formula \eqref{eq:ratefunction-diffusions} follows again from Lemma \ref{lm_prep2}).
\end{proof}

\subsection{Applications to concrete examples}\label{s_examples} In this subsection, we consider common form of drifts $b^N$ that appear in applications.

\subsubsection{Example: $2$-point interaction}\label{s_ex1}

We start with the example of $2$-point, translation-invariant interaction; while this is a particular case of the $k$-point interaction, we discuss this case separately, to highlight the essential ingredients of the result. In this example, we consider a simple class of drifts, that commonly appear in various fields of application, namely
\begin{equation}\label{eq:drift_2point}
    b_t^N(x_i,z^N_\x) = \frac{1}{N}\sum_{j\ne i}\varphi_t(x_i-x_j),\qquad z^N_\x = \frac{1}{N}\sum_{j=1}^N\delta_{x_j}
\end{equation}
for some Borel map $\varphi:[0,T]\times \R^{d}\to\R^d$.

\begin{remark}\label{rem_def_drift_2pt}
Strictly speaking, \eqref{eq:drift_2point} is not a good definition, because the right-hand side depends on $i$ and not just on $x_i$ and $z^N_\x$. However we can give a rigorous definition with a harmless change of \eqref{eq:drift_2point}. Precisely, we can define
\begin{align}\label{eq:drift_2point_rig}
    b_t^N(x,\mu) = \int_{\R^d\setminus \{x\}} \varphi_t(x-y)\,d \mu_t(y).
\end{align}
Indeed, note that (recall $\boldsymbol{W}=(W^1,\ldots,W^N)$ is a $N$-tuple of $d$-dimensional independent Brownian motions with law $\tilde P^N=\mathbb{W}^{\otimes N}$)
\begin{align}
    \frac{1}{N}\sum_{j\ne i}\varphi_t(W^i_t-W^j_t) = \int_{\R^d\setminus \{x\}} \varphi_t(x-y)\,d z^N_{\boldsymbol{W},t}(y), \quad \text{for a.e.\ $t\in[0,T]$ and $\tilde P^N$-a.e.\ $\boldsymbol{W}$.}\label{eq:def_drift_2pt}
\end{align}
because the set $\{t\in[0,T]: W^i_t = W^j_t, j\ne i\}$ has Lebesgue measure zero for $\tilde P^N$-a.e.\ $\boldsymbol{W}=(W^1,\ldots,W^N)$. Therefore, if Girsanov theorem can be applied to the particle system \eqref{eq_SDEm}, as it is the case in the next result, then \eqref{eq:def_drift_2pt} holds also replacing $\boldsymbol{W}$ and $\tilde P^N$ respectively with $\boldsymbol{X}^N=(X^{N,1},\ldots X^{N,N})$, the solution to the particle system \eqref{eq_SDEm} with drift \eqref{eq:drift_2point_rig}, and its law $\tilde Q^N$. Hence such solution $\boldsymbol{X}^N$ solves also the original particle system with drift \eqref{eq:drift_2point}.
\end{remark}

We introduce also the drift of the associated McKean--Vlasov SDE, namely
\begin{align*}
    b_t(x,\mu) = \int_{\R^d} \varphi_t(x-y)\,d \mu_t(y).
\end{align*}

\begin{prop}\label{lm_plpq}
	Assume the condition \eqref{eq_IC_exp_int} on the initial law $\rho_0$. Suppose that
	\begin{equation}\label{eq_expphitr}
       \E \left[ e^{\beta \int_{0}^T |\varphi_t|^2(W_t^{1}-W_t^{2})\,d t} \right] < \infty, \qquad \forall \beta\in \R,
        \end{equation}
where $W^1,W^2$ are independent Brownian motions with common initial law $\rho_0$.

Then the family $\{Q_b^N\}$ of laws of the empirical processes associated to the interacting system \eqref{eq_SDEm} with drifts of the form \eqref{eq:drift_2point} satisfies an LDP with rate function $\cF$ given in \eqref{eq:ratefunction-diffusions}.

\smallskip

In particular, \eqref{eq_expphitr} holds whenever $\varphi$ is in $L_t^{q}(L^{p}_{x-y})$, for $p,q$ satisfying
	\begin{align}
	    2\le p,q\leq \infty,\quad \frac{d}{p} + \frac{2}{q} <1.\label{eq_KR_cond}
	\end{align}

\end{prop}

Proposition \ref{lm_plpq} is a special case of Proposition \ref{lm_plp12q} with $k=2$, $p_1=p,p_2=+\infty$, and hence we will postpone the proof. 

\begin{remark}
The condition $\varphi\in L^q_t(L^p_{x-y})$ with $p,q$ satisfying \eqref{eq_KR_cond} is well-known in the literature on the so-called regularization by noise phenomenon (where an ill-posed ODE or PDE gains well-posedness by addition of a suitable noise). Indeed, a $d$-dimensional SDE, with additive noise and drift in $L^q_t(L^p_x)$, with $p,q$ satisfying \eqref{eq_KR_cond}, has the strong existence and uniqueness property, see for example \cite{KryRoc2005,FedFla2011} among many others; on the contrary, if $p,q$ do not satisfy \eqref{eq_KR_cond} not even with equality, there exist counterexamples to well-posedness for SDEs, even in the weak sense, see e.g. \cite[Section 7]{BFGM2019}.

For this reason the exponents of \eqref{eq_KR_cond} are likely optimal for irregular drifts $\varphi$ in our example: we expect that, for $L^q_t(L^p_{x-y})$ drifts without condition \eqref{eq_KR_cond}, even the 2-particle system is not well-posed. However, there are likely drifts that satisfies \eqref{eq_expphitr} but are not in a $L^q_t(L^p_x)$ class with \eqref{eq_KR_cond}.

\end{remark}

\begin{remark}\label{rem_psubz}
A relevant example of function $\varphi$ verifying condition \eqref{eq_KR_cond} is
\begin{align*}
    &\varphi(t,z)=  |z|^{\alpha} g\left(\frac{z}{|z|}\right)1_{|z|\le R} +h(z)1_{|z|>R},
\end{align*}
with $g:\mathbb{S}^{d-1}\rightarrow \mathbb{S}^{d-1}$ and $h:\R^d\rightarrow\R^d$ both Borel bounded, $R>0$, and exponent $\alpha$ satisfying
\begin{align*}
    \alpha>-1 \;\text{ for }\; d\ge 2,\quad \text{ and }\quad
\alpha>-1/2 \;\text{ for }\; d=1.
\end{align*}
	
In particular, when $d\geq 2$ this includes the case 
\[
\varphi(x) = -\nabla \Phi(x),\qquad \Phi(t,x)=|x|^{\alpha}, \qquad \alpha>0,
\]
in some neighborhood around $x=0$, and with $\varphi$ bounded outside of this neighborhood. 
However, $\Phi(x)=\log |x|$ does not fall in this class. Moreover, as shown in \cite{Hariya2014}, the exponential moment estimate of \eqref{eq_expphitr} actually blows up when $\beta$ is large enough. 

The latter is also related to the fact that for $d\geq 2$ and a logarithmic potential $\Phi(x)=\log |x|$, the law of SDE \eqref{eq_SDEm} might no longer be absolutely continuous with respect to the law of non-interacting Brownian motions---it is possible in the case of a sufficiently strong attractive force that the particles hit each other (see \cite{Fournier2015}). Interestingly enough, as is done in \cite{Fournier2015}, it can still be shown that the particle system converges to the corresponding McKean--Vlasov equation. However, whether in this case a large deviation principle still exists is not known. 
\end{remark}

\subsubsection{Example: $k$-point interaction}\label{s_ex2}

We can also treat the case of a $k$-interaction drift, namely
\begin{equation}\label{eq:drift_kpoint}
    b_t^N(x_i,z^N_\x) = \frac{1}{N^{k-1}}\sum_{j_1,\ldots, j_{k-1}\ne i\text{ all distinct}}\varphi_t(x_i,x_{j_1},\ldots x_{j_{k-1}})
\end{equation}
for some Borel function $\varphi:[0,T]\times \R^{kd}\rightarrow \R^d$. As in the previous example, the rigorous definition of $b^N$ can be given as in Remark \ref{rem_def_drift_2pt}. Similarly we take the drift of the associated McKean--Vlasov SDE
\begin{equation}\label{eq:drift_kpoint2}
    b_t(x,\mu) = \int_{\R^{(k-1)d}} \varphi_t(x,y_1,\ldots y_{k-1})\,d \mu_t^{\otimes (k-1)}(y_1,\ldots y_{k-1}).
\end{equation}

\begin{prop}\label{lm_plp12q}
    Assume the condition \eqref{eq_IC_exp_int} on the initial law $\rho_0$. Suppose that 
     \begin{equation}\label{eq_expphi}
       \E \left[ e^{\beta \int_{0}^T |\varphi_t|^2(W_t^{1},\ldots,W_t^{k})\,d t} \right] < \infty, \qquad \forall \beta\in \R
        \end{equation}
    where $W^1,\dots,W^k$ are independent Brownian motions with common initial law $\rho_0$.
    
    Then the family $\{Q_b^N\}$ of laws of the empirical processes associated to the interacting system \eqref{eq_SDEm} with drifts of the form \eqref{eq:drift_kpoint} satisfies an LDP with rate function $\cF$ given in \eqref{eq:ratefunction-diffusions}.
    
    In particular, \eqref{eq_expphi} holds whenever
    \begin{equation}\label{eq_Lpqcond}
        \varphi\in L^{q}_t(L^{p_1}_{x_1}(\ldots (L^{p_k}_{x_k})\ldots )),
    \end{equation}
    with $p_1,\ldots p_k,q$ satisfying
	\begin{equation}\label{eq_Lpqcond_pq}
	p_1,\ldots p_k,q\in[2,\infty],\quad \frac{d}{p_1}+\ldots +\frac{d}{p_k}+\frac{2}{q}<1.
	\end{equation}
\end{prop}

\begin{remark}
Similarly to Lemma~\ref{lm_estlp12q}, \eqref{eq_Lpqcond} can be replaced by 
\[\varphi\in L^{q}_t(L^{p_1}_{x_{\sigma(1)}}(\ldots (L^{p_k}_{x_{\sigma(k)}})\ldots ))\]
    for some permutation $\sigma$ of $\{1,\ldots k\}$.
\end{remark} 

\begin{remark}\label{rem_psum}
The space of $\varphi$ satisfying condition \eqref{eq_expphi} is a vector space (as easily checked). Hence, condition \eqref{eq_expphi} also holds in the more general case of $\varphi=\sum_{j=1}^m \varphi_j$ for some $m$, where, for any $j=1,\ldots m$,
\begin{align*}
    \varphi_j\in L^{q^{(j)}}_t(L^{p^{(j)}_1}_{x_1}(\ldots (L^{p^{(j)}_k}_{x_k})\ldots ))
\end{align*}
with $p^{(j)}_1,\ldots p^{(j)}_k,q^{(j)}$ satisfying \eqref{eq_Lpqcond_pq}. In particular, we can allow $\varphi$ to be a sum of a bounded function and a function satisfying \eqref{eq_Lpqcond} and \eqref{eq_Lpqcond_pq} as in Theorem~\ref{eq:intro_SDE}.
\end{remark}

\begin{proof}[Proof of Proposition~\ref{lm_plp12q}]
We will use three different approximations: 
\begin{enumerate}
    \item First, we consider $b$ with $\varphi$ Lipschitz bounded and show that $(b^N,b)$ are $\FLip$-inducing, by bounding away the self-interactions. 
    \item Then we consider the case $b$ with $\varphi$ Borel bounded. We apply Theorem \ref{thm_pmainlip} to get the desired LDP.
    \item Finally, we extend this to $\varphi$ satisfying \eqref{eq_expphi} using truncation of $\varphi$ and the approximation given in step (2). 
\end{enumerate} 

The fact that \eqref{eq_expphi} holds under conditions \eqref{eq_Lpqcond} and \eqref{eq_Lpqcond_pq} follows from Lemma \ref{lm_kha_pq}.

\medskip

\textbf{(1)} Assume that $\varphi: [0,T]\times\R^{k d}\to \R^{d}$ is Borel bounded and that the map $x\mapsto \varphi(t,x)$ is globally Lipschitz continuous for every $t\in[0,T]$, with Lipschitz constant $\Lip(\varphi)$ independent of $t$. We will show that $(b^N,b)$ is $\FLip$-inducing. 

It is well known that $b\in \FLip$: indeed, for any $x_1,x_2 \in \R^d$ and $\mu_1,\mu_2\in \cP(\R^d)$, we have
    \begin{align*}
       &|b_{t}^N(x_1,\mu_1) - b_{t}^N(x_2,\mu_2)| \\
       &\hspace{3em}\le \int_{\R^{(k-1)d}} \left|\varphi_t(x_1,y_1,\ldots y_{k-1}) - \varphi_t(x_2,y_1',\ldots y_{k-1}')\right|\,d\gamma(y_1,y_1')\cdot \ldots \cdot d\gamma(y_{k-1},y_{k-1}') \\
       &\hspace{3em}\le \Lip(\varphi)\left( |x_1-x_2| + \int_{\R^{(k-1)d}} \left(\sum_{i=1}^{k-1}|y_i-y_i'|\right)\,d\gamma(y_1,y_1')\cdot \ldots \cdot d\gamma(y_{k-1},y_{k-1}')\right)\\
       &\hspace{3em}=\Lip(\varphi)\left( |x_1-x_2| + (k-1) \int_{\R^{(k-1)d}} |y-y'|\,d\gamma(y,y')\right),
    \end{align*}
    for any coupling $\gamma$ between $\mu$ and $\nu$; optimizing over all couplings then yields the required Lipschitz estimate. Moreover, clearly $|b^N|\leq |b|$ and hence  $\sup_{N} \|b^N\|_{\infty}<\infty$. Finally, to show \eqref{eq_FLipver2},  note that, similarly to Lemma \ref{lmg_uapprox}, we have for every $N$ and $P^N$-almost every $z^N$ 
     \begin{equation*}
     \begin{aligned}
         &\int_0^T \bigl\langle |b_t-b_{t}^N|^2(\cdot,z^N_t),z^N_t \bigr\rangle \, d t \\
         &\leq \frac{2}{N} \int_0^T \sum_{i=1}^N \left(\, \frac{1}{N^{k-1}}\left|\left(\sum_{j_1,\ldots, j_{k-1}\ne i\text{ all distinct}} -\sum_{j_1,\ldots, j_{k-1}\neq i}\right) \varphi_t(W^i_t,W^{j_1}_t,\ldots W^{j_{k-1}}_t) \, \right| \right)^2 d t\, \\
         &\hspace{1em} +\frac{2}{N} \int_0^T \sum_{i=1}^N \left(\, \frac{1}{N^{k-1}}\left| \left(\sum_{j_1,\ldots, j_{k-1}\ne i}-\sum_{j_1,\ldots, j_{k-1}}\right) \varphi_t(W^i_t,W^{j_1}_t,\ldots W^{j_{k-1}}_t)
         \right| \, \right)^2 d t\\
         &\leq 2T \|\varphi\|^2_{\infty}  \frac{1}{N^{2(k-1)}} \left(\left| \frac{(k-1)(k-2)}{2}N^{k-2}\right|^2+\left|(k-1)N^{k-2}\ \right|^2 \, \right) \leq \frac{C}{N^2},
    \end{aligned}
     \end{equation*}
for some generic constant $C$. Hence $(b^N,b)$ is $\FLip$-inducing.

\textbf{(2)} Next, suppose that $\varphi:[0,T]\times\R^{kd}\rightarrow \R^d$ is Borel bounded. We can find a sequence of Lipschitz approximations $(\varphi_{\lambda})_{\lambda>0}$, with $\varphi_{\lambda}$ as in \textbf{(1)} for each $\lambda>0$, such that $\varphi_\lambda\rightarrow \varphi$ Lebesgue-a.e. as $\lambda\to0$ and $\|\varphi_\lambda\|_\infty\le \|\varphi\|_{\infty}$ (for example, we can take $\varphi_\lambda$ as convolutions of $\varphi$ with approximations of identity). We take $b_\lambda^N$, $b_\lambda$ the corresponding drifts for the particle system and the McKean-Vlasov SDE (as respectively in \eqref{eq:drift_kpoint}, \eqref{eq:drift_kpoint2} with $\varphi_\lambda$ in place of $\varphi$). In order to apply Theorem \ref{thm_pmainlip}, with $(b^N_\lambda,b_\lambda)$ as sequence of $\FLip$-inducing drifts, we verify now the assumptions \eqref{eq_b_g_n}, \eqref{eq_b_g} and \eqref{eq_pmain_alt2}. The estimates \eqref{eq_b_g_n} and \eqref{eq_b_g} are easily obtained with
    \[
        g_\lambda(t,x_1,\ldots,x_{k}) = \left|\varphi(t,x_1,\ldots,x_{k}) - \varphi_{\lambda}(t,x_1,\ldots,x_{k})\right|^2.
    \]
  Concerning \eqref{eq_pmain_alt2}, we start noting that
  \begin{align*}
      g_\lambda(t,W^1,\ldots W^k_t)\rightarrow 0\quad \text{ and }\quad g_\lambda(t,W^1,\ldots W^k_t) \le 4\|\varphi\|_{\infty}^2,  \quad \text{for $\mathbb{P}\otimes dt$-a.e. $(\omega,t)$}.
  \end{align*}
  Hence, for each $\beta>0$, we apply dominated convergence theorem twice, i.e.\ to the time integral and then to the expectation, thereby obtaining
  \begin{align}
      \E \left[ e^{\beta \int_0^T  g_\lambda(t,W_t^{1},\ldots,W_t^{k})\,dt} \right] \rightarrow 1\quad\text{ as }\quad\lambda\to 0,\label{eq_exp_approx_bounded}
  \end{align}
  that is \eqref{eq_pmain_alt2}. Hence we can apply Theorem \ref{thm_pmainlip} and obtain the desired LDP.
  
\textbf{(3)} Finally, assume that $\varphi$ satisfies \eqref{eq_expphi}. We take
\begin{align*}
\tilde{\varphi}_\lambda = (\varphi\wedge 1/\lambda) \vee (-1/\lambda).
\end{align*}
We take also an increasing sequence $(\beta_\lambda)_\lambda$ with $\beta_\lambda \to \infty$. For each $\lambda>0$ fixed, applying \eqref{eq_exp_approx_bounded} to $\tilde{\varphi}_\lambda$ in place of $\varphi$, we get the existence of a Lipschitz function $\varphi_\lambda$ as in \textbf{(1)}, such that
\begin{align}
    \E \left[ e^{\beta_\lambda \int_0^T  |\tilde{\varphi}_\lambda - \varphi_\lambda|^2(t,W_t^{1},\ldots,W_t^{k})\,dt} \right]  <1+\lambda.\label{eq_exp_approx_diag_arg}
\end{align}
Now we take $b_\lambda^N$, $b_\lambda$ the drifts for the particle system and the McKean-Vlasov SDE, as respectively in \eqref{eq:drift_kpoint}, \eqref{eq:drift_kpoint2} with $\varphi_\lambda$ in place of $\varphi$. As before, in order to apply Theorem \ref{thm_pmainlip}, with $(b^N_\lambda,b_\lambda)$ as sequence of $\FLip$-inducing drifts, we verify the assumptions \eqref{eq_b_g_n}, \eqref{eq_b_g} and \eqref{eq_pmain_alt2}. As before, the estimates \eqref{eq_b_g_n} and \eqref{eq_b_g} are easily obtained with
    \[
        g_\lambda(t,x_1,\ldots,x_{k}) = \left|\varphi(t,x_1,\ldots,x_{k}) - \varphi_{\lambda}(t,x_1,\ldots,x_{k})\right|^2.
    \]
Concerning \eqref{eq_pmain_alt2}, we split $g_\lambda \le 2|\varphi-\tilde{\varphi}_\lambda|^2 + 2|\tilde{\varphi}_\lambda-\varphi_\lambda|^2$ and study the two terms separately. For the first term, for every $\beta>0$, we have by dominated convergence theorem
\begin{align*}
    \limsup_{\lambda\to 0}\E \left[ e^{\beta \int_0^T  |\varphi-\tilde{\varphi}_\lambda|^2(t,W_t^{1},\ldots,W_t^{k})\,dt} \right] =1.
\end{align*}
For the second term, the bound \eqref{eq_exp_approx_diag_arg} implies, for every $\beta>0$,
\begin{align*}
    \limsup_{\lambda\to 0}\E \left[ e^{\beta \int_0^T  |\tilde{\varphi}_\lambda-\varphi_\lambda|^2(t,W_t^{1},\ldots,W_t^{k})\,dt} \right] \le
    \limsup_{\lambda\to 0} \E \left[ e^{\beta_\lambda \int_0^T  |\tilde{\varphi}_\lambda-\varphi_\lambda|^2(t,W_t^{1},\ldots,W_t^{k})\,dt} \right]  =1.
\end{align*}
The two bounds above give \eqref{eq_pmain_alt2}. Hence we can apply Theorem \ref{thm_pmainlip} and obtain the desired LDP. The proof is complete.
\end{proof}

\subsubsection{Example: Measure dependent drift}\label{s_ex3}

As an example of a more general interaction, we consider drifts of the form (cf.\ \cite{GPP19})
\begin{equation}\label{eq:drift_general}
    b_t^N(x_i,z^N_\x) = \frac{1}{N}\sum_{j\ne i}\Psi\left(x_i,x_j,\frac{1}{N}\sum_{\ell\ne i}\varphi_t(x_i,x_\ell),\frac{1}{N}\sum_{\ell\ne j}\varphi_t(x_j,x_\ell),(z^N_\x)_t \right),
\end{equation}
where $\Psi:\R^{4d}\times \cP(\R^d)\to \R^d$ and $\varphi:[0,T]\times \R^{2d}\to\R^d$ are Borel maps. As in the previous examples, the rigorous definition of $b^N$ can be given as in Remark \ref{rem_def_drift_2pt}. Similarly we take the drift of the associated McKean--Vlasov SDE
\begin{align*}
    b_t(x,\mu)=\int_{\R^d}\Psi \left(x,y, \int_{\R^d}\varphi_t(x,z)\,d\mu_t(z),\int_{\R^d}\varphi_t(y,z)\,d\mu_t(z),\mu_t \right) d\mu_t(y).
\end{align*}
Throughout we will use both the 1-Wasserstein metric, $W_1$ and the bounded Lipschitz metric, $d_{BL}$ on $\cP(\R^d)$, with the latter given by 
\begin{equation*}
d_{BL}(\mu,\nu):= \sup_{\phi:\|\phi\|_{\infty}\leq 1,\Lip(\phi)\leq 1} \left\{ \int_{\R^d} \phi(x) d\mu(x)-\int_{\R^d} \phi(x) d\nu(x)\right\}.
\end{equation*}
Note that $d_{BL}\leq W_1$.

\begin{prop}\label{thm_pmark1}
	Assume that the initial law $\rho_0$ satisfies condition \eqref{eq_IC_exp_int}. Furthermore, let $\Psi\in \mathrm{Lip}(\R^{4d}\times (\cP(\R^d),d_{BL}))$, and suppose that there exists a constant $L$ such that
	\begin{equation}\label{eq_MD_Lgrowth2}
\Psi \left(x,y, a,b,\mu \right) \leq L(1+|a|+|b|), \qquad x,y,a,b\in \R^d, \mu\in \cP(\R^d)
	\end{equation}
	and 	
	 \begin{equation}\label{eq_expphi2}
       \E \left[ e^{\beta \int_{0}^T |\varphi_t|^2(W_t^{1},W_t^{2})\,d t} \right] < \infty,  \qquad \text{for all $\beta\in \R$},
        \end{equation}
	where $W^1,W^2$ are independent Brownian motions with common initial law $\rho_0$.
	
	Then the family $Q_b^N$ of laws of the empirical processes associated to the interacting system \eqref{eq_SDEm} induced by drifts of the form \eqref{eq:drift_general} satisfies an LDP with rate function $\cF$ given in \eqref{eq:ratefunction-diffusions}.
	
	\smallskip
	
	In particular, \eqref{eq_expphi2} holds whenever $\varphi$ is translation invariant with $\varphi\in L_t^q(L_{x-y}^p)$ for $p,q\in[2,\infty]$ satisfying
	\begin{equation}\label{eq_Lpqcond_pq2}
	    \frac{d}{p} + \frac{2}{q} <1.
	\end{equation}
\end{prop}

\begin{remark}
A couple of comments:
    \begin{enumerate}
        \item Similar to Remark \ref{rem_psum}, Proposition \ref{thm_pmark1} holds for any $\varphi=\sum_{j=1}^m \varphi_j$ with each $\varphi_j$ satisfying \eqref{eq_Lpqcond_pq2} for suitable exponents $p^{(j)},q^{(j)}$. 
        \item As in Remark~\ref{rem:selfinteraction}, when $\varphi$ is bounded, we can include self-interactions in the summations in \eqref{eq:drift_2point}, \eqref{eq:drift_kpoint} and \eqref{eq:drift_general} without changing the results above.
    \end{enumerate}
\end{remark}

\begin{remark}\label{rem:intro_sde}
The example of Theorem \ref{thm:intro_sde} is a particular case of \eqref{eq:drift_general}, which can be seen by setting
   \begin{equation*}
       \Psi(x,x_2,y,x_3,\mu):=\psi(x,\mu,y).
   \end{equation*} 
The LDP for $Q_b^N$ follows from Proposition \ref{thm_pmark1}. Moreover, the convergence to the McKean-Vlasov SDE will be shown in the next section, and follows directly from Proposition \ref{cor_poc_ex}.
\end{remark}

\begin{proof}[Proof of Proposition~\ref{thm_pmark1}]
    The proof is an adaptation of the proof for Proposition~\ref{lm_plp12q}. As in step \textbf{(1)} of the proof of Lemma~\ref{lm_plp12q}, we first consider suitable bounded Lipschitz functions $\varphi$ with corresponding drifts $b$. Also in this case, it is not difficult to see that in this case $(b^N,b)$ is $\FLip$-inducing. First, to show that $b\in \FLip$, note that $b$ is bounded by \eqref{eq_MD_Lgrowth2} and the boundedness of $\varphi$ yields
    \begin{align*}
        | b_{t}(x_1,\mu_1)- b_{t}(x_2,\mu_2)| \le \Lip(\Psi)\iint_{\R^d\times\R^d} \Bigl(|x_1 - x_2| + |y_1-y_2| + \mathrm{(I)} + \mathrm{(II)} \Bigr)\, d\gamma(y_1,y_2)+\mathrm{(III)}
    \end{align*}
    where $\mathrm{(III)}=\Lip(\Psi)\, d_{BL}(\mu,\nu)$ and
    \begin{align*}
        \mathrm{(I)}
        &\le \Lip(\varphi)\iint_{\R^\times\R^d} \Bigl(|x_1-x_2| + |w_1-w_2|\Bigr)\,d\gamma(w_1,w_2), \\
        \mathrm{(II)} 
        &\le \Lip(\varphi)\iint_{\R^\times\R^d} \Bigl(|y_1-y_2| + |w_1-w_2|\Bigr)\,d\gamma(w_1,w_2).
    \end{align*}
Noting  $d_{BL}\leq W_1$, inserting estimates (I) and (II) into the previous inequality and optimizing over all couplings $\gamma$ between $\mu_1$ and $\mu_2$ yields
\[
    | b_{t}(x_1,\mu_1)- b_{t}(x_2,\mu_2)| \le \Lip(\Psi)(2+3\Lip(\varphi)) \bigl( |x_1-x_2| +  W_1(\mu_1,\mu_2) \bigr),
\]
i.e., the Lipschitz estimates holds. Next, note that we have for $P^N$-almost every $z^N$ 
\begin{align*}
\int_0^T \bigl\langle |b_t-b_{t}^N|^2(\cdot,z^N_t),z^N_t \bigr\rangle \, d t &\leq \left(\Lip(\Psi) \|\varphi\|_{\infty} \frac{2}{N}+\frac{1}{N} L(1+2\|\varphi\|_{\infty}) \right)^2,
\end{align*}
which vanishes as $N\to \infty$. \\

Finally, since $\Psi$ is Lipschitz, we find that for any $\varphi,b$ and approximating sequence $\varphi^{\lambda},b^{\lambda}$ the estimates \eqref{eq_b_g_n} and \eqref{eq_b_g} may be obtained with $g_\lambda(t,x,y) = |\varphi_t - \varphi_t^{\lambda}|^2$, and the rest of the proof follows similar to Proposition \ref{lm_plp12q}.
\end{proof}

\paragraph{\bf\em Background}
There are several papers dealing with large deviations for McKean--Vlasov SDEs. One of the first papers is \cite{Dawsont1987}, which proves an LDP for the paths of empirical measure, assuming continuity on the drift $b$ among other hypotheses. The papers \cite{Moral2003} and \cite{budhiraja2012} prove an LDP (for the empirical measures on the path space, as here), again assuming also continuity of the drift. The work \cite{Lac2018} proves propagation of chaos and a large deviation upper bound, assuming $b$ bounded and continuous in the measure argument, but with respect to a stronger topology (the $\tau$ topology). Outside the context of bounded (or linear growth) drift, we are aware only of the result in \cite{Fontbona2004}, which shows an LDP for a system of one-dimensional particles with a repulsive two-body interaction of order $1/x$, that is the framework of Subsection \ref{s_ex1} but with $\varphi(x)=1/x$, which is outside the class we can deal with here.

We also recall the recent papers \cite{Orr2018} and \cite{ReiSalTug2019} on large deviations for interacting diffusions/McKean--Vlasov SDEs when also the noise intensity tends to $0$, and \cite{GPP19} for the large deviations of the Brownian one-dimensional hard-rod system.

\subsection{Uniqueness for the McKean--Vlasov SDE}\label{sec:uniqueness}

In this subsection we consider the McKean--Vlasov SDE associated with the particle system \eqref{eq_SDEm}, namely
\begin{align}
\left\{
\begin{aligned}\label{eq_McKVla}
&dX_t = b_t(X_t,\law(X_t))\,dt +dW_t,\\
&X_0 \text{ with law }\rho_0,
\end{aligned}
\right.
\end{align}
for a given Borel drift $b:[0,T]\times\R^d \times \cP(\R^d)\rightarrow \R^d$ and a given initial law $\rho_0$.

We assume conditions on $b$ which include the examples in the previous subsection. We show, under these conditions, that the McKean--Vlasov SDE \eqref{eq_McKVla} admits a unique weak solution. As a consequence, in all examples in the previous section, the empirical measures associated with the particle system \eqref{eq_SDEm} converge, as $N\to\infty$, to the law of the solution of the McKean--Vlasov SDE \eqref{eq_McKVla}.

We keep the notation of the previous section, with $S=C([0,T];\R^d)$, $\mathbb{W}$ the $d$-dimensional Wiener measure with initial law $\rho_0$ and $W$ Brownian motion with initial law $\rho_0$ (similarly $W^i$ are independent Brownian motion starting from $\rho_0$). We fix $m> 2$ and, for $\eta>0$, we call
\begin{align*}
M_\eta = \biggl\{\mu\in \cP(S) \;\bigg|\; \left\|\frac{d\mu}{d\mathbb{W}}\right\|_{L^m(S,\mathbb{W})} \le \eta\biggr\}.
\end{align*}

\begin{thm}\label{thm_uniq_McKVla}
Fix the initial law $\rho_0 \in \cP(\R^d)$. Fix $m>2$. Assume that, for every $\eta>0$, for every $\beta>0$,
\begin{align}
\sup_{\mu\in \cP(S),\,R(\mu\|\mathbb{W})\le \eta} \E \left[\exp\left(\beta \int_0^T |b_t(W_t,\mu_t)|^2 dt \right)\right] <\infty.\label{eq_Xmu_exp_int}
\end{align}
Assume also that there is a non-negative Borel function $g:[0,T]\times (\R^{d})^k\rightarrow \R$ such that, for every $\eta>0$ and some  constant $C_\eta\ge0$,
\begin{align}
\begin{aligned}\label{eq_Xmu_Gibbs_bd}
&\int_0^T |b_t(x_t,\mu_t) -b_t(x_t,\nu_t)|^2 dt\\
&\hspace{4em}\le C_\eta \int_{S^{k-1}} \left(\int_0^T g_t(x_t,y_t)\,dt\right) \left|u^{\otimes(k-1)}(y)- v^{\otimes(k-1)}(y)\right|^2 d\mathbb{W}^{\otimes (k-1)}(y)\\[0.8em]
&\hspace{5.5em}\text{for }\mathbb{W}\text{-a.e. $x$ and every  $\mu,\nu \in M_\eta$},
\end{aligned}
\end{align}
where $u=d\mu/d\mathbb{W}$ and $v=d\nu/d\mathbb{W}$,
and that, for some $\tilde{m}>m$ with $\tilde{m}\ge (m/2)^\prime = \frac{m}{m-2}$,
\begin{align}
\E \left[\left( \int_0^T g_t(W^1_t,\ldots W^k_t)\, dt \right)^{\tilde{m}}\right] <\infty.\label{eq_Gibbs_int_g}
\end{align}
Then there exists a solution $X$ to the McKean--Vlasov equation and its law $\bar{\mu}$ is unique among the probability measures $\mu$ with $R(\mu\|\mathbb{W})<\infty$.
\end{thm}

\begin{proof}
For any $\mu\in\cP(S)$ with $R(\mu\|\mathbb{W})<\infty$, assumption \eqref{eq_Xmu_exp_int} gives via Girsanov theorem \ref{thm_girs} the existence of a weak solution $X^\mu$ to the SDE
\begin{align}
    dX_t^\mu = b_t( X_t^\mu,\mu_t)\,dt + dW_t,\label{eq_Xmu}
\end{align}
with law $F(\mu):=\mathbb{W}^\mu=\law(X^\mu)$ given by
\begin{align*}
\frac{dF(\mu)}{d\mathbb{W}}(W) = \exp\bigl(V_b(W,\mu)\bigr),
\end{align*}
with $V_b$ as defined in Section~\ref{s_proc_expl}. Note that $F(\mu)$ is the unique law solving \eqref{eq_Xmu} and having finite entropy with respect to $\mathbb{W}$. Indeed, if $\nu$ is the law of another solution $Y^\mu$ to \eqref{eq_Xmu} with $R(\nu\|\mathbb{W})<\infty$, then, by Lemma \ref{lm_entvar},
\begin{align*}
\int_S \int_0^T |b_t(x_t,\mu_t)|^2 dt \,d\nu(x) \le R(\nu\|\mathbb{W}) + \log \E \left[e^{\int_0^T |b_t(W_t,\mu_t)|^2 dt} \right] <\infty.
\end{align*}
Therefore the uniqueness condition \eqref{eq_Girs_constraint} is met under $\nu$ and so $\nu=F(\mu)$. 

Moreover, the density of $F(\mu)$ with respect to $\mathbb{W}$ is in $L^\gamma(S,\mathbb{W})$ for every finite $\gamma>1$: Indeed, following standard computations (similarly to the proof of Theorem \ref{thm_pmainlip}) and using assumption \eqref{eq_Xmu_exp_int}, we have for some $c_\gamma$,
\begin{align}
\E\left[\left|\frac{dF(\mu)}{d\mathbb{W}}(W)\right|^\gamma\right] \le \E\left[e^{c_\gamma\int_0^T |b_t(W_t,\mu_t)|^2 dt}\right] <\infty.\label{eq_Lgamma_Girs}
\end{align}
Finally, note that a process $X$, with $R(\law (X)\|\mathbb{W})<\infty$, is a solution to the McKean--Vlasov SDE if and only if its law is a fixed point for $F$. We will show now that, for every $\eta>0$, for $T>0$ sufficiently small, the map $F$ is a contraction on $M_\eta$, endowed with the $L^m(S,\mathbb{W})$ norm. Existence and uniqueness of the fixed point, hence of the law of the McKean--Vlasov SDE, for a general $T$ follow by a standard iteration argument.

Let $\mu$ and $\nu$ be two measures in $M_\eta$. Using the elementary inequality $|e^a-e^b|\le \frac{1}{2}(e^a+e^b) |a-b|$, we get
\begin{align*}
\left|\frac{dF(\mu)}{d\mathbb{W}}(W)-\frac{dF(\nu)}{d\mathbb{W}}(W)\right| \le \frac12\left(\frac{dF(\mu)}{d\mathbb{W}}(W)+\frac{dF(\nu)}{d\mathbb{W}}(W)\right)|V_b(W,\mu)-V_b(W,\nu)|.
\end{align*}
So by H\"older's inequality with $1/\gamma+1/\tilde{m}=1/m$ (with shorthand $L^\gamma=L^\gamma(S,\mathbb{W})$, $\gamma>1$),
\begin{align*}
\left\|\frac{dF(\mu)}{d\mathbb{W}}-\frac{dF(\nu)}{d\mathbb{W}}\right\|_{L^m} \le \frac12 \left(\left\|\frac{dF(\mu)}{d\mathbb{W}}\right\|_{L^\gamma} +\left\|\frac{dF(\nu)}{d\mathbb{W}}\right\|_{L^\gamma}\right) \|V_b(\cdot,\mu)-V_b(\cdot,\nu)\|_{L^{\tilde{m}}}.
\end{align*}
By \eqref{eq_Lgamma_Girs}, \eqref{eq_Xmu_exp_int} and the fact that $R(\mu\|\mathbb{W})\le \|d\mu/d\mathbb{W}\|_{L^\gamma}^\gamma$ for any $\mu \in \cP(S)$, we have, for some $C_\eta>0$,
\begin{align*}
\left\|\frac{dF(\mu)}{d\mathbb{W}}\right\|_{L^\gamma} +\left\|\frac{dF(\nu)}{d\mathbb{W}}\right\|_{L^\gamma} \le 2\sup_{\rho\in M_\eta}\E\left[e^{c_\gamma\int_0^T |b_t(W_t,\rho_t)|^2 dt}\right] \le C_\eta.
\end{align*}
By definition of $V_b$ (in Subsection \ref{s_proc_expl}), we obtain
\begin{align*}
\|V_b(\cdot,\mu)-V_b(\cdot,\nu)\|_{L^{\tilde{m}}} 
&\le \left\|\int_0^T (b_t(W_t,\mu_t)-b_t(W_t,\nu_t))\cdot dW_t\right\|_{L^{\tilde{m}}} \\
&\hspace{2em}+ \frac12 \Bigl\|\int_0^T \left\|b_t(W_t,\mu_t)|^2-|b_t(W_t,\nu_t)|^2\Bigr| dt\right\|_{L^{\tilde{m}}} = \textrm{(I)} + \textrm{(II)}.
\end{align*}
For the first term, we obtain by the Burkholder--Davis--Gundy inequality
\begin{align*}
    \textrm{(I)} &\le \tilde c_1\left\|\int_0^T |b_t(W_t,\mu_t)-b_t(W_t,\nu_t)|^2 dt\right\|_{L^{\tilde{m}/2}}^{\frac{1}{2}} \le c_1\left\|\int_0^T |b_t(W_t,\mu_t)-b_t(W_t,\nu_t)|^2 dt\right\|_{L^{\tilde{m}}}^{\frac{1}{2}}
\end{align*}
for some constants $\tilde c_1, c_1>0$. As for the second term, we estimate as follows 
\begin{align*}
    \textrm{(II)} &\le \tilde c_2 \left\|\left(\int_0^T\!\! \Bigl(|b_t(W_t,\mu_t)|^2+|b_t(W_t,\nu_t)|^2\Bigr) dt \right)^{\frac{1}{2}} \cdot \left(\int_0^T\!\! |b_t(W_t,\mu_t)-b_t(W_t,\nu_t)|^2 dt\right)^{\frac{1}{2}} \right\|_{L^{\tilde{m}}} \\
    &\le c_2  \left(\left\|\int_0^T\!\! |b_t(W_t,\mu_t)|^2 dt\right\|_{L^{\tilde{m}}}^{\frac{1}{2}}+\left\|\int_0^T\!\! |b_t(W_t,\nu_t)|^2 dt\right\|_{L^{\tilde{m}}}^{\frac{1}{2}}\right) \left\| \int_0^T\!\! |b_t(W_t,\mu_t)-b_t(W_t,\nu_t)|^2 dt\right\|_{L^{\tilde{m}}}^{\frac{1}{2}}
\end{align*}
for some constants $\tilde c_2,c_2>0$. For the term with $\int_0^T |b_t(W_t,\mu_t)|^2 dt$, the inequality $a^{\tilde{m}}\le e^{\tilde{m} a}$ yields
\begin{align*}
\left\|\int_0^T |b_t(W_t,\mu_t)|^2 dt\right\|_{L^{\tilde{m}}}^{\frac{1}{2}} \le \E \left[ e^{\tilde{m}\int_0^T |b_t(W_t,\mu_t)|^2 dt} \right]^{\frac{1}{2\tilde{m}}} \le C_R,
\end{align*}
where we have used again \eqref{eq_Xmu_exp_int} and the fact that $R(\rho\|\mathbb{W})\le \|d\rho/d\mathbb{W}\|_{L^\gamma}^\gamma$. The same argument holds for the term with $\int_0^T |b_t(W_t,\nu_t)|^2 dt$. 

Putting the terms (I) and (II) together, we then obtain
\[
    \|V_b(\cdot,\mu)-V_b(\cdot,\nu)\|_{L^{\tilde{m}}} \le c_3 \left\| \int_0^T\!\! |b_t(W_t,\mu_t)-b_t(W_t,\nu_t)|^2 dt\right\|_{L^{\tilde{m}}}^{\frac{1}{2}}
\]
for a constant $c_3>0$. Using assumption \eqref{eq_Xmu_Gibbs_bd} (with the notation $\boldsymbol{W}=(W^1,\ldots, W^N)$ and $\boldsymbol{W}^{,1}=(W^2,\ldots,W^N)$, where $W^i$ are independent $d$-dimensional Brownian motions with initial law $\rho_0$), we further estimate the right-hand side to obtain
\begin{align*}
&\left\| \int_0^T |b_t(W_t,\mu_t)-b_t(W_t,\nu_t)|^2 dt\right\|_{L^{\tilde{m}}}^{1/2} \\
&\le c_4 \E^{W^1}\left[ \left( \E^{\boldsymbol{W^{,1}}}\left[\int_0^T g_t(\boldsymbol{W}_t) dt \left| \frac{d\mu}{d\mathbb{W}}^{\otimes(k-1)}(\boldsymbol{W^{,1}})-\frac{d\nu}{d\mathbb{W}}^{\otimes(k-1)}(\boldsymbol{W^{,1}}) \right|^2\right] \right)^{\tilde{m}} \right]^{1/(2\tilde{m})}\\
&\le c_5 \E^{W^1}\left[ \left( \E^{\boldsymbol{W^{,1}}}\left[\int_0^T g_t(\boldsymbol{W}_t)\,dt \;\cdot \right. \right.\right.\\
&\hspace{4em} \left.\left. \cdot \sum_{j=2}^k\left| \frac{d\mu}{d\mathbb{W}}(W^j)-\frac{d\nu}{d\mathbb{W}}(W^j) \right|^2 \prod_{2\le \ell\le k,\ell\neq j} \left(\frac{d\mu}{d\mathbb{W}}\vee\frac{d\nu}{d\mathbb{W}}(W^\ell)\right)^2 \right)^{\tilde{m}} \right]^{1/(2\tilde{m})}\\
&\le c_6 \sum_{j=2}^k\E^{W^1}\left[ \left(\E^{\boldsymbol{W^{,1}}} \left[\left(\int_0^T g_t(\boldsymbol{W}_t)\,dt\right)^{(m/2)^\prime}\right] \right)^{\tilde{m}/(m/2)^\prime} \cdot \right.\\
&\quad \quad \left.\cdot \left (\E^{\boldsymbol{W^{,1}}}\left[\left| \frac{d\mu}{d\mathbb{W}}(W^j)-\frac{d\nu}{d\mathbb{W}}(W^j) \right|^m \prod_{2\le \ell\le k,\ell\neq j} \left(\frac{d\mu}{d\mathbb{W}}\vee\frac{d\nu}{d\mathbb{W}}(W^\ell)\right)^m\right] \right)^{2\tilde{m}/m} \right]^{1/(2\tilde{m})}\\
&\le c_7 \left( \E \left[\int_0^T g_t(\boldsymbol{W}_t)\, dt \right]^{\tilde{m}} \right)^{1/(2\tilde{m})} \left\|\frac{d\mu}{d\mathbb{W}}\vee \frac{d\nu}{d\mathbb{W}}\right\|_{L^m}^{k-2} \left\|\frac{d\mu}{d\mathbb{W}} - \frac{d\nu}{d\mathbb{W}} \right\|_{L^m}\\
&\le c_8 \left( \E \left[\int_0^T g_t(\boldsymbol{W}_t)\, dt \right]^{\tilde{m}} \right)^{1/(2\tilde{m})} \eta^{k-2}\left\|\frac{d\mu}{d\mathbb{W}} - \frac{d\nu}{d\mathbb{W}}\right\|_{L^m}
\end{align*}
for appropriate constants $c_i>0$, $i=1,\ldots,8$. We conclude that, for some $C_\eta>0$,
\begin{align*}
\left\|\frac{dF(\mu)}{d\mathbb{W}}-\frac{dF(\nu)}{d\mathbb{W}}\right\|_{L^m} \le C_\eta\left( \E \left(\int_0^T g_t(\boldsymbol{W}_t)\, dt \right)^{\tilde{m}} \right)^{1/(2\tilde{m})} \left\|\frac{d\mu}{d\mathbb{W}} - \frac{d\nu}{d\mathbb{W}} \right\|_{L^m}.
\end{align*}
By assumption \eqref{eq_Gibbs_int_g}, we can find $T>0$ small enough such that
\begin{align*}
C_\eta\left( \E \left[\int_0^T g_t(\boldsymbol{W}_t)\, dt \right]^{\tilde{m}} \right)^{1/(2\tilde{m})} <1.
\end{align*}
Hence, for such $T$, $F$ is a contraction on $M_\eta$, thereby concluding the proof.
\end{proof}

\begin{remark}\label{rmk_extension_uniq_McKVla}
As the proof shows, assumption \eqref{eq_Xmu_Gibbs_bd} may be replaced by the weaker one
\begin{align}\label{eq_Xmu_Gibbs_bd_2}
\begin{aligned}
&\int_0^T |b_t(x_t,\mu_t) -b_t(x_t,\nu_t)|^2 dt \\
&\hspace{1em}\le C_\eta \int_{S^{k-1}} \left(\int_0^T g(t,x_t,y_t)\,dt\right) \sum_{j=2}^k \left|u(y^j)- v(y^j)\right|^2 \prod_{\stackrel{2\le \ell\le k}{\ell\neq j}} (u \vee v)(y^\ell)^2 d\mathbb{W}^{\otimes (k-1)}(y) \\[0.2em]
&\hspace{3em}\text{for }\mathbb{W}\text{-a.e. $x$ and every  $\mu,\nu \in M_\eta$},
\end{aligned}
\end{align}
where $u=d\mu/d\mathbb{W}$ and $v=d\nu/d\mathbb{W}$.
\end{remark}

\begin{cory}\label{cor_poc}
Under the assumptions of Theorems \ref{thm_pmainlip} and \ref{thm_uniq_McKVla} (possibly modified as in Remark~\ref{rmk_extension_uniq_McKVla}), as $N\to\infty$, the family of empirical measures $z^N_{\boldsymbol{X}}$ associated with the particle system \eqref{eq_SDEm} converges almost surely to the (unique) law $\bar{\mu}$ of the McKean--Vlasov SDE \eqref{eq_McKVla}.
\end{cory}

\begin{proof}
By Theorem \ref{thm_uniq_McKVla}, there exists only one solution $X$ to \eqref{eq_McKVla} with finite entropy with respect to $\mathbb{W}$, which is then the unique zero of the rate function $\cF$ of the LDP for $(z^N_{\boldsymbol{X}})_N$ in Theorem \ref{thm_pmainlip}. Hence Lemma~\ref{lm:almost-sure} applies.
\end{proof}

\begin{remark}\label{rem_poc}
As noted by Sznitman in \cite{Szn1998,sznitman1991}, convergence in law of the empirical measures $z^N$ to the constant variable $\bar{\mu}$ implies that the sequence $\tilde Q^N_{b^N}\in \cP(S^N)$ is $\bar{\mu}$-chaotic, in the sense that for every $k\in \mathbb{N}$, 
\begin{equation*}
    \law \left( X^{N,1},\dots,X^{N,k} \right) \to  \bar{\mu}^{\otimes k},
\end{equation*}
weakly as $N\to \infty$ on $S=C([0,T];\R^d)$. In particular, we have a form of propagation of chaos, namely that for all $k\in \mathbb{N}$, 
\begin{equation*}
    \law \left( X_t^{N,1},\dots,X_t^{N,k} \right) \to \bar{\mu}_t^{\otimes k}, \qquad \forall t\in [0,T].
\end{equation*}
\end{remark}
\medskip
Now we come back to the examples in Subsections \ref{s_ex1}, \ref{s_ex2}, \ref{s_ex3}, recalling the drift for the corresponding McKean--Vlasov SDEs, namely:
\begin{itemize}
\item for the example in Subsection \ref{s_ex1},
\begin{align*}
b_t(x,\mu)=\int_{\R^d}\varphi_t(x,y)\,d\mu_t(y);
\end{align*}
\item for the example in Subsection \ref{s_ex2},
\begin{align*}
b_t(x,\mu)=\int_{\R^{(k-1)d}}\varphi_t(x,y_2,\ldots y_k)\,d\mu_t(y_2,\ldots y_k);
\end{align*}
\item for the example in Subsection \ref{s_ex3},
\begin{align*}
b_t(x,\mu)=\int_{\R^d}\Psi \left(x,y, \int_{\R^d}\varphi_t(x,z)\,d\mu_t(z),\int_{\R^d}\varphi_t(y,z)\,d\mu_t(z),\mu \right) d\mu_t(y).
\end{align*}
\end{itemize}

\begin{prop}\label{cor_poc_ex}
For the examples in Subsections \ref{s_ex1}, \ref{s_ex2}, \ref{s_ex3}, under the assumptions of Propositions \ref{lm_plpq}, \ref{lm_plp12q} and \ref{thm_pmark1} respectively, the corresponding McKean--Vlasov SDEs admit a unique solution $X$ with $R(\law( X)\|\mathbb{W})<\infty$, and the family of empirical measures $z^N_{\boldsymbol{X}}$ converges almost surely to the law of $X$.
\end{prop}

 \begin{proof}
 We have seen in the previous section that the examples above satisfy the assumptions of Theorem \ref{thm_pmainlip}, so it is enough to verify the assumptions of Theorem \ref{thm_uniq_McKVla}, possibly modified as in Remark \ref{rmk_extension_uniq_McKVla}. Assumption \eqref{eq_Xmu_exp_int} is a consequence of \eqref{eq:finite_lambda} in the proof of Theorem \ref{thm_pmainlip}. Assumptions \eqref{eq_Xmu_Gibbs_bd} and \eqref{eq_Gibbs_int_g} are satisfied (with any given $m$ and $\tilde{m}$ as in Theorem \ref{thm_uniq_McKVla}):
 \begin{itemize}
 \item for the example in Subsection \ref{s_ex1}, taking $g(t,x,y)= |\varphi(t,x,y)|^2$;
 \item for the example in Subsection \ref{s_ex2}, taking $g(t,x_1,\ldots x_k)= |\varphi(t,x_1,\ldots x_k)|^2$.
 \end{itemize}
 We will not show the proof of this fact, which is similar to, and easier than, the next proof in the example in Subsection \ref{s_ex3}.

For the example in Subsection \ref{s_ex3}, we will show that assumptions \eqref{eq_Xmu_Gibbs_bd_2} and \eqref{eq_Gibbs_int_g} are satisfied, taking
 \begin{align*}
 g(t,x,y,z) = \left(\Lip(\Psi)^2+L^2\right)\left(1+|\varphi_t(x,z)|^2+|\varphi_t(y,z)|^2\right),
 \end{align*}
where $L$ is the linear growth constant for $\Psi$ as in \eqref{eq_MD_Lgrowth2}. We start showing \eqref{eq_Xmu_Gibbs_bd_2}. Denoting
 \[
 \Psi_t(x,y,\mu^1,\mu^2,\mu^3) = \Psi \left(x_t,y_t, \int_{S}\varphi_t(x_t,z_t)d\mu^1(z),\int_{S}\varphi_t(y_t,z_t)d\mu^2(z),\mu^3 \right),
 \]
 $u=d\mu/d\mathbb{W}$ and $v=d\nu/d\mathbb{W}$, we have the following estimates (with a generic constant $C>0$)
 \begin{align*}
 \int_0^T |b_t(x_t,\mu_t) -b_t(x_t,\nu_t)|^2 dt &= \int_0^T \left| \int_{S} \Psi_t(x,y,\mu,\mu,\mu)\, d\mu(y) - \int_{S} \Psi_t(x,y,\nu,\nu,\nu) d\nu(y) \right|^2 dt\\
 &\le C \int_0^T \left| \int_{S} \Psi_t (x,y,\mu,\mu,\mu) d(\mu-\nu)(y)\right|^2 dt\\
 &\hspace{1em} +C \int_0^T \left| \int_{S} [\Psi_t (x,y,\mu,\mu,\mu) -\Psi_t(x,y,\nu,\mu,\mu)] d\nu(y)\right|^2 dt\\
 &\hspace{1em} +C \int_0^T \left| \int_{S} [\Psi_t (x,y,\nu,\mu,\mu) -\Psi_t(x,y,\nu,\nu,\mu)] d\nu(y)\right|^2 dt \\
 &\hspace{1em} +C \int_0^T \left| \int_{S} [\Psi_t (x,y,\nu,\nu,\mu) -\Psi_t(x,y,\nu,\nu,\nu)] d\nu(y)\right|^2 dt\\
 &=:\mathrm{(I)}+\mathrm{(II)}+\mathrm{(III)}+\mathrm{(IV)}.
 \end{align*}
 For the term $\mathrm{(I)}$, we have
 \begin{equation*}
     \begin{aligned}
   \mathrm{(I)}&=C \int_0^T \left| \int_{S} \Psi_t (x,y,\mu,\mu,\mu) d(\mu-\nu)(y)\right|^2 dt \\ 
   &\leq  CL^2 \int_0^T \left| \int_{S} \left(1+\left|\int_{S}\varphi_t(x_t,z_t)d\mu(z)\right|+ \left|\int_{S}\varphi_t(y_t,z_t)d\mu(z)\right|\right) |u(y)-v(y)| d\mathbb{W}(y)\right|^2   dt\\
   &\leq  CL^2 \int_0^T \left| \int_{S}\int_{S} \left(1+|\varphi_t(x_t,z_t)|+ |\varphi_t(y_t,z_t)|\right) |u(y)-v(y)| d\mathbb{W}(y) d\mu(z) \right|^2   dt\\
   &=  CL^2 \int_0^T \left| \int_{S^2} \left(1+|\varphi_t(x_t,z_t)|+ |\varphi_t(y_t,z_t)|\right) |u(y)-v(y)| u(z) d\mathbb{W}^{\otimes 2}(y,z) \right|^2   dt\\
   &\leq 3 C \int_{S^2} \left(\int_0^T L^2 (1+|\varphi(t,x_t,z_t)|^2+|\varphi(t,y_t,z_t)|^2) dt\right) |u(y)-v(y)|^2 u(z)^2 d\mathbb{W}^{\otimes 2}(y,z).
     \end{aligned}
 \end{equation*}
 For the term $\mathrm{(II)}$, we have
 \begin{align*}
 \mathrm{(II)} &\le C \int_0^T \int_{S} \Lip(\Psi)^2 \left| \int_{S} \varphi(t,x_t,z_t) d(\mu-\nu)(z) \right|^2 d\nu(y) dt\\
 &= C \int_0^T \Lip(\Psi)^2 \left|\int_{S^2} \varphi(t,x_t,z_t) (u(z)-v(z))u(y) d\mathbb{W}^{\otimes 2}(y,z) \right|^2 dt\\
 &\le C \int_{S^2} \left(\int_0^T \Lip(\Psi)^2 |\varphi(t,x_t,z_t)|^2 dt\right) |u(z)-v(z)|^2 u(y)^2 d\mathbb{W}^{\otimes 2}(y,z),
 \end{align*}
 where $u(y)$ has been added artificially to satisfy \eqref{eq_Xmu_Gibbs_bd_2} with $k=3$. As for $\mathrm{(III)}$,
 \begin{align*}
 \mathrm{(III)} &\le C \int_0^T \int_{S} \Lip(\Psi)^2 \left| \int_{S} \varphi(t,y_t,z_t) d(\mu-\nu)(z) \right|^2 d\nu(y) dt\\
 &= C \int_0^T \int_{S} \Lip(\Psi)^2 \left|\int_{S} \varphi(t,x_t,z_t) (u(z)-v(z)) d\mathbb{W}(z) \right|^2 v(y) d\mathbb{W}(y) dt\\
 &\le C \int_{S^2} \left(\int_0^T \Lip(\Psi)^2 |\varphi(t,y_t,z_t)|^2 dt\right) |u(z)-v(z)|^2 v(y)^2 d\mathbb{W}^{\otimes 2}(y,z).
 \end{align*}

 In view of the term $\mathrm{(IV)}$, we recall that, for any $t\in [0,T]$,
 \begin{align*}
 d_{BL}(\mu_t,\nu_t) \le d_{BL}(\mu,\nu),
 \end{align*}
 which follows from the fact that the map $e_t:S\to \R^d$, $e_t(x):=x_t$, is $1$-Lipschitz. Hence,
 \begin{align*}
     \mathrm{(IV)}&=C \int_0^T \left| \int_{S} [\Psi_t (x,y,\nu,\nu,\mu) -\Psi_t(x,y,\nu,\nu,\nu)] d\nu(y)\right|^2 dt\\
     &\leq C \int_0^T \Lip(\Psi)^2 \, d_{BL}(\mu_t,\nu_t)^2  dt \\
     &\leq CT \, \Lip(\Psi)^2 \, d_{BL}(\mu,\nu)^2.
 \end{align*}
We also recall that the bounded Lipschitz metric $d_{BL}$ is bounded by the total variation distance $d_{TV}$. Therefore we derive, adding again an artificial term $u(z)$,
 \begin{align*}
     \mathrm{(IV)}&\leq C T \,\Lip(\Psi)^2 d_{TV}(\mu,\nu)^2  = C T \,\Lip(\Psi)^2 \|u-v\|^2_{L^1} \\
     &=C \left(\int_0^T \,\Lip(\Psi)^2\, d t\right) \left(\int_{S^2}  |u(y)-v(y)| u(z) d\mathbb{W}^{\otimes 2}(y,z)\right)^2\\
     &\leq C \int_{S^2} \left(\int_0^T \,\Lip(\Psi)^2\, d t\right) |u(y)-v(y)|^2 u(z)^2 d\mathbb{W}^{\otimes 2}(y,z).
 \end{align*}
Putting together $\mathrm{(I)}$, $\mathrm{(II)}$, $\mathrm{(III)}$, $\mathrm{(IV)}$, we get \eqref{eq_Xmu_Gibbs_bd_2}.

 Finally, to complete the proof, we verify assumption \eqref{eq_Gibbs_int_g} for $g$ (with any fixed $m$ and $\tilde{m}$ as in Theorem \ref{thm_uniq_McKVla}). Note that it is enough to prove 
 \begin{equation*}
     \E \left[\left( \int_0^T |\varphi_t(W^1_t,W^2_t)|^2\, dt \right)^{\tilde{m}}\right] <\infty.
 \end{equation*}
 But this easily follows from the assumption \eqref{eq_expphi2} on $\varphi$,
  \begin{equation*}
       \E \left[ e^{\beta \int_{0}^T |\varphi_t|^2(W_t^{1},W_t^{2})\,d t} \right] < \infty,  \qquad \forall \beta\in \R,
 \end{equation*}
 which implies the finiteness of all moments of the variable $\int_0^T |\varphi_t(W^1_t,W^2_t)|^2 d t$. 
 \end{proof}

\paragraph{\bf\em Background}
The convergence of the particle system to the McKean--Vlasov SDE, in the sense of Corollary \ref{cor_poc}, is classical in the case of Lipschitz bounded drift, see e.g. \cite{sznitman1991}. The case of non-Lipschitz drift has also been treated in various works and we mention only some of them. In the context of the example in Subsection \ref{s_ex1}, the paper \cite{Jabin2018} proves the convergence, with quantitative bounds, for $\varphi$ in $W^{-1,\infty}$ (which includes our example) such that $\text{div}(\varphi)$ is in $W^{-1,\infty}$. The work \cite{Godinho2015} covers the case $\varphi=-\nabla \Phi$ with $\Phi(x)=|x|^\alpha$ with $\alpha$ in $(0,1)$, which is a relevant example of Subsection \ref{s_ex1} for $d\ge 2$ (see Remark \ref{rem_psubz}). Both in \cite{Jabin2018} and \cite{Godinho2015} the initial conditions are assumed to be diffuse in a suitable sense. As examples of convergence in critical cases, that are not covered by our results, we recall \cite{Fournier2015}, for $\varphi=-\nabla \Phi$ with $\Phi(x)=\log|x|$, and \cite{Fournier2014}, for the 2D Navier-Stokes equations and the associated vortex system, that is $\varphi(x)=x^{\perp}/|x|^2$. Outside the context of Subsection \ref{s_ex1}, we already mentioned \cite{Lac2018}, proving convergence in the $\tau$ topology when $b$ is bounded and satisfies a suitable continuity assumption in the measure argument. The paper \cite{Jab2019} proves convergence for a general measure-dependent drift, assuming a quite weak condition but depending on the solution to the McKean--Vlasov itself; the result is then applied to the case of bounded drifts. However we are not aware of a convergence result that covers our Corollary \ref{cor_poc} and Proposition \ref{cor_poc_ex}.

Concerning (weak or strong) uniqueness for McKean--Vlasov SDEs with irregular drifts, we mention \cite{MehSta2019,HamSisSzp2018,Cha2020,RocZha2018,MisVer2016,FlaPriZan2019}. It is also worth mentioning \cite{Del2019} on a regularization by noise phenomenon, via an infinite-dimensional noise, for a related mean field game problem.

\newpage
\appendix

\numberwithin{equation}{section}
\section{Limits of convex functions}\label{s_convex}
Let $V$ be a vector space. We consider a sequence of convex functions $\phi_n:V\to \REx$ (with domains $D_n$), and two convex functions $\phi_L$ and $\phi_U$ with $\phi_L\leq \phi_U$ on $V$ (with respective domains $D_L$, $D_U$), which will act as asymptotic lower and upper bounds for $\phi_n$ in a sense specified below. 

Moreover, we consider pairs of sequences and points $(\{y_n\},y) \in V^{\mathbb{N}} \times V$, which for simplicity will be referred to as the pairs $(y_n,y)$, and denote $(y_n,y)\in D$ for the statement $y\in D_U$ and $y_n\in D_n$ for every $n$. We then have the following approximation result. 

\begin{thm}\label{thm_conv1}
	Let $\gamma>1$, and let the sequence of pairs $(y_{\lambda,n},y_{\lambda})$ be such that for every $\lambda>0$,
	\begin{subequations}
		\begin{align}
		\limsup_{n \to \infty} \phi_n(\gamma y_{\lambda,n})&<\infty,\label{eq_cty1a}\\
		\phi_U(\gamma y_{\lambda})&<\infty,\label{eq_cty1b}
		\end{align}
	\end{subequations}
	\begin{equation}
	\begin{aligned}\label{eq_cty2}
	\phi_L(y_{\lambda})&\leq \liminf_{n\to \infty} \phi_n\big(y_{\lambda,n}\big) \leq \limsup_{n\to \infty} \phi_n\big(y_{\lambda,n}\big)\leq\phi_U(y_{\lambda}).
	\end{aligned}
	\end{equation}
	Moreover, let the couple $(x_n,x)$ be such that there exists a $K\in\R$ such that for all $\beta \in \R$,  
	\begin{subequations}\label{eq_cthc1}
		\begin{align}
		\limsup_{\lambda \to 0}\limsup_{n\to \infty} \phi_n \big(\beta(y_{\lambda,n}-x_n)\big)\leq K,\label{eq_tconvr1}\\
		\limsup_{\lambda \to 0} \phi_U\big(\beta (y_{\lambda}-x)\big)\leq K.\label{eq_tconvr2}
		\end{align}
	\end{subequations}
	Then for any $0<\gamma'<\gamma$,
	\begin{subequations}
	\begin{align}
	\limsup_{n \to \infty} \phi_n(\gamma' x_{n})&<\infty, \label{eq_ctx1a}  \\
	\phi_U(\gamma' x)&<\infty,\label{eq_ctx1b}
	\end{align}
	\end{subequations}
	and that
	\begin{align}\label{eq_ctx1}
	\phi_L(x) \leq \liminf_{n\to \infty} \phi_n\big(x_{n}\big) \leq \limsup_{n\to \infty} \phi_n\big(x_{n}\big)\leq\phi_U(x).
	\end{align}
\end{thm}

\begin{remark}\label{rem_convex1}
Some comments on these assumptions:
\begin{enumerate}[label=(\roman*)]
\item Note that the convex functions $\phi_L,\phi_N,\phi$ are all allowed to be \emph{improper}, i.e.\ allowed to be equal to $-\infty$ on their domain, or to have empty domain.
\item The constant $K$ in \eqref{eq_cthc1} is independent of $\beta$, which is crucial in proving the convergence. As an example, note that when $\phi_U$ is even with $\phi_U(0)=0$ the assumptions imply that \eqref{eq_ctx1b} holds for $K=0$ as well. 
\end{enumerate}
\end{remark}

First, we will establish some technical properties for limits of convex functions: a generalization of the classical statement on continuity of convex function on the interior of their domains to certain pointwise limits of convex functions, and a result that shows how under \eqref{eq_tconvr1} the limits in $n,\lambda$ of $\phi_n(x),\phi_n(y_{\lambda,n})$ in effect `commute'.  

\begin{lm}\label{lm_convex1}~
Let $g_n:\R\to \REx$ be a sequence of convex functions such that for some $a,b\in \R$
	\[\begin{aligned}
	\limsup_{n\to\infty} g_n(a)<\infty,\qquad \limsup_{n\to\infty} g_n(b)<\infty.
	\end{aligned}\]
	Then the functions
	\begin{equation*}
	\begin{aligned}
	\bar g(z) :=\limsup_{n\to \infty}g_n(z),\qquad
	g(z) :=\liminf_{n\to \infty}g_n(z)
	\end{aligned}
	\end{equation*}
	are both either equal to $-\infty$ on $(a,b)$, or finite and continuous on $(a,b)$.
\end{lm}

\begin{proof}
First, note that there exists a large enough $N$ such that for all $n\geq N$ both $g_n(a)$ and $g_n(b)$ are bounded from above, and in particular $[a,b]\subset D(g_n)$. Moreover, it is easy to verify that $\bar g=\limsup_{n\to \infty} g_n(z)$ is convex as well, $[a,b]\subset D(\bar g)$, and with $\bar g$ bounded from above on $[a,b]$. We will show that either $g \equiv -\infty$ on $(a,b)$ or that $\limsup_{n\to \infty} \|g_n\|_{\infty}$ is finite --- in which case we will derive that $g$ is Lipschitz continuous at any proper sub-interval $[c,d]\subsetneq [a,b]$ with $a<c<d<b$. The case for $\bar g$ follows from a similar argument and the fact that $\bar g$ is convex itself.\\

Now, suppose that $\limsup_{n\to \infty} \|g_n\|_{\infty}=\infty$. Since $\bar g(z)$ is bounded from above on $[a,b]$, it follows that there exists a subsequence $n'\in \mathbb{N}$ and a sequence $z_{n'}\in [a,b]$ such that 
\[\lim_{n'\to \infty} g_{n'}(z_{n'})=-\infty.\]
By compactness in $[a,b]$ we can choose a converging sequence $z_{n'}\to z^*$. We will treat the case $z^*\in (a,b)$ and $z^*=a,b$ separately. First, assume the former and fix $s\in (a,b)$. Then for any small enough $\epsilon>0$ and large enough $n'$ such that $g_{n'}(z_{n'})<\infty$ and $|z_{n'}-z^*|<\epsilon$ with $a+\epsilon<z^*<b-\epsilon$, we have by convexity of $g_{n'}$ 
\begin{align*}
g_{n'}(s)&\leq \frac{z_{n'}-s}{z_{n'}-a} g_{n'}(a)+\frac{s-a}{z_{n'}-a}  g_{n'}(z_{n'})\\
&\leq g_{n'}(a)^++\frac{s-a}{z^*-a-\epsilon}  g_{n'}(z_{n'}),
\end{align*}
whenever $s\in (a,z_{n'}]$, and where $g_{n'}(a)^+=\min(0,g_{n'}(a))$. Repeating the argument for $s\in [z_{n'},b)$ we derive
\begin{align*}
 \limsup_{n'\to \infty} g_{n'}(s) &\leq \limsup_{n'\to \infty} \max\left(g_{n'}^+(a),g_{n'}^+(b)\right)+\max\left(\frac{s-a}{z^*-a-\epsilon},\frac{b-s}{b-z^*-\epsilon} \right) \limsup_{n'\to \infty}g_{n'}(z_{n'})\\
 &=-\infty,
\end{align*}
 In particular we conclude that $g(s):=\liminf_{n\to\infty}g_n(s)=-\infty$. The case for $z^*=b$ (or $z^*=a$) is similar, using the fact that $s\in (a,z_n')$ for large enough $n'$.

Next, suppose otherwise, i.e. $\limsup_{n\to \infty} \|g_n\|_{\infty}<\infty$. Recall from classical convex analysis on $\R^d$ (e.g.\ similar to \cite[Theorem 6.7]{Evans2015}) that bounded convex functions on convex sets $O$ are uniformly Lipschitz on certain subsets $A\subsetneq O$ with $d(A,O^c)>0$. In particular, when $f:[a,b]\to \R$ is convex and bounded, $a+\delta\leq c<d\leq b-\delta$, for some $\delta>0$, the Lipschitz constant of $f$ on the interval $[c,d]$ is bounded by \[K:=2 \delta^{-1} \|f\|_{\infty}.\]
Since pointwise limits or pointwise minima of $K$-Lipschitz functions are also $K$-Lipschitz, and 
\begin{equation*}
g(z)=\lim_{n\to \infty} \left(\lim_{m\to \infty} \min_{n\leq l\leq m} g_l(z) \right),
\end{equation*}
it follows that $g$ is $K$-Lipschitz as well. Since $c,d$ are arbitrary, we conclude that $g$ is continuous on $(a,b)$. 
\end{proof}

\begin{lm}\label{lm_convex2}
	Let $(\phi_n)_{n\in\N}$ be a sequence of convex functions on $V$ and $(y_{\lambda,n})$ a family of sequences and $(x_n)$ a sequence in $V$. Suppose there exists a $\gamma>1$ such that
	\begin{equation}\label{eq_convexr1}
	\limsup_{n\to \infty}\phi_{n}\big(\gamma y_{\lambda,n}\big)< \infty,\qquad\text{for all $\lambda>0$},
	\end{equation}
	and suppose there exists some constant $K\in\R$ such that
	\begin{equation}\label{eq_convexr2}
	\limsup_{\lambda \to 0}\limsup_{n\to \infty}\phi_{n}\big(\beta(y_{\lambda,n}-x_n)\big)\leq K\qquad\text{for all $\beta \in \R$}.
	\end{equation}
	Then for any $0\leq \gamma'<\gamma$,
	\begin{equation}\label{eq_convexr1b}
	\limsup_{n\to \infty} \phi_n(\gamma' x_n)<\infty,
	\end{equation}
	and
	\begin{subequations}
		\begin{align}
		\lim_{\lambda \to 0}\limsup_{n\to \infty}\phi_{n}\big( y_{\lambda,n}\big)&=\limsup_{n\to \infty}\phi_{n}\big(x_{n}\big),\label{eq_convexrr1a}	\\
		\lim_{\lambda \to 0}\liminf_{n\to \infty}\phi_{n}\big( y_{\lambda,n}\big)&=\liminf_{n\to \infty}\phi_{n}\big(x_{n}\big).\label{eq_convexrr1b}
		\end{align}
	\end{subequations}
\end{lm}
\begin{proof}
First, note that if $\phi$ is convex, we have for any $x,y\in V$ and $\alpha\in(0,1)$ for which both $\alpha^{-1}(y-x),(1-\alpha)^{-1}x\in D(\phi)$ that $x+y\in D(\phi)$
	\[
	\phi(x+y) 
	\leq \alpha \phi\big(\alpha^{-1}x\big)+(1-\alpha)\phi\big((1-\alpha)^{-1}y \big),
	\]
	and therefore, 
	\begin{equation}\label{eq_convex1a}
	\begin{aligned}
	\phi(y)&=\phi(y-x+x)\\
	&\leq \alpha \phi \left(\alpha^{-1}(y-x)\right) + (1-\alpha)\phi\left((1-\alpha)^{-1}x\right).
	\end{aligned}
	\end{equation}
Now, we begin by establishing \eqref{eq_convexr1b}. By \eqref{eq_convexr1} and \eqref{eq_convexr2} there exists a large enough $N$ such that for all $n\geq N^*$, $\lambda$, $\beta\in R$,  and all $\beta \in \R$ both $\gamma y_{\lambda,n}\in D_n$ and $\beta (x-y_{\lambda})\in D_n$. Therefore, for any $\gamma'\in[0,\gamma)$, by \eqref{eq_convex1a} we have
	\[
	 \phi_n(\gamma' x_{n}) \leq (1-\alpha) \phi_n\left((1-\alpha)^{-1} \gamma' (x_{n}-y_{\lambda,n})\right)+\alpha \phi_n(\gamma y_{\lambda,n}).
	\]
	with $\alpha=\gamma'/\gamma$. Passing to the limes superior in $n$, we obtain
	\[
	 \limsup_{n\to \infty} \phi_n(\gamma' x_{n}) \leq 	(1-\alpha) \limsup_{n\to \infty} \phi_n\left((1-\alpha)^{-1} \gamma' (x_{n}-y_{\lambda,n})\right)+\alpha \limsup_{n\to \infty} \phi_n(\gamma y_{\lambda,n}).
	\]
	Note that the left-hand side is independent of $\lambda>0$, while the second term on the right-hand side is finite for every $\lambda>0$ by \eqref{eq_convexr1}. Moreover, by \eqref{eq_convexr2} it follows that the first term on the right-hand side is finite for sufficiently small $\lambda>0$. Hence, the left-hand side is finite as well and in particular, we have shown \eqref{eq_convexr1b}. 
	
Next, from \eqref{eq_convex1a} we find for any $\alpha\in(0,1)$
	\begin{align*}
	    \phi_n(\alpha x_{n}) &\leq (1-\alpha) \phi_n\left(\frac{\alpha}{1-\alpha} (x_{n}-y_{\lambda,n})\right)+\alpha \phi_n(y_{\lambda,n}),\\
	    \phi_n(y_{\lambda,n}) &\leq (1-\alpha) \phi_n\left((1-\alpha)^{-1} (y_{\lambda,n}-x_{n})\right)
		+\alpha \phi_n\left(\alpha^{-1}x_{n}\right).
	\end{align*}
	Taking limits in $\lambda$ and $n$, we obtain
	\begin{subequations} \label{eq_convexss2}
		\begin{align}
		\limsup_{n\to \infty} \phi_n(\alpha x_{n}) &\leq (1-\alpha)K+\alpha \liminf_{\lambda \to 0} \limsup_{n\to \infty} \phi_n(y_{\lambda,n}),\label{eq_convexs2a}\\
		\limsup_{\lambda \to 0}  \limsup_{n\to \infty} \phi_n(y_{\lambda,n}) &\leq(1-\alpha)K+
		\alpha \limsup_{n \to \infty}  \phi_n\left(\alpha^{-1} x_{n}\right),
		\end{align}
	\end{subequations}
	where we made use of \eqref{eq_convexr2}. Note that $\bar g(z):= \limsup_{n\to \infty} \phi_n(z x_n)$ is bounded from above around $z=1$, and by Lemma \ref{lm_convex1} is either continuous in $(0,\gamma')$ or $\bar g(z)\equiv -\infty$ in a neighborhood around $z=1$. In both cases, passing to the limit $\alpha\to 1$ in \eqref{eq_convexs2a} and \eqref{eq_convexs2a}, we recover \eqref{eq_convexrr1a}.
	
	Similarly, to establish \eqref{eq_convexrr1b}, we first note that after repeating the argument of \eqref{eq_convexss2}, we find
	\begin{align*}
		\liminf_{n\to \infty} \phi_n(\alpha x_{n}) &\leq (1-\alpha)K+	\alpha \liminf_{\lambda\to 0}\liminf_{n\to \infty} \phi_n( y_{\lambda,n}),\\
		\limsup_{\lambda \to 0}\liminf_{n\to \infty} \phi_n( y_{\lambda,n}) &\leq (1-\alpha)K+
		\alpha\liminf_{n\to \infty} \phi_n\left(\alpha^{-1} x_{n}\right).
	\end{align*}
Set $g_n(z):=\phi_n\left(z x_{n}\right)$, $g(z):=\liminf_{n\to \infty} g_n(z)$, and recall that $\bar g(z)$ is bounded from above on $[0,\gamma']$. Hence, again applying Lemma \ref{lm_convex2} and passing to the limit $z\to 1$ in $g(z)$, and $\alpha\to 1$ above, we conclude the proof. 
\end{proof}

Now, as a special case, for when $\phi_n=\phi$, $x_n=x$ and $y_{\lambda,n}=y_{\lambda}$, we have the following statement.

\begin{cory}\label{cor_convex1}
	Let $\phi$ be a convex function on $V$. Suppose there is a $\gamma>1$ such that $\phi(\gamma y_{\lambda})<\infty$ for all $\lambda>0$, and that there exists a constant $K\in\R$ such that 
	\[
	\limsup_{\lambda \to 0}\phi\big(\beta(x-y_{\lambda})\big)\leq K\qquad\text{for all $\beta \in \R$}.
	\]
	Then for any $0<\gamma'<\gamma$, $\phi(\gamma' x)< \infty$, and $\phi(x)=\lim_{\lambda \to 0} \phi(y_{\lambda})$.
\end{cory}

Together, these results show how to connect $\phi_L(x)$ to $\phi_L(y_{\lambda})$, $\phi_n(x_n)$ to $\phi_n(y_{\lambda,n})$, and $\phi_U(x)$ to $\phi_U(y_{\lambda})$. Since by assumption $\phi_n(y_{\lambda,n})$ is connected to $\phi_L(y_{\lambda})$ and $\phi_U(y_{\lambda})$, we derive corresponding statements for $(x_n,x)$. Indeed, we now show how the above results imply Theorem \ref{thm_conv1}. 

\begin{proof}[Proof of Theorem \ref{thm_conv1}]
	Applying Corollary~\ref{cor_convex1} to both $\phi_L$ and $\phi_U$ separately, we obtain from \eqref{eq_cty1b} and \eqref{eq_tconvr2} the convergences
	\[
	\phi_L(x)=\lim_{\lambda \to 0} \phi(y_{\lambda}),\qquad
	\phi_U(x)=\lim_{\lambda \to 0} \phi(y_{\lambda}).
	\]
	Similarly, from \eqref{eq_cty1a} and \eqref{eq_tconvr1} it follows from Lemma~\ref{lm_convex2} that
		\begin{align*}
		\lim_{\lambda \to 0}\limsup_{n\to \infty}\phi_{n}\big( y_{\lambda,n}\big)&=\limsup_{n\to \infty}\phi_{n}\big(x_{n}\big),	\\
		\lim_{\lambda \to 0}\liminf_{n\to \infty}\phi_{n}\big( y_{\lambda,n}\big)&=\liminf_{n\to \infty}\phi_{n}\big(x_{n}\big).
		\end{align*}
	Finally, we use the above and the relationship between $y_{\lambda,n}$ and $y_{\lambda}$ assumed in \eqref{eq_cty2}, to obtain
	\begin{alignat*}{3}
	\phi_L(x) = \lim_{\lambda \to 0} \phi_L(y_{\lambda})
	&\leq \lim_{\lambda \to 0}\liminf_{n\to \infty}\phi_{n}\big(y_{\lambda,n}\big) \\
	&=\liminf_{n\to \infty}\phi_{n}\big(x_{n}\big) &&\leq \limsup_{n\to \infty}\phi_{n}\big( x_{n}\big)\\
	& &&=\lim_{\lambda \to 0}\limsup_{n\to \infty}\phi_{n}\big(y_{\lambda,n}\big)
	\leq \lim_{\lambda \to 0} \phi_U\big(y_{\lambda}\big)
	=\phi_U\big(x\big).
	\end{alignat*}
	The boundedness conditions follow similarly.
\end{proof}

\section{Variational representation of exponential integrals}\label{app_var}

Below we will give an extension of the classic variational formulation of exponential integrals. This is exploited in various arguments of the paper.

\begin{lm}\label{lm_entvar}
	Let $(X,\mathcal{A})$ be a measurable space and $V:X\to [0,\infty]$ be a non-negative $\mathcal{A}$-measurable function. Then 
	\begin{equation}\label{eq_gentv3}
	\log \int_X e^{V} d\mu = \sup_{\nu \in D} \left\{ \int_X V\,d\nu - R(\nu\|\mu) \right \},
	\end{equation}
	where
	\[
	D=\left\{ \nu \in \cP(X) \, \Big| \, \mbox{$R(\nu\|\mu)<\infty$} \right\}.
	\]
\end{lm}

In the case of Polish spaces $X$, the result also follows directly from \cite[Proposition 4.5.1]{Dupuis2013}, which deals with potentials $V$ that are either bounded from below or from above. Nevertheless, for completeness, we give here a direct and elementary proof. 
\begin{proof}
	We will first prove that for every $\nu\in D$, 
	\begin{equation}\label{eq_gentv5}
	\int_X V\, d \nu - R(\nu\|\mu)\leq \log \int_X e^{V} d \mu,
	\end{equation}
	and then we will show by approximation of bounded functions that 
	\begin{equation}\label{eq_gentv7}
	 \log \int_X e^{V} d\mu \le \sup_{\nu \in D} \left\{\int_X V\, d\nu - R(\nu\|\mu) \right \},
	\end{equation}
	which together will imply \eqref{eq_gentv3}.
	
	For the first inequality, we take any $\nu \in D$ and assume, without loss of generality, that $e^{V}$ is integrable with respect to $\mu$. In this case, Young's inequality (in the form $ab\le e^a -b +b\log b$) and the finiteness of $R(\nu\|\mu)<+\infty$ imply the finiteness of $\int_X V\,d\nu$. Let $\mu_V \in \cP(S)$ be defined by
	\[
	\mu_V = \frac{1}{Z_V} e^{V} \mu\qquad\text{with}\qquad Z_V= \int_X e^{V} d \mu.
	\]
	Now let us rewrite the expression on the left-hand side of \eqref{eq_gentv5} as follows,
	\begin{equation}\label{eq_gentv6}
	\begin{aligned}
	\int_X V\, d \nu - R(\nu\|\mu) &= \int_X V\, d \nu - \int_X \log \frac{d \nu}{d \mu} \, d\nu  \\
	&= \log \int_X e^{V} d \mu + \int_X \log e^{V} d \nu - \log \int_X e^{V} d \mu - \int_X \log \frac{d \nu}{d \mu} \, d\nu \\
	&= \log \int_X e^{V} d \mu - \int_X \log \left(\frac{d \nu}{d \mu }\left(\frac{d \mu_V}{d \mu}\right)^{-1} \right)d\nu \\
	&= \log \int_X e^{V} d \mu -R(\nu\|\mu_V)
	\end{aligned}
	\end{equation}
	By non-negativity of the entropy $R(\cdot\|\mu_V)$, \eqref{eq_gentv5} holds for any $\nu \in D$. 
	
	To show \eqref{eq_gentv7}, consider the sequence $V_n=\min\{V,n\}$ for $n\in\N$. Note that $V_n$ is a non-decreasing sequence of non-negative bounded functions converging pointwise to $V$. Since $V_n$ is bounded, we have that $R(\mu_{V_n}\|\mu)<+\infty$ and that $\int_X e^{V_n} d\mu<\infty$, and we define again $\mu_{V_n}$ as above. Hence, repeating the argument of \eqref{eq_gentv6} for $V_n$ it follows that 
	\begin{align*}
	\int_X V_n\, d \nu - R(\nu\|\mu) 
	&= \log \int_X e^{V_n}\, d \mu -R(\nu\|\mu_{V_{n}}).
	\end{align*}
	In particular, taking $\nu=\mu_{V_n}$, we get
	\begin{align*}
	\log \int_X e^{V_n}\, d \mu = \int_X V_n\, d \mu_{V_n} - R(\mu_{V_n}\|\mu) \le \int_X V\, d \mu_{V_n} - R(\mu_{V_n}\|\mu) 
	\end{align*}
	Maximizing over $n$, we conclude \eqref{eq_gentv7}.
\end{proof}

\begin{remark}
	In the proof for bounded $V$ the equality in \ref{eq_gentv3} follows by checking that $\nu=\mu_{V}$ is the unique maximizer \eqref{eq_gentv3}. However, this is not possible for a general unbounded $V$, even when $e^{V}$ is integrable with respect to $\mu$. A simple example over $X=[0,1/2]$ follows from taking $V(x)=-\log x + \alpha \log \log x^{-1} +1$ with any $1<\alpha<2$: for this example, $\mu_V$ is well defined and clearly $R(\mu_V\|\mu_V)=0$, but both $R(\mu_V\|\mu)$ and $\int V d\mu_V$ are infinite. In particular, $\mu_V$ does not belong to $D$ and is not a maximizer of \eqref{eq_gentv3}, even though $\mu_{V_n}$ is a maximizing sequence.
	
	In contrast, if $V$ satisfies a slightly stronger exponential integrability condition, then $R(\nu\|\mu_V)<\infty$ does imply $R(\nu\|\mu)<\infty$ and $\int V d\nu < \infty$, as we show below.
\end{remark}

\begin{cory}\label{cory_centvar}
Let $V:X\to \REx $ be a measurable function and $\mu_V=Z_V^{-1}e^{V}\mu$ with normalization constant $Z_V$. Further, suppose there exists $\gamma>1$ such that 
	\begin{equation*}
	\int_X e^{\gamma |V|} d\mu <\infty.
	\end{equation*}
Then 
\begin{equation*}
R(\nu\|\mu_V)<\infty \iff R(\nu\|\mu)<\infty.
\end{equation*}
Moreover, the following equality holds:
\begin{equation}\label{eq_centvar2}
R(\nu\|\mu_{V})=R(\nu\|\mu)-\int_X V  d\nu + \log \int_X e^{V} d\mu.
\end{equation}
\end{cory}

\begin{proof}
Repeating carefully the proof of \eqref{eq_gentv6}, one gets formula \eqref{eq_centvar2} if $V$ is in $L^1(\nu)$ and one of the two conditions $R(\nu\|\mu)<\infty$ and $R(\nu\|\mu_V)<\infty$ holds. Hence it remains to show that each of the conditions $R(\nu\|\mu)<\infty$ and $R(\nu\|\mu_V)<\infty$ implies that $V$ is in $L^1(\nu)$.

In the case $R(\nu\|\mu)<\infty$, by \eqref{eq_gentv5} we have
\begin{equation*}
	\int_X |V|\, d \nu \leq R(\nu\|\mu)+ \log \int_X e^{|V|} d \mu,
\end{equation*}
which is finite by assumption.

In the case $R(\nu\|\mu_V)<\infty$, we apply again \eqref{eq_gentv5} but with base measure $\mu_V$ instead of $\mu$ and $(\gamma-1)V$ instead of $V$, getting
\begin{equation*}
\begin{aligned}
	\int_X (\gamma-1)|V|\, d \nu &\leq  R(\nu\|\mu_V)+\log \int_X e^{(\gamma-1)|V|} d \mu_V\\
	&\leq  R(\nu\|\mu_V)+\log \int_X e^{\gamma |V|} d \mu-\log \int_X e^V d\mu,\\
\end{aligned}
\end{equation*}
which is finite by assumption. The proof is complete.
\end{proof}

\begin{lm}\label{lm_estrel}
	Let $F:X^k\to [0,\infty)$, $k\in\N$ be a nonnegative measurable function satisfying
	\[
	 \int_{X^k} \exp(F(x_1,\ldots,x_k))\,d \mu^{\otimes k} <\infty,
	\]
	and $\nu \in \cP(S)$ is such that $R(\nu\|\mu)<\infty$. Then
	\[
	 \log \int_X \exp\left(\int_{X^{k-1}} F(x,y)\,d\nu^{\otimes k-1}(y)\right)\,d\mu(x) \le (k-1)R(\nu\|\mu) + \log\int_{X^k} e^F d\mu^{\otimes k}.
	\]
\end{lm}
\begin{proof} A simple application of Lemma~\ref{lm_entvar} yields
    \begin{align*}
     &\log \int_X \exp\left( \int_{X^{k-1}} F(x,y)\,d\nu^{\otimes k-1}(y) \right) d\mu(x) \\
     &\hspace{4em} = \sup_{\rho} \biggl\{ \left\langle \rho,\int_{X^{k-1}} F(x,y)\,d\nu^{\otimes k-1}(y)\right\rangle - R(\rho\|\mu)\biggr\} \\
     &\hspace{4em} = \sup_{\rho} \biggl\{ \langle \rho\otimes \nu^{\otimes k-1}, F\rangle - R(\rho\otimes \nu^{\otimes k-1}\|\mu\otimes\mu^{\otimes k-1})\biggr\} + (k-1)R(\nu\|\mu) \\
     &\hspace{4em} \le \sup_{\sigma} \biggl\{ \langle \sigma, F\rangle - R(\sigma\|\mu^{\otimes k})\biggr\} + (k-1)R(\nu\|\mu) 
    = \log \int_{X^k}  e^F d\mu^{\otimes k} + (k-1)R(\nu\|\mu),
    \end{align*}
    which is the desired estimate.
\end{proof}

\section{On measurability and exponential approximations}\label{app_meas}

\subsection{Measurability of integral maps}

In this subsection we give a measurability result for integral maps. This in particular implies measurability of $\cE_V$ and $\cE_V^N$ in Section \ref{sgibbs}. In the following, $X$ and $Y$ are Polish spaces, we recall that $\cP(Y)$ endowed with the weak topology is also Polish; $X$, $Y$ and $\cP(Y)$ are endowed with their Borel $\sigma$-algebras.

\begin{thm}\label{thm_Borel_int_meas}
Given any Borel function
\begin{align*}
    X\times Y\times\cP(Y)\ni(x,y,\mu)\mapsto f(x,y,\mu)\in [-\infty,+\infty],
\end{align*}
then the set $\{\mu\in \cP(Y)\mid f\in L^1(\mu)\}$ is Borel and the mapping
\begin{align*}
    F_f:X\times\cP(Y)\ni(x,\mu)\mapsto 1_{f\in L^1(\mu)} \int_Y f(x,y,\mu)\,d\mu(y)
\end{align*}
is also Borel.
\end{thm}

The case when $f(x,y,\mu)=W(y)$ for some measurable $W$ is a classical question, and treated in great generalization in for example \cite[Chapter 8]{Bogachev2007II}. However, for simplicity, in our setting we stick to the case of metric spaces, and adapt an argument of \cite{Jorgensen2015} (Theorem 15.13). First, we provide a generalization of Lemma 7.3.12 of \cite{Dembo10}.
\begin{lm}\label{lm_intercon}
	Suppose that $f$ is in $C_b(X\times Y\times \cP(Y))$. Then $F_f$ is in $C_b(X\times \cP(Y))$.
\end{lm}
\begin{proof}
In the following, we denote $Z=X\times Y\times\cP(Y)$. The idea of the proof is similar to the fact that for Polish spaces $X$ and $Y$, the set of functions $\{f(x)g(y)\;|\;f\in C_b(X),g\in C_b(Y)\}$ is convergence determining for $\cP(X\times Y)$ (see for example Theorem 3.4.5b of \cite{Kurtz86}), which is used in Lemma 7.3.12 of \cite{Dembo10}. Namely, first note for any $g(x,y,\mu):=g_1(x)g_2(y)g_3(\mu)$, with $g_1\in C_b(X)$, $g_2\in C_b(Y)$ and $g_3\in C_b(\cP(Y))$, the boundedness and continuity of $F_g$ is trivial. 

Now, consider a sequence $(x_n,\mu_n)_n$ converging to $(x^*,\mu^*)$. In particular the subsets $L:=\{x_n\} \cup \{x^*\} \subset X$ and $M:=\{\mu_n\}_{n\ge 1} \cup \{\mu^*\} \subset \cP(Y)$ are compact in $X$ and $\cP(Y)$ respectively. In particular, the set $M\subset\cP(Y)$ is tight, i.e., for every $\epsilon>0$ there exists a compact set $K_{\epsilon}\subset Y$ such that $\mu^*(K^c)<\epsilon$ and $\mu_n(K^c_{\epsilon})<\epsilon$ for all $n\ge 1$.

Therefore, fix any $\epsilon>0$. By applying Stone-Weierstrass on the compact set $B_{\epsilon}=L\times K_{\epsilon}\times M$, we find a sequence $(g^{\epsilon,l})_{l\ge 1}\in C_b(B_{\epsilon})$ with $g^{\epsilon,l}(x,y,\mu):=g^{\epsilon,l}_1(x)g^{\epsilon,l}_2(y)g^{\epsilon,l}_3(\mu)$, where $g^{\epsilon,l}_1 \in C_b(L)$, $g^{\epsilon,l}_2 \in C_b(K_{\epsilon})$ and $g^{\epsilon,l}_3\in C_b(M)$, such that 
\[
\lim_{l\to \infty} \|f-g^{\epsilon,l}\|_{C_b(B_\epsilon)} =0. 
\]
Now also fix $l\in\N$. By Tietze's extension theorem we find a $\tilde g_2^{\epsilon,l} \in C_b(Y)$ such that 
\[
 \tilde g_2^{\epsilon,l}(y)=g_2^{\epsilon,l}(y)\;\;\text{on $K_{\epsilon}$},\qquad \|\tilde g_2^{\epsilon,l}\|_{C_b(Y)}\leq \|g_2^{\epsilon,l}\|_{C_b(K_\epsilon)}.
\]
Hence, define on $\tilde B=L\times Y\times M$ the continuous function $\tilde g^{\epsilon,l}(x,y,\mu):=g^{\epsilon,l}_1(x)\tilde g^{\epsilon,l}_2(y)g^{\epsilon,l}_3(\mu)$, and note that by construction $\|\tilde g^{\epsilon,l}\|_{C_b(\tilde B)}\leq \|g^{\epsilon,l}\|_{C_b(B_\epsilon)}$.

We now compute for any $l,n\ge 1$,
\[
\begin{aligned}
|F_f(x^*,\mu^*) -F_f(x_n,\mu_n)| &= \left| \int_Y f(x^*,y,\mu^*)\,d\mu^*(y)-\int_Y f(x_n,y,\mu_n)\,d\mu_n(y)\right|\\
&\hspace{-2em}\leq \left| \int_{K_{\epsilon}} f(x^*,y,\mu^*)\,d\mu^*(y)-\int_{K_{\epsilon}} f(x_n,y,\mu_n)\,d\mu_n(y)\right|+ 2 \epsilon \|f\|_{C_b(Z)} \\
&\hspace{-2em}\leq  \left| \int_{K_{\epsilon}} g^{\epsilon,l}(x^*,y,\mu^*)\,d\mu^*(y)-\int_{K_{\epsilon}} g^{\epsilon,l}(x_n,y,\mu_n)\,d\mu_n(y)\right|\\
&\hspace{7em}+2 \|f-g^{\epsilon,l}\|_{C_b(B_{\epsilon})}+2\epsilon \|f\|_{C_b(Z)}\\
&\hspace{-2em}\leq  \left| \int_{Y} \tilde g^{\epsilon,l}(x^*,y,\mu^*)\,d\mu^*(y)-\int_{Y} \tilde g^{\epsilon,l}(x_n,y,\mu_n)\,d\mu_n(y)\right|\\
&\hspace{4em}+2\epsilon \|g^{\epsilon,l}\|_{C_b(B_{\epsilon})}+2\|f-g^{\epsilon,l}\|_{C_b(B_{\epsilon})}+2\epsilon \|f\|_{C_b(Z)}.
\end{aligned}
\]
Thus, by the continuity of $F_{\tilde g^{\epsilon,l}}$ on $\tilde B$, 
\[
\limsup_{n\to \infty} |F_f(x^*,\mu^*) -F_f(x_n,\mu_n)| \leq 2 \epsilon \|g^{\epsilon,l}\|_{C_b(B_\epsilon)}+2 \|f-g^{\epsilon,l}\|_{C_b(B_\epsilon)}+2\epsilon \|f\|_{C_b(Z)}.
\]
Taking subsequent limits, first in $l\to \infty$ and then in $\epsilon\to 0$ we conclude the proof. 	
\end{proof}
Next, we paraphrase Theorem 4.33 of \cite{Jorgensen2015}. 
\begin{thm}[Monotone class theorem]\label{lm_quickt}
	Let $X$ be a metrizable space, and let $\cF$ be a vector subspace of $B_b(X)$ including $C_b(X)$. Then $\cF=B_b(X)$ if and only if $\cF$ is closed under monotone sequential pointwise limits in $B_b(X)$.  
\end{thm}

\begin{proof}[Proof of theorem \ref{thm_Borel_int_meas}]
    Let $\cF$ be the set of bounded Borel-measurable functions given by
	\[
	\cF = \Bigl\{ f \in B_b(X\times Y\times \cP(Y))\; \Big|\; F_f:X\times \cP(Y)\to \R \mbox{ is Borel-measurable}  \Bigr\}.
	\]
	It is clear that $\cF$ is a vector subspace of $B_b(X\times Y\times \cP(Y))$, and by Lemma \ref{lm_intercon} it contains $C_b(X\times Y\times \cP(Y))$. To show stability under monotone pointwise limit, consider any sequence $f_n$ in $\cF$ with $f_n\uparrow f$ and $f$ in $B_b(X\times Y\times \cP(Y))$. For any measure $\mu \in \cP(Y)$, it follows from monotone convergence that
	\[
	\lim_{n\to \infty} F_{f_n}(x,\mu) = \lim_{n\to \infty} \int_{Y} f_n(x,y,\mu) \, d \mu(y) = \int_{Y} f_n(x,y,\mu) \, d \mu(y) = F_f(x,\mu).
	\]
	Hence by the monotone class theorem, we get that $\cF=B_b(X\times Y\times \cP(Y))$, that is $F_f$ is Borel for any Borel bounded function $f$.
	
	For a Borel non-negative function $f$, it follows from approximation with bounded function and monotone convergence that
	\begin{align*}
	    (x,\mu) \mapsto \int_Y f(x,y,\mu)\, d\mu(y)
	\end{align*}
	is Borel, in particular the set $\{\mu \mid f\in L^1(\mu)\}$ is Borel. The case of a general Borel function $f$ follows splitting $f$ into $f^+$ and $f^-$. The proof is complete.
\end{proof}

\begin{cory}\label{thm_minteru}
	Let $S$ be a Polish space and let $V:S^k\to \REx$ be Borel-measurable. Then both $\cE_V, \cE^N_V:\cP(S)\to \REx$ are also Borel-measurable. 
\end{cory}

\subsection{Exponential approximation theorem in Polish spaces}

\begin{thm}\label{thm_qap1}
	Suppose $X$ is Polish, $V:X\to\REx$ is Borel measurable such that
	\begin{equation}\label{eq_appol1}
	\int_X e^{\beta |V|} \, d\mu <+\infty\qquad\text{for any $\beta\ge 0$.}\\
	\end{equation}
	Then there exists continuous bounded functions $V_{\lambda}:S\to \R$ such that
	\[
	\lim_{\lambda\to 0}\log \int_X e^{\beta |V-V_{\lambda}|} \, d\mu=0\qquad\text{for any $\beta\ge 0$.}
	\]
\end{thm}

The argument is similar to the denseness of $C_b(S)$ in $L^p(S)$ spaces, and rely on Tietze's extension theorem and Lusin's Theorem, see for example \cite[Theorems 2.47 and 12.8]{Jorgensen2015} (note that any Polish space $S$ is normal, and so these theorems apply).

\begin{proof}[Proof of Theorem \ref{thm_qap1}]
First, suppose that $V$ is bounded. Then by Lusin's theorem, for any $\epsilon>0$, there exists a compact set $K_{\epsilon}$ such that $\mu(S\setminus K_{\epsilon})<\epsilon$.
	
By Tietze's extension theorem, we obtain a continuous function $V_{\epsilon}$ such that $V_{\epsilon}=V$ on the compact set $K_{\epsilon}$, and satisfying $\sup_{x\in S} |V_{\epsilon}(x)| \leq \sup_{x\in S} |V(x)|$.
	Therefore, for any $\beta \geq 0$,
	\begin{align*}
	\int_X e^{\beta |V-V_{\epsilon}|} \, d\mu &= \int_{K_{\epsilon}} e^{\beta |V-V_{\epsilon}|}\, d\mu+\int_{S\setminus K_{\epsilon}} e^{\beta |V-V_{\epsilon}|}\, d\mu\\
	&=\int_{K_{\epsilon}} e^{0}\, d\mu+\int_{S\setminus K_{\epsilon}} e^{\beta |V-V_{\epsilon}|}\, d\mu \leq 1+\epsilon\cdot e^{2\beta \sup_{x\in S}|V|(x)}.
	\end{align*}
	Note that the choices $\epsilon, K_{\epsilon}, V_{\epsilon}$ are independent of $\beta$, and thus, by continuity of the logarithm,
	\[
	\lim_{\epsilon\to 0} \log \int_X e^{\beta |V-V_{\epsilon}|}\,d\mu =0.
	\]
Next, for the case of unbounded $V$, suppose \eqref{eq_appol1} holds. Let $V_n$ be the bounded truncation of $V$ to the interval $[-n,n]$, i.e. 
\[
V_{n} := \min \{ n, \, \max \{ -n, V\}\},\qquad n\in\N.
\]
It is clear that $V_n$ is bounded and $\lim_{n\to\infty}|V-V_n|(x)\to 0$ pointwise in $x\in X$. Moreover, $|V-V_n|\leq |V|$ for all $n\in\N$. Hence, by the assumption on $V$ and the dominated convergence,
\begin{align*}
\lim_{n \to \infty} \int_X e^{\beta|V-V_n|} \, d\mu 
=1\qquad\text{for any $\beta \geq 0$},
\end{align*}
which implies that
\[
\lim_{n \to \infty} \log \int_X e^{\beta|V-V_n|} \, d\mu=0\qquad\text{for any $\beta \geq 0$}.
\]
From the previous argument on bounded functions, we obtain, for each $n\in\N$, a sequence of bounded continuous functions $(V_{n,\epsilon})_{\epsilon>0}$ such that 
\[
\lim_{\epsilon \to 0} \log \int_X e^{\beta|V_n-V_{n,\epsilon}|} \, d\mu=0\qquad \text{for any $\beta \geq 0$}.
\]
Finally, since by convexity,
\[
\log \int_X e^{\beta|V-V_{n,\epsilon}|} \, d\mu \leq \frac{1}{2}\log \int_X e^{2\beta|V-V_{n}|} \, d\mu+\frac{1}{2}\log \int_X e^{2\beta|V_n-V_{n,\epsilon}|} \, d\mu,
\]
we can find an appropriate sequence $V_{\lambda}:=V_{n(\lambda),\epsilon(\lambda)}$ such that
\[
\lim_{\lambda\to 0} \log \int_X e^{\beta|V-V_{\lambda}|} \, d\mu=0\qquad\text{for any $\beta\geq 0$},
\]
thereby concluding the proof.
\end{proof}

\section{Stochastic estimates and technical parts of Section \ref{s_process}}
\numberwithin{equation}{subsection}

\subsection{Basic facts on Girsanov theorem and Novikov condition}

\begin{thm}[Girsanov]\label{thm_girs}
Let $W$ be an $m$-dimensional Brownian motion on a filtered probability space (satisfying the standard assumption), with general initial law $\law(W_0)$, let $b:[0,T]\times\R^m\rightarrow\R^m$ be a Borel function. Consider the SDE
\begin{equation}\label{SDEsg}
\begin{aligned}
		d X_t &= b(t,X_t) \, d t+ d W_t, \qquad t \in [0,T],\\
		\law(X_0) &= \law(W_0).
\end{aligned}
\end{equation}
If
\begin{equation}\label{eq_Novikov}
\E \left[ \, e^{\frac{1}{2}\int_0^T |b(t,W_t)|^2 \, d t}  \right]< \infty,
\end{equation}
then there exists a weak solution to \eqref{SDEsg}. Its law $\tilde P^X$ is equivalent to the Wiener measure $\mathbb{W}$ (starting with the same initial law of $X_0$) and satisfies,
\begin{equation*}
\frac{d P^X}{d \mathbb{W}}(W):=\exp\left(-\frac{1}{2}\int_0^T |b(t,W_t)|^2 \, d t + \int_0^T b(t,W_t)\cdot d W_t\right).
\end{equation*}
Moreover, if $Y$ is another weak solution to \eqref{SDEsg} (defined possibly on another probability space satisfying the standard assumption), with law $\tilde P^Y$, such that
\begin{align}
    \int_0^T |b(t,Y_t)|^2 dt <\infty \quad \tilde P^Y\text{-a.s.},\label{eq_Girs_constraint}
\end{align}
then $\tilde P^Y$ coincides with $\tilde P^X$.
\end{thm}
	
This result is classical, here we recall uniqueness, in the line of \cite{Fedrizzi2009}, Section 3.
	
\begin{proof}
Existence and the representation formula are a classical consequence of Girsanov theorem, which can be applied thanks to Novikov condition \eqref{eq_Novikov}. When the initial distribution $\rho_0$ is a Dirac delta, uniqueness follows from \cite[Theorem 7.7]{LipShi2001}. Uniqueness in the general case follows from conditioning $X_0$ to be a single point.
\end{proof}

\begin{lm}\label{lm_estnk}

Let $W$ be an $m$-dimensional Brownian motion and let $b:\R^m\to\R$ be a Borel function such that
\begin{equation*}
\E \left[ \, e^{2\int_0^T |b(t,W_t)|^2 dt }  \right]< \infty.
\end{equation*}
Then 
\begin{equation*}
\E \left[ \, e^{\int_0^T b(t,W_t)\cdot d W_t}  \right]\le  \E \left[ \, e^{2\int_0^T |b(t,W_t)|^2 dt }  \right]^{\frac{1}{2}}
\end{equation*}
\end{lm}
\begin{proof}
The proof is classical, we sketch the idea: By Novikov criterion, the exponential local martingale
\begin{align*}
    \exp\left(-\frac{1}{2}\int_0^T |2 b(t,W_t)|^2 dt  + \int_0^T (2 b)(t,W_t)\cdot d W_t\right)
\end{align*}
is a true martingale, in particular it has expectation $1$. Then it is enough to apply H\"older inequality to get the required estimate.
\end{proof}

\subsection{Definition of the log-densities in Subsection \ref{s_proc_expl}}
\label{s_def_log_densities}

Here we give the precise definition of $\cE^N_b$ and $\cE_b$ in Subsection \ref{s_proc_expl}.

We recall that $S=C([0,T];\R^d)$ and $\mathbb{W}$ is the Wiener measure on $S$. For $b:[0,T]\times\R^d\times\cP(\R^d) \rightarrow \R^d$ Borel function, we define
\begin{align*}
V_b^1(x,\mu) := \int_0^T |b_t(x_t,\mu_t)|^2 dt,\quad x\in S,\,\mu\in \cP(S).
\end{align*}
Note that $V^1_b$ is a Borel map (by Fubini theorem applied to the map $(t,x,\mu)\mapsto |b_t(x_t,\mu_t)|^2$). For $\mu$ in $\cP(S)$, if $V^1_b(x,\mu)$ is finite for $\mathbb{W}$-a.e. $x$, then we can define the stochastic integral
\begin{align}
\int_0^T b_t(x_t,\mu_t) \cdot dx_t,\label{eq_stoch_int_b}
\end{align}
on the space $(S,(\cG_t)_t,\mathbb{W})$, where $\cG_t$ is the $\sigma$-algebra generated on $S$ by the $\mathbb{W}$-negligible sets and by the projection $\pi_{[0,t]}:S\rightarrow S_t=C([0,t];\R^d)$ on $[0,t]$; in particular $\cG_T$ is the completion of $\cB(S)$ with respect to $\mathbb{W}$. Hence we can define a map $V^2_b(\cdot,\mu):S\rightarrow \R$ which is a representative of the stochastic integral \eqref{eq_stoch_int_b} and is measurable with respect to $\cB(S)$.

Now we define
\begin{align*}
\cE_b(\mu) := \begin{cases}
    \int_S \left(\frac12 V^1_b(x,\mu) -V^2_b(x,\mu)\right) d\mu & \text{if }V^1_b(\cdot,\mu),V^2_b(\cdot,\mu) \in L^1(S,\mu), \\
    0 & \text{otherwise}
\end{cases}\qquad \mu\in\cP(S)
\end{align*}
(for $\mu$ which is absolutely continuous with respect to $\mathbb{W}$, $\cE(\mu)$ does not depend on the specific choice of $V^2_b(\cdot,\mu)$). Note that, if $R(\mu\|\mathbb{W})<\infty$ and $\E [e^{\frac12 V^1_b(W,\mu)}] <\infty$, then $V^1_b(\cdot,\mu)$ and $V^2_b(\cdot,\mu)$ are in $L^1(\mu)$: indeed, by Lemma \ref{lm_entvar},
\begin{align*}
&\frac12 \int_S V^1_b(x,\mu)\, d\mu \le R(\mu\|\mathbb{W}) + \log \int_S e^{\frac12 V^1_b(x,\mu)} d\mathbb{W} <\infty,
\end{align*}
and, by Lemmas \ref{lm_entvar} and \ref{lm_estnk} (using $e^{|a|}\le e^a+e^{-a}$),
\begin{align*}
\frac12 \int_S |V^2_b(x,\mu)| d\mu &\le R(\mu\|\mathbb{W}) + \log \left(\int_S e^{\frac12 V^2_b(x,\mu)} d\mathbb{W} + \int_S e^{-\frac12 V^2_b(x,\mu)} d\mathbb{W} \right)\\
&\le R(\mu\|\mathbb{W}) + \log \left(2\int_S e^{\frac12 V^1_b(x,\mu)} d\mathbb{W}\right)<\infty.
\end{align*}
In Lemma \ref{lm_Borel_cE} we show that $\cE_b$ is Borel (at least on the set $\{\mu\mid R(\mu\|\mathbb{W})<\infty\}$).

Coming to $\cE^N_b$, we recall that $W^i$, $1\le i\le N$, are independent $d$-dimensional Brownian motions on some filtered probability space $(\Omega,(\cF_t)_t,\mathbb{P})$ (under the standard assumption) and $z^N_{\boldsymbol{W}}$ is the empirical measure associated with $W^i$. We assume on $b$ that
\begin{align*}
\int_S V^1_b(x,z^N_{\boldsymbol{W}})\, dz^N_{\boldsymbol{W}}(x) <\infty \qquad \mathbb{P}\text{-a.s.}.
\end{align*}
Under this assumption, we can define the stochastic integral
\begin{align}
\int_0^T b_t(x_t,\mu_t) \cdot dx_t\label{eq_stoch_int_b_N}
\end{align}
on the space $(\bar{\Omega},(\cH_t)_t,\bar{P})$. Here $\bar{\Omega}=S\times \cP(S)$ and $\bar{P}$ is the law of $W^1\otimes z^N_{\boldsymbol{W}}$, or equivalently of $W^i\otimes z^N_{\boldsymbol{W}}$ for any $1\le i\le N$, under $\mathbb{W}$. Also $\cH_t$ is the $\sigma$-algebra on $\bar{\Omega}$ generated by the $\bar{P}$-negligible sets and by $\pi_{[0,t]}\otimes (\pi_{[0,t]})_\#$, where $\pi_{[0,t]}:S\rightarrow S_t=C([0,t];\R^d)$ is the projection on time $[0,t]$ and $(\pi_{[0,t]})_\#:\cP(S) \rightarrow \cP(S_t)$ is the corresponding image measure map; in particular, $\cH_T$ is the completion under $\bar{P}$ of $\cB(S)\otimes \cB(\cP(S))$. Hence we can define a map $V^{2,N}_b:S\times \cP(S)\rightarrow \R$ which is a representative of the stochastic integral \eqref{eq_stoch_int_b_N} and is measurable with respect to $\cB(S)\otimes \cB(\cP(S))$. Note that, for every $1\le i\le N$,
\begin{align}
V^{2,N}_b(W^i,z^N_{\boldsymbol{W}}) = \int_0^T b_t(W^i_t,z^N_{\boldsymbol{W},t})\, dW^i_t \qquad \mathbb{P}\text{-a.s.}\label{eq_stoch_int_def}
\end{align}
and that, $\mathbb{P}$-a.s., $V^{2,N}_b(\cdot,z^N_{\boldsymbol{W}})$ is in $L^1(z^N_{\boldsymbol{W}})$.

Now we define
\begin{align*}
\cE^N_b(\mu) :=
\begin{cases}
    \int_S \left(\frac12 V_b^1(x,\mu) -V^{2,N}_b(x,\mu)\right) d\mu & \text{if }V^1_b(\cdot,\mu),V^{2,N}(\cdot,\mu)\in L^1(S,\mu),\\
    0 & \text{otherwise}
\end{cases}
\qquad \mu\in\cP(S).
\end{align*}
By Theorem \ref{thm_Borel_int_meas}, the function $\cE^N_b$ is Borel.

\begin{lm}\label{lm_Borel_cE}
Assume that $\E [e^{\beta V^1_b(W,\mu)}] <\infty$ for every $\beta>0$. Then the map 
\[
    \cP(S)\ni \mu\mapsto \cE_b(\mu)1_{R(\mu\|\mathbb{W})<\infty}\quad\text{is Borel.}
\]
\end{lm}

\begin{proof}
The map
\begin{align*}
\mu\mapsto \int_S \frac12 V^1_b(x,\mu)\, d\mu
\end{align*}
is Borel by Theorem \ref{thm_Borel_int_meas}, so it is enough to show Borel measurability of
\begin{align*}
F:\mu\mapsto \left(\int_S \frac12 V^2_b(x,\mu)\,d\mu \right)1_{R(\mu\|\mathbb{W})<\infty} = \E \left[\int_0^T b_t(x_t,\mu_t)\cdot dW_t\, \frac{d\mu}{d\mathbb{W}}(W)\right] 1_{R(\mu\|\mathbb{W})<\infty}.
\end{align*}

We start with the case of $b$ in $C_b([0,T]\times \R^d\times \cP(\R^d))$ ($\cP(\R^d)$ being endowed with the weak topology). We take a sequence $\Pi^n$ of partitions $0=t_0< t_1<\ldots <t_n=T$ on $[0,T]$ with size tending to $0$ in $n$. For each $n$, we call $b^n:[0,T]\times S \times \cP(S)\rightarrow \R$ the Borel function
\begin{align*}
b^n(t,x,\mu) = \sum_i b(t_i,x_{t_i},\mu_{t_i}) 1_{t\in[t_i,t_{i+1})}
\end{align*}
and $I^n:S\times \cP(S) \rightarrow \R$ the Borel function defined by
\begin{align*}
I^n(\gamma,\mu) = \sum_i b(t_i,\gamma_{t_i},\mu_{t_i}) \cdot (\gamma_{t_{i+1}}-\gamma_{t_i}).
\end{align*}
Again by Theorem \ref{thm_Borel_int_meas}, the map
\begin{align*}
F_n:\mu \mapsto \E [I^n(W,\mu)\frac{d\mu}{d\mathbb{W}}(W)]1_{R(\mu\|\mathbb{W})<\infty}
\end{align*}
is Borel. Now, for each $\mu$ with $R(\mu\|\mathbb{W})<\infty$, we have by Lemmas \ref{lm_entvar} and \ref{lm_estnk}, for every $\beta>0$,
\begin{align*}
\beta|F(\mu)-F_n(\mu)|&\le \log \E[e^{\beta|V_b^2(W,\mu)-I^n(W,\mu)|}]+R(\mu\|\mathbb{W})\\
&= \log \E[e^{\beta|\int_0^T (b_t(W_t,\mu_t)-b^n(t,W,\mu)) \cdot dW_t|}]+R(\mu\|\mathbb{W})\\
&\le \log \left( 2\E[e^{2\beta^2\int_0^T |b_t(W_t,\mu_t)-b^n(t,W,\mu)|^2 \cdot dt}] \right) +R(\mu\|\mathbb{W}).
\end{align*}
Since $b$ is continuous, $ b_t(W_t,\mu_t)-b^n(t,W,\mu)$ tends to $0$ for every $t$ in $[0,T)$ and every $W$ and $\mu$. Hence, for every fixed $\beta>0$, by dominated convergence theorem and boundedness of $b$, $\E[e^{2\beta^2\int_0^T |b_t(W_t,\mu_t)-f^n(t,W,\mu)|^2 \cdot dt}]$ tends to $1$ and so
\begin{align*}
\limsup_n |F(\mu)-F_n(\mu)|\le \frac{1}{\beta} (\log 2 +R(\mu\|\mathbb{W})).
\end{align*}
By arbitrariness of $\beta$, $F$ is the pointwise limit of the Borel functions $F_n$, hence $F$ is Borel (for $b$ continuous and bounded).

The case of $b$ Borel bounded follows from the case of $b$ continuous bounded via a monotone class argument (cf.\ Theorem \ref{lm_quickt}): the stability assumption needed for the monotone class theorem can be verified as in the proof of convergence of $F_n$ to $F$. Finally, the case of general $b$ (satisfying $\E[e^{\beta V^1_b(W,\mu)}]<\infty$ for every $\beta$) follows approximating $b$ with bounded $b^n$ and proceeding as in the proof of convergence of $F_n$ to $F$. The proof is complete.
\end{proof}

\subsection{Proof of Lemma \ref{lm_prep1} and relative entropy representation}\label{s_proof_SDE_Gibbslike}

In this subsection, we assume the setting at the beginning of Section \ref{s_process}.

\begin{proof}[Proof of Lemma \ref{lm_prep1}]
The SDE \eqref{eq_SDEm} is an SDE on $\R^{d N}$ for the vector $\boldsymbol{X}^N$, where the $i$-th component of the drift is $(t,\x)\mapsto b^N_t(x_i,z^N_\x)$. Note that, for this SDE, Novikov condition is satisfies, indeed
\begin{align*}
\E \left[e^{\frac12 \sum_{i=1}^N\int_0^T |b^N_t(W^i_t,z^N_{\boldsymbol{W},t})|^2 dt}\right] = \E\left[\exp\left(\frac{N}{2} \int_S \int_0^T |b^N_t(x_t,z^N_{\boldsymbol{W},t})|^2 dt\, dz^N_{\boldsymbol{W}}(x)\right)\right] <\infty.
\end{align*}
Girsanov theorem gives then the existence of a weak solution $\boldsymbol{X}$, with law $\tilde{Q}^N_{b^N}$. The uniqueness in law condition \eqref{eq_Girs_constraint} reads here
\begin{align*}
\sum_{i=1}^N \int^T_0 |b^N_t(x^i_t,z^N_{\x,t})|^2 dt <\infty \quad \tilde{Q}^N_{b^N}\text{-a.s.},
\end{align*}
which is equivalent to \eqref{eq_particle_uniq}. The representation formula of the law $\tilde{Q}^N_{b^N}$ in Girsanov theorem reads here (recall the definition of $V^1_b$ and $V^{2,N}_b$ in \eqref{eq_V1_V2})
\begin{align*}
\frac{d \tilde{Q}_{b^N}^N}{d \tilde{P}^N}(\boldsymbol{W}) &=\exp\left(-\frac{1}{2}\sum_{i=1}^N\int_0^T |b^N_t(W^i_t,z^N_{\boldsymbol{W},t})|^2 \, d t + \sum_{i=1}^N\int_0^T b^N_t(W^i_t,z^N_{\boldsymbol{W},t})\cdot d W^i_t\right)\\
&=\exp\left(-\frac{1}{2}N \int_{S}  V^1_{b^N} (x,z^N_{\boldsymbol{W}}) dz^N_{\boldsymbol{W}} (x) + N \int_{S}  V^{2,N}_{b^N} (x,z^N_{\boldsymbol{W}}) dz^N_{\boldsymbol{W}} (x)\right)\\
&= \exp\left(-N \cE^N_{b^N}(z^N_{\boldsymbol{W}})\right),
\end{align*}
where we used \eqref{eq_stoch_int_def} for the stochastic integral. The first formula in \eqref{eq_Girs_density_particle} is proved. The second formula (for the law $Q^N_{b^N}$ of the empirical measure) follows from the first one by a standard argument from measure theory.
\end{proof}

\begin{lm}[Relative entropy representation of $\cE$]\label{lm_prep2}
Assume that
\begin{align}
\E \left[e^{\frac12 \int^T_0 |b_t(W_t,\mu_t)|^2 dt}\right]<\infty \quad \forall\mu\text{ with }R(\mu\|\mathbb{W}) <\infty.\label{eq_Novikov_McKVla}
\end{align}
Then we have the following representation formula:
\begin{align}
R(\mu\|\mathbb{W}) + \cE_{b}(\mu)=
 \begin{cases}
    R(\mu\|\mathbb{W}^{\mu}) & \mbox{if $R(\mu\|\mathbb{W})<\infty$} \\
    + \infty, & \mbox{otherwise},
\end{cases}\label{eq_entropy_rep}
\end{align}
where $\mathbb{W}^{\mu}$ is the law of the process $X^{\mu}$ satisfying the SDE
\begin{equation*}
d X^{\mu}_t = b_t(X^{\mu}_t,\mu_t) \, d t + d W_t
\end{equation*}
with initial law $\rho_0$ (the law $\mathbb{W}^{\mu}$ exists and is uniquely determined by Girsanov theorem \ref{thm_girs}).

Moreover, when $b$ is in the class $\FLip$, the restriction $R(\mu\|\mathbb{W})<\infty$ in \eqref{eq_entropy_rep} may be dropped.
\end{lm}

\begin{proof}
	When $R(\mu\|\mathbb{W})=\infty$, the representation formula holds trivially (recall that $\cE_b(\mu)$ is finite for every $\mu$). Now fix $\mu$ with $R(\mu\|\mathbb{W})<\infty$. By the condition \eqref{eq_Novikov_McKVla}, we can apply Girsanov theorem, which gives the formula
	\begin{align}\label{eq_entropyg1}
	\begin{aligned}
    \frac{d \mathbb{W}^{\mu}}{d\mathbb{W}}(W)
    &= \exp\left(-\frac{1}{2}\int_0^T |b_t(W_t,\mu_t)|^2 \, d t + \int_0^T b_t(W_t,\mu_t)\cdot d W_t\right) \\
    &= \exp\left(-\frac12 V^1_b(W,\mu) +V^2_b(W,\mu)\right).
    \end{aligned}
	\end{align}
	In particular $\mathbb{W}$ and $\mathbb{W}^{\mu}$ are equivalent and so $\mu$ is absolutely continuous also with respect to $\mathbb{W}^{\mu}$. Hence we can compute the relative entropy
	\begin{align*}
	R(\mu\|\mathbb{W}^{\mu}) &= \int \log \frac{d \mu}{d \mathbb{W}^{\mu,b}} d \mu = \int \log \frac{d \mu}{d \mathbb{W}} d \mu - \int \log \frac{d \mathbb{W}^{\mu}}{d \mathbb{W}} d\mu\\
	&= R(\mu\|\mathbb{W}) - \int \left(-\frac12 V^1_b(W,\mu) +V^2_b(W,\mu)\right) d\mu(W)\\
	&= R(\mu\|\mathbb{W}) +\cE_b(\mu),
	\end{align*}
	which is the desired representation formula.
	
	For $b$ in $\FLip$ we have to prove that for every $\mu$, 
	\begin{equation}\label{eq_entropyg2}
	R(\mu\|\mathbb{W}^{\mu,b})<\infty \iff R(\mu\|\mathbb{W})<\infty.    
	\end{equation}
	Note that $\mathbb{W}^{\mu,b}$ is well-defined for every $\mu$ when $b$ is in $\FLip$. Now, fix any $\mu$, and note that in particular $b$ is bounded and hence Girsanov's formula \eqref{eq_entropyg1} still holds. Moreover, for every $\beta>0$ and $i=1,2$, by boundedness of $b$ and Lemma \ref{lm_estnk} we have,
	\begin{equation*}
	\E \left[e^{\beta V^i_b(W,\mu)}\right]<\infty.
	\end{equation*}
	Applying Corollary \ref{cory_centvar} to the measures $\mathbb{W}$ and $\mathbb{W}^{\mu,b}$ we easily deduce \eqref{eq_entropyg2}.
	\end{proof}
	
\subsection{\texorpdfstring{$L^q_t(L^p_x)$}{Lpq}-estimates}

Khasminskii Lemma is classical, see for example \cite{Kha1959}, \cite[Chapter 1, Lemma 2.1]{Szn1998}, \cite[Lemma 13]{FedFla2011}.

\begin{lm}[Khasminskii Lemma]\label{lm_estkha}
Let $W$ be a $d$-dimensional Brownian motion starting from $0$, let $f:[0,T]\times \R^d\rightarrow \R$ be a non-negative Borel function and assume that
\begin{align*}
    \alpha_f:= \sup_{x\in \R^d} \E \int_0^T f(t,x+W_t)\, dt <1.
\end{align*}
Then it holds
\begin{align*}
    \sup_{x\in \R^d}\E \left[\exp\left[ \int_0^T f(t,x+W_t)\, dt \right]\right] \leq
    \frac{1}{1-\alpha_f}.
\end{align*}
\end{lm}

\begin{lm}[$L^q_t(L^p_x)$ estimates]\label{lm_estlpq}
Let $W$ be a $d$-dimensional Brownian motion starting from $0$. Take $1\le p,q \le \infty$ satisfying
\begin{align}
    \frac{d}{p}+\frac{2}{q}<2.\label{eq_Lpq_hp}
\end{align}
Then there exists a constant $C$ (depending on $p,q,d$ and $T$) such that, for every $f:[0,T]\times \R^d\rightarrow \R$ non-negative Borel function,
\begin{align}
    \sup_{x\in \R^d} \int_0^T \E[f(t,x+W_t)]\,dt \le C \|f\|_{L^q_t(L^p_x)}.\label{eq_estlpq}
\end{align}
\end{lm}

This bound is classical (see e.g. \cite[Lemma 11]{FedFla2011}), with an elementary proof that we recall here.

\begin{proof}
We use H\"older's inequality applied at $t$ and $x$ fixed for the convolution with the Gaussian density $p_t$:
\begin{align*}
    \E[f(t,x+W_t)]= f_t\star p_t(x) \le \|f_t\|_{L^p_x}\|p_t\|_{L^{p'}_x}
\end{align*}
We recall that $\|p_t\|_{L^{p'}_x} \le ct^{-d/2p}$ for some constant $c$ depending on $p$ and $d$ (as one can see via the change of variable $x'=t^{-1/2}x$). Therefore, for every $x$, we get by H\"older's inequality,
\[
    \int_0^T f(t,x+W_t)\,dt \leq
    c \int_0^T \|f_t\|_{L^p_x} t^{-d/2p}\,dt \le
    c \|f\|_{L^q_t(L^p_x)} \left( \int_0^T t^{-dq'/2p}\, dt \right)^{1/q'}\quad \text{for $q<\infty$}
\]
(note that \eqref{eq_Lpq_hp} implies $q>1$, so $q'<\infty$). As for $q=\infty$, we estimate similarly,
\[
    \int_0^T f(t,x+W_t)\,dt \leq
    c \|f_t\|_{L_t^\infty(L^p_x)}\int_0^T  t^{-d/2p}\,dt.
\]
Now the assumption on $p$ and $q$ is equivalent to $dq'/2p<1$ for $q<\infty$ and to $d<2p$ for $q=\infty$. Hence the time integral of $t^{-dq'/2p}$ is finite. Hence the bound \eqref{eq_estlpq} holds with $C=c\left(\int_0^T t^{-dq'/2p} dt\right)^{1/q'}$. The proof is complete. 
\end{proof}

The previous bound can be easily generalized to the case of $k$ independent Brownian motions, as in the following:

\begin{lm}\label{lm_estlp12q}
Let $W^1,\ldots W^k$ be $k$ independent $d$-dimensional Brownian motions starting from $0$. Take $1\le p_1,\ldots p_k,q \le \infty$ satisfying
\begin{align*}
    \frac{d}{p_1}+\ldots \frac{d}{p_k}+\frac{2}{q}<2.
\end{align*}
Then there exists a constant $C$ (depending on $p_1,\ldots p_k,q,d$ and $T$) such that, for every $f:[0,T]\times \R^{kd} \rightarrow \R$ non-negative Borel function,
\begin{align}
    \sup_{x_1,\ldots x_k\in \R^d} \int_0^T \E[f(t,x_1+W^1_t,\ldots x_k+W^k_t)] dt \le C \|f\|_{L^q_t(L^{p_1}_{x_1}(\ldots (L^{p_k}_{x_k})\ldots))}.\label{eq_Lp_bounds_k}
\end{align}
More generally, one can replace the above right-hand side by $\|f\|_{L^q_t(L^{p_1}_{x_{\sigma(1)}}(\ldots (L^{p_k}_{x_{\sigma(k)}})\ldots))}$ for any permutation $\sigma$ of $\{1,\ldots k\}$.
\end{lm}

\begin{proof}
The proof is similar to the previous one. We write
\begin{align*}
\E[f(t,x_1+W^1_t,\ldots x_k+W^k_t)]= f_t\star p_t^{\otimes k}(x_1\ldots x_k)
\end{align*}
and use H\"older inequality in the $x_k$ variable, to get
\begin{align*}
\E[f(t,x_1+W^1_t,\ldots x_k+W^k_t)] \le c\|f_t(x_1,\ldots x_{k-1},\cdot)\|_{L^{p_k}_{x_k}} t^{-d/2p_k} p_t^{\otimes k}(x_1,\ldots x_{k-1}).
\end{align*}
Then we proceed similarly with the other variables and get
\begin{align*}
\E[f(t,x_1+W^1_t,\ldots x_k+W^k_t)] \le c\|f_t\|_{L^q_t(L^{p_1}_{x_1}(\ldots (L^{p_k}_{x_k})\ldots))} t^{-d/2p_1} \cdot\ldots\cdot t^{-d/2p_k}.
\end{align*}
We then conclude on \eqref{eq_Lp_bounds_k} as in the previous proof, taking
\begin{align}
C=C_T=c\left(\int_0^T t^{-dq'(1/2p_1+\ldots 1/2p_k)}dt\right)^{1/q'} = c_1 T^{1-1/q-d/(2p_1)-\ldots-d/(2p_k)},\label{eq_Lp_bounds_const}
\end{align}
for some constant $c_1>0$ independent of $T$. The bound for a general permutation $\sigma$ follows from \eqref{eq_Lp_bounds_k} applied to $f(x_{\sigma^{-1}(1)},\ldots x_{\sigma^{-1}(k)})$.
\end{proof}

Finally, we put together the previous bounds to obtain an exponential estimate for $L^q(L^p)$ functions (see \cite[Corollary 14]{FedFla2011} for a similar statement).

\begin{lm}\label{lm_kha_pq}
Let $W^1,\ldots W^k$ be $k$ independent $d$-dimensional Brownian motions starting from $0$. Take $1\le p_1,\ldots p_k,q \le \infty$ satisfying
\begin{align}
    \frac{d}{p_1}+\ldots \frac{d}{p_k}+\frac{2}{q}<2.\label{eq_KR_k}
\end{align}
Then there exists a constant $c>0$ (depending on $p_1,\ldots p_k,q,T$) such that, for every $f:[0,T]\times \R^{kd} \rightarrow \R$ non-negative Borel function with $f\in L^q_t(L^{p_1}_{x_1}\ldots (L^{p_k}_{x_k})\ldots)$,
\begin{align*}
    \sup_{x_1,\ldots x_k\in \R^d} \E \left[\exp\left[ \int_0^T f(t,x_1+W^1_t,\ldots x_k+W^k_t)\, dt \right]\right] \le \exp\left[c \left(1+\|f\|_{L^q_t(L^{p_1}_{x_1}\ldots (L^{p_k}_{x_k})\ldots)}^{1/(1-\alpha)}\right) \right],
\end{align*}
with $\alpha=1-1/q-d/(2p_1)-\ldots -d/(2p_k)$.
\end{lm}

\begin{proof}
We take
\begin{align*}
    h = \Bigl(2c_1 \|f\|_{L^q_t(L^{p_1}_{x_1}\ldots (L^{p_k}_{x_k})\ldots)}\Bigr)^{-1/(1-\alpha)} \wedge T,
\end{align*}
and let $t_j=hj\wedge T$, and $m$ the first positive integer with $t_m=T$, in particular
\begin{align*}
    m= \left\lceil\frac{T}{h}\right\rceil \le T\Bigl(2c_1 \|f\|_{L^q_t(L^{p_1}_{x_1}\ldots (L^{p_k}_{x_k})\ldots)}\Bigr)^{1/(1-\alpha)} +1
\end{align*}
With this choice of $h$, we have
\begin{align*}
    C_h\sup_{j=0,\ldots m-1} \left(\int_{t_j}^{t_{j+1}} \|f_t\|_{L^{p_1}_{x_1}(\ldots (L^{p_k}_{x_k})\ldots)}^q dt\right)^{1/q} \le C_h \|f\|_{L^q_t(L^{p_1}_{x_1}\ldots (L^{p_k}_{x_k})\ldots)} \le \frac12,
\end{align*}
$C_h$ being the constant in \eqref{eq_Lp_bounds_const}. As a consequence of Lemma \ref{lm_estlp12q}, we have
\begin{align*}
    &\sup_{j=0,\ldots,m-1}\sup_{x_1,\ldots x_k\in \R^d} \int_{t_j}^{t_{j+1}} \E[f(t,x_1+W^1_t-W^1_{t_j},\ldots x_k+W^k_t-W^k_{t_j})]\, dt\\
    &\hspace{4em}\le C_h\sup_{j=0,\ldots, m-1} \left(\int_{t_j}^{t_{j+1}} \|f_t\|_{L^{p_1}_{x_1}(\ldots (L^{p_k}_{x_k})\ldots)}^q dt\right)^{1/q} \le \frac12.
\end{align*}
Hence, we can apply Lemma \ref{lm_khashapp} and get
\begin{align}
    \sup_{j=0,\ldots,m-1}\sup_{x_1,\ldots x_k\in \R^d}\E \left[\exp\left[ \int_{t_j}^{t_{j+1}} f(t,x_1+W^1_t-W^1_{t_j},\ldots x_k+W^k_t-W^k_{t_j})\, dt \right]\right] \le 2.\label{eq_small_int}
\end{align}
Now we come back to the bound on the whole time interval $[0,T]$. We split the time integral over $[0,T]$ into the integrals over $[t_j,t_{j+1}]$ and use conditional expectation with respect to $\cF_{t_{m-1}}$: we have, for every $x_1,\ldots, x_k\in \R^d$,
\begin{align*}
    &\E \left[\exp\left[ \int_0^T f(t,x_1+W^1_t,\ldots x_k+W^k_t) \,dt \right]\right]\\
    &\hspace{2em}= \E \left[ \prod_{j=0}^{m-1}\exp\left[ \int_{t_j}^{t_{j+1}} f(t,x_1+W^1_t,\ldots x_k+W^k_t)\, dt \right]\right]\\
    &\hspace{2em}= \E \left[ \prod_{j=0}^{m-2}\exp\left[ \int_{t_j}^{t_{j+1}} f(t,x_1+W^1_t,\ldots x_k+W^k_t)\, dt \right] \cdot \right.\\
    &\hspace{10em} \left. \cdot\, \E\left[ \exp\left[ \int_{t_{m-1}}^{t_m} f(t,x_1+W^1_t,\ldots x_k+W^k_t)\, dt \right] \bigg| \cF_{t_{m-1}}\right] \phantom{\prod_{j=0}^{m-2}}\hspace{-2em}\right]
\end{align*}
(all exponentials are $\ge1$ and so the above products make sense and we can use the rule $\E[XY] = \E[X\E[Y\mid \cF_s]]$ for $X$ $\cF_s$-measurable). Now we apply the Markov property and the bound \eqref{eq_small_int} and get
\begin{align*}
    &\E \left[\exp\left[ \int_0^T f(t,x_1+W^1_t,\ldots x_k+W^k_t) \,dt \right]\right]\\
    &= \E \left[ \prod_{j=0}^{m-2}\exp\left[ \int_{t_j}^{t_{j+1}} f(t,x_1+W^1_t,\ldots x_k+W^k_t)\, dt \right] \cdot \right.\\
    &\quad \left. \cdot \E\left[ \exp\left[ \int_{t_{m-1}}^{t_m} f(t,y_1+W^1_t-W^1_{t_{m-1}},\ldots y_k+W^k_t-W^k_{t_{m-1}})\, dt \right]\right] \bigg| _{y_1=x_1+W^1_{t_{m-1}},\ldots, y_k=x_k+W^k_{t_{m-1}}} \right]\\
    &\le 2 \E \left[ \prod_{j=0}^{m-2}\exp\left[ \int_{t_j}^{t_{j+1}} f(t,x_1+W^1_t,\ldots x_k+W^k_t) \,dt \right] \right].
\end{align*}
Iterating this argument on $j$, we get finally, for every $x_1,\ldots x_k \in \R^d$,
\begin{align*}
    \E \left[\exp\left[ \int_0^T f(t,x_1+W^1_t,\ldots x_k+W^k_t) dt \right]\right] \le 2^m \le 2^{T(2c_1 \|f\|_{L^q_t(L^{p_1}_{x_1}\ldots (L^{p_k}_{x_k})\ldots)})^{1/(1-\alpha)} +1},
\end{align*}
which concludes the proof.
\end{proof}

\bibliography{library}
\bibliographystyle{myalpha}

\end{document}